\documentclass[11pt]{amsart}

\usepackage{
    amsmath,
    amsfonts,
    amssymb,
    amsthm,
    amscd,
    booktabs,
    comment,
    enumitem,
    etoolbox,
    gensymb,    
    mathtools,
    mathdots
}
\usepackage[usenames,dvipsnames]{xcolor}
\usepackage[all]{xy}

\makeatletter
\@namedef{subjclassname@2020}{%
  \textup{2020} Mathematics Subject Classification}
\makeatother

\newtoggle{comments}
\newtoggle{details}
\newtoggle{detailsnote}

\iftoggle{comments}{%
  \newcommand{\acomments}[1]{
    \ \\
    {\color{red}
      \textbf{AS:} #1
    }
    \ \\
    }
  \newcommand{\question}[1]{
    \ \\
    {\color{blue}
      \textbf{Question:} #1
    }
    \ \\
    }
}{%
  \newcommand{\acomments}[1]{}
  \newcommand{\question}[1]{}
}

\iftoggle{details}{%
  \newcommand{\details}[1]{
      \ \\
      {\color{OliveGreen}
        \textbf{Details:} #1
      }
      \\
  }
}{%
  \newcommand{\details}[1]{}
}

\usepackage[T1]{fontenc}
\usepackage{bbm}                     
\usepackage[colorlinks=true, linkcolor=blue, citecolor=blue, urlcolor=blue, breaklinks=true]{hyperref}
\usepackage{dsfont}

\DeclareFontFamily{OT1}{pzc}{}
\DeclareFontShape{OT1}{pzc}{m}{it}{<-> s * [1.10] pzcmi7t}{}
\DeclareMathAlphabet{\mathpzc}{OT1}{pzc}{m}{it}

\usepackage[capitalize]{cleveref}   

\crefname{defin}{Definition}{Definitions}
\crefname{eg}{Example}{Examples}
\crefname{lem}{Lemma}{Lemmas}
\crefname{theo}{Theorem}{Theorems}
\crefname{equation}{}{}
\crefname{enumi}{}{}

\newcommand\go{\mathsf{I}}              

\usepackage{tikz}
\usetikzlibrary{arrows,patterns}
\usepackage{tikz-cd}

\tikzset{pinhead/.style={gray,fill=yellow!40!white}}
\newcommand{\pin}[3]{
    \path (#1) node[inner sep=1.6pt] (x) {} to (#2) node[rectangle,rounded corners,draw,pinhead,inner sep=2.5pt](y) {$\color{black}\scriptstyle#3$};
    \draw[Triangle Cap-,thick,gray] (x)--(y);
    \fill[black] (#1) circle (2pt);
}

\newcommand{\pinpin}[4]{
\path (#1) node[inner sep=1.6pt] (x) {} to (#3) node[rectangle,rounded corners,draw,pinhead,inner sep=2.5pt](y) {$\color{black}\scriptstyle#4$};
    \draw[Triangle Cap-,thick,gray] (x)--(#2)--(y);
    \fill[black] (#1) circle (2pt);
    \fill[black] (#2) circle (2pt);
}

\newcommand\singdot[2][red]{
    \filldraw[fill=white, draw=#1] (#2) circle (1.8pt)
}
\newcommand\multdot[4][red]{
    \filldraw[fill=white, draw=#1] (#2) circle (1.8pt) node[anchor=#3] {{\color{#1} \dotlabel{#4}}}
}

\newcommand{\dotlabel}[1]{$\scriptstyle{#1}$}

\newcommand{\token}[3]{
    \filldraw[blue] (#1) circle (1.5pt) node[anchor=#2] {\dotlabel{#3}}
}

\newcommand{\xtoken}[1]{
    \filldraw[green] (#1) circle (1.5pt)
}
\newcommand\teleport[2]{
    \draw[blue] (#1) -- (#2);
    \filldraw[blue] (#1) circle (1.5pt);
    \filldraw[blue] (#2) circle (1.5pt);
}

\tikzset{wei/.style={draw=red,double=red!40!white,double distance=1.5pt,thin}}
\tikzset{anchorbase/.style={>=To,baseline={([yshift=-0.5ex]current bounding box.center)}}}

\newcommand{\xtokstrand}[1][a]{
    \begin{tikzpicture}[centerzero]
        \draw (0,-0.2) -- (0,0.2);
        \xtoken{0,0};
    \end{tikzpicture}
}

\tikzset{anchorbase/.style={>=To,baseline={([yshift=-0.5ex]current bounding box.center)}}}
\tikzset{ 
    centerzero/.style={>=To,baseline={([yshift=-0.5ex](#1))}},
    centerzero/.default={0,0}
}
\tikzset{wipe/.style={white,line width=4pt}}

\newtheorem{theo}{Theorem}[section]

\newtheorem{prop}[theo]{Proposition}
\newtheorem{lem}[theo]{Lemma}
\newtheorem{cor}[theo]{Corollary}

\theoremstyle{definition}
\newtheorem{defn}[theo]{Definition}
\newtheorem{rmk}[theo]{Remark}
\newtheorem{egg}[theo]{Example}

\numberwithin{equation}{section}
\allowdisplaybreaks

\setenumerate[1]{label=(\alph*)}          

\begin{document}

\title{Higher-level affine wreath product algebras}

\author{Thomas Moran}
\address[T.M.]{
  Department of Mathematics and Statistics \\
  University of Ottawa \\
  Ottawa, ON K1N 6N5, Canada
}
\urladdr{\href{https://sites.google.com/view/tmoran/}{https://sites.google.com/view/tmoran/}}
\email{tmora083@uottawa.ca}

\begin{abstract}
We define and study two new classes of algebras, called \emph{higher-level affine wreath product algebras} and \emph{higher-level affine Frobenius Hecke algebras}. They depend on a Frobenius superalgebra and are defined, respectively, as path algebras of the \emph{higher-level affine wreath product category} and \emph{higher-level affine Frobenius Hecke category}. Our constructions produce a broad range of new higher-level algebras under a unified framework. Special cases include higher-level analogues of the degenerate affine Hecke algebra and affine Sergeev algebras, both of which appear to be new.
\end{abstract}



\ifboolexpr{togl{comments} or togl{details}}{%
  {\color{magenta}DETAILS OR COMMENTS ON}
}{%
}

\maketitle
\thispagestyle{empty}

\tableofcontents

\section{Introduction}

\subsection{Quantum wreath products}
Affine Hecke algebras and their degenerate counterparts play an important role in the representation theory of Lie algebras. In recent years, several modified versions of these algebras have appeared in the literature. For example, in \cite{Savage}, Savage defined the affine wreath product algebras, which are generalizations of degenerate affine Hecke algebras. They depend on a Frobenius superalgebra, and have been used in the categorification of lattice Heisenberg algebras in \cite{Savage2}. Affine wreath product algebras provide a simultaneous unification and generalization of existing constructions in the literature that were previously treated separately. For example, special cases include the degenerate affine Hecke algebras, affine Sergeev algebras \cite{Nazarov}, wreath Hecke algebras \cite{Wan-Wang}, and affine zigzag algebras \cite{Kleshchev-Muth}.
\\ \indent A quantum analogue of the affine wreath product algebras was defined by Rosso and Savage in \cite{Rosso-Savage}. These are called affine Frobenius Hecke algebras (or quantum affine wreath algebras), and depend on a symmetric Frobenius superalgebra. Special cases of the affine Frobenius Hecke algebras include affine Yokonuma--Hecke algebras \cite{Chlouveraki-d'Andecy} and affine Hecke algebras. Both affine wreath product algebras and affine Frobenius Hecke algebras are special cases of quantum wreath products, in the sense of \cite{Lai-Nakano-Ziqing}.

\subsection{Higher-level Hecke algebras} 
In recent years, several new classes of higher-level Hecke-type algebras have appeared in the literature. Loosely speaking, these algebras are obtained by introducing red strands, which interact with the black strands corresponding to the original Hecke algebra. They are closely related to the cyclotomic quotients of Hecke algebras, and have been used to categorify tensor products of simple modules over a quantum group (see \cite{Webster}). Examples of such higher-level Hecke algebras include the tensor product algebras from \cite[Def.~4.7]{Webster}, and the higher-level affine Hecke algebras (or the type \(F\) affine Hecke algebras) from \cite[Def.~2.7]{Maksimau-Stroppel} and \cite[Def.~5.5]{Webster2}. Another recent example is the affine Schur category from \cite{Song-Wang2}.
\\ \indent It is natural to ask whether one can combine the above-mentioned approaches by defining higher-level versions of the affine wreath product algebras and affine Frobenius Hecke algebras. In this paper, we answer this question in the affirmative by constructing such algebras. We expect that these constructions will lead to new developments in categorification and representation theory.

\subsection{Higher-level affine wreath product algebras}

Fix a commutative ground ring $ \Bbbk $. To every Frobenius superalgebra \(A\), we associate the \emph{higher-level affine wreath product category} $ \mathcal{LAW}(A) $ and the \emph{higher-level affine wreath product algebra} $ W_{d,\mathbf{Q}}^{\operatorname{aff}}(A) $, which is a path algebra of $ \mathcal{LAW}(A) $. The algebra $ W_{d,\mathbf{Q}}^{\operatorname{aff}}(A) $ can simultaneously be viewed as a higher-level version of the affine wreath product algebra from \cite{Savage}, and as a degenerate version of the higher-level affine Frobenius Hecke algebra discussed below. This construction yields a broad range of new higher-level algebras under a unified framework. For example, by specializing \(A\), we obtain the following:
\begin{enumerate}
\item[\textbullet] When $ A = \Bbbk $, the higher-level affine wreath product algebra is a higher-level version of the degenerate affine Hecke algebra. This algebra appears to be new, although it does implicitly appear as a path algebra of the affine Schur category from \cite{Song-Wang2}.
\item[\textbullet] If $ A = \operatorname{Cl} $ is the two-dimensional Clifford superalgebra (see Example \ref{Cliff_parity_pres}), then the higher-level affine wreath product algebra yields a higher-level version of the affine Sergeev algebra. (See \cite[§3]{Nazarov} for the definition of the affine Sergeev algebra.) This construction appears to be new.
\item[\textbullet] When $ A = \Bbbk G $ is the group algebra of a finite group \(G\) (see Example \ref{Group_algebra}), then the higher-level affine wreath product algebra is a higher-level version of the wreath Hecke algebra. (See \cite[Def.~2.4]{Wan-Wang} for the definition of the wreath Hecke algebra.) This construction appears to be new.
\item[\textbullet] When \(A\) is a certain skew-zigzag algebra (see \cite[§5]{Couture} and \cite[§3]{Huerfano-Khovanov}), the higher-level affine wreath product algebra yields a higher-level version of the affine zigzag algebra. (See \cite[Def.~4.4]{Kleshchev-Muth} for the definition of the affine zigzag algebra.) This construction appears to be new.
\end{enumerate}

In Theorem \ref{monoidal_functor}, we state and prove the existence of a strict monoidal functor $ \Phi \colon \mathcal{LAW}(A) \rightarrow \mathcal{AW}(A) $, where $ \mathcal{AW}(A) $ is the affine wreath product category (see Definition \ref{pluto2}). We use this functor to determine a $ \Bbbk $-basis of the morphism spaces $ \mathcal{LAW}(A) $, which is given in Theorem \ref{HL-basis}. We also use the functor $ \Phi $ to obtain a polynomial representation of $ \mathcal{LAW}(A) $ (see Corollary \ref{poly}). In Section \ref{Cyclotomic_quotients}, we define cyclotomic quotients of the algebras $ W_{d,\mathbf{Q}}^{\operatorname{aff}}(A) $. We relate these to cyclotomic wreath product algebras in Theorem \ref{cyclotomic_aff_thm}. 

\subsection{Higher-level affine Frobenius Hecke algebras}

In this paper, we also define higher-level versions of the affine Frobenius Hecke algebras. To every symmetric Frobenius superalgebra \(A\) and scalar $ z \in \Bbbk $, we associate the \emph{higher-level affine Frobenius Hecke category} $ \mathcal{LAH}(A,z) $ and the \emph{higher-level affine Frobenius Hecke algebra} $ H_{d,\mathbf{Q}}^{\operatorname{aff}}(A,z) $, which is a path algebra of $ \mathcal{LAH}(A,z) $. The algebra $ H_{d,\mathbf{Q}}^{\operatorname{aff}}(A,z)$ can simultaneously be viewed as a higher-level version of the affine Frobenius Hecke algebra, as a quantum version of the higher-level affine wreath product algebra discussed above, or as a Frobenius algebra generalization of the higher-level affine Hecke algebra from \cite[Def.~2.7]{Maksimau-Stroppel}. This construction yields a broad range of new higher-level algebras under a unified framework. For example, by specializing \(A\), we obtain the following:
\begin{enumerate}
\item[\textbullet] When $ A = \Bbbk $, the higher-level affine Frobenius Hecke algebra recovers the higher-level affine Hecke algebra from \cite[Def.~2.7]{Maksimau-Stroppel}. See Remark \ref{usualaff} for more details.
\item[\textbullet] When $ A = \Bbbk G $ is the group algebra of a finite group \(G\) (see Example \ref{Group_algebra}), we obtain a higher-level version of the affine Yokonuma--Hecke algebra. (See \cite[Def.~3.1]{Chlouveraki-d'Andecy} for the definition of the affine Yokonuma--Hecke algebra.) This construction appears to be new.
\end{enumerate}

In Theorem~\ref{HL-basis_quantum}, we give a $ \Bbbk $-basis of the morphism spaces $ \mathcal{LAH}(A,z) $. In Section \ref{Cyclotomic_quotients_quantum}, we define cyclotomic quotients of the algebras $ H_{d,\mathbf{Q}}^{\operatorname{aff}}(A,z) $. We relate these to cyclotomic Frobenius Hecke algebras in Theorem \ref{quantum_cyclotomic_thm}.

\subsection{Directions of future research}

Taking \(A\) to be the rank-two Clifford superalgebra, our constructions yield a higher-level version of the affine Sergeev algebra. In a future paper, we will define a higher-level version of the degenerate spin affine Hecke algebra, and we will show that it is Morita superequivalent to the higher-level affine Sergeev algebra. This result will be a generalization of \cite[Thm.~4.1]{Wang}. We now list some other potential directions of future research: 

\begin{enumerate}
\item[\textbullet] Using the techniques of this paper, it should be possible to construct a higher-level version of the quantum wreath products from \cite{Lai-Nakano-Ziqing}. Then the higher-level affine wreath product algebra and higher-level affine Frobenius Hecke algebra should follow as special cases of such an algebra.
\item[\textbullet] In \cite{Davidson-Kujawa-Muth}, Davidson, Kujawa and Muth defined the affine Frobenius web category $ \mathbf{Web}^{\operatorname{aff}}_{A} $, which depends on a Frobenius superalgebra \(A\). It should be possible to define a higher-level version of $ \mathbf{Web}^{\operatorname{aff}}_{A} $. The resulting category will then contain both $ \mathbf{Web}^{\operatorname{aff}}_{A} $ and $ \mathcal{LAW}(A) $ as monoidal subcategories. Furthermore, when $ A = \Bbbk $, the resulting category should be related to the affine Schur category from \cite{Song-Wang2}.
\end{enumerate}

\subsection*{Note on the arXiv version} \label{arXiv}

For the interested reader, the tex file of the arXiv version of this paper includes hidden details of some straightforward computations and arguments that are omitted in the pdf file. These details can be displayed by switching the \texttt{details} toggle to true in the tex file and recompiling.

\subsection*{Acknowledgements} I would like to thank Alistair Savage for his guidance, and for his very helpful comments and suggestions whilst working on this project.

\section{Preliminaries} \label{Preliminaries_paper}

In this section, we recall some basic facts about superalgebras, monoidal supercategories, and monoids.

\subsection{Superalgebras and supermodules}

As mentioned in the previous section, $ \Bbbk $ is a fixed commutative ground ring. We define $ \mathbb{N} $ to be the set of non-negative integers, and we define $ \mathbb{N}_{+} $ to be the set of positive integers. All algebras and modules in this paper are associative superalgebras and supermodules. For simplicity, we will often drop the prefix ``super''. For a homogeneous element $ f $ in an algebra or module, we let $\overline{f} \in \mathbb{Z}_{2} $ denote its parity. When we write an equation involving parities of elements, we implicitly assume these elements are homogeneous; we then extend by linearity. We define the \emph{center} of an algebra \(F\) to be
\begin{equation} \label{centerdef}
Z(F) := \{f \in F : fg = (-1)^{\bar{f}\bar{g}}gf \ \text{for all} \ g \in F\}.
\end{equation}
We define the \emph{even center} of \(F\) to be the set of those elements in $ Z(F) $ that are even. We say that two elements $ f,g \in F $ \emph{commute} if $ fg = (-1)^{\bar{f}\bar{g}}gf $. Furthermore, we say that an element $ f \in F $ is \emph{central} if $ f \in Z(F) $, and we say that \(F\) is \emph{commutative} if $ Z(F) = F $. 
\\ \indent We say that $ f \in F $ is a \emph{left zero-divisor} if there exists a nonzero $ g \in F $ such that $ fg = 0 $. Similarly, we say that $ f \in F $ is a \emph{right zero-divisor} if there exists a nonzero $ g \in F $ such that $ gf = 0 $.

\begin{defn} \label{uncharted_general}
We say that an element $ f \in F $ is \emph{regular} if \(f\) is neither a left nor right zero-divisor in \(F\). So \(f\) is regular if and only if the maps $ F \rightarrow F $, $ g \mapsto fg $, and $ F \rightarrow F $, $ g \mapsto gf $, are injective.
\end{defn}

\begin{defn} \label{multi_pen}
We define $ \mathcal{E}_{F} $ to be the set of elements in $ F $ that are even, central, and regular. We also define the set $ \widehat{\mathcal{E}}_{F} := \mathcal{E}_{F} \cup \{\go\} $, where $ \go $ is a formal symbol.
\end{defn}

For algebras $ F $ and $ F' $, multiplication in the superalgebra $ F \otimes F' $ is defined by
\begin{equation}
    (f' \otimes g) (f \otimes g') = (-1)^{\bar f \bar g} f'f \otimes gg'
\end{equation}
for homogeneous $f,f' \in F$, $g,g' \in F'$.

\subsection{Monoidal supercategories}

Throughout this paper, we will work with \emph{strict monoidal supercategories}, in the sense of \cite{Brundan-Ellis}. We refer the reader to \cite[§2]{Brundan-Savage-Webster2} for a summary of this topic well adapted to the current work, or to \cite{Brundan-Ellis} for a thorough treatment. We summarize here a few crucial properties that play an important role in the present paper.
\\ \indent A \emph{supercategory} means a category enriched in the category of vector superspaces with parity-preserving morphisms.  Thus, its morphism spaces are vector superspaces and composition is parity-preserving.  In a \emph{strict monoidal supercategory}, morphisms satisfy the \emph{super interchange law}:
\begin{equation}\label{interchange}
    (f' \otimes g) \circ (f \otimes g')
    = (-1)^{\bar f \bar g} (f' \circ f) \otimes (g \circ g').
\end{equation}
We denote the unit object by $ \mathds{1} $ and the identity morphism of an object $X$ by $1_X$.  We will use the usual calculus of string diagrams, representing the horizontal composition $f \otimes g$ (resp.\ vertical composition $f \circ g$) of morphisms $f$ and $g$ diagrammatically by drawing $f$ to the left of $g$ (resp.\ drawing $f$ above $g$).  Care is needed with horizontal levels in such diagrams due to the signs arising from the super interchange law:
\begin{equation}\label{intlaw}
    \begin{tikzpicture}[anchorbase]
        \draw (-0.5,-0.5) -- (-0.5,0.5);
        \draw (0.5,-0.5) -- (0.5,0.5);
        \filldraw[fill=white,draw=black] (-0.5,0.15) circle (5pt);
        \filldraw[fill=white,draw=black] (0.5,-0.15) circle (5pt);
        \node at (-0.5,0.15) {$\scriptstyle{f}$};
        \node at (0.5,-0.15) {$\scriptstyle{g}$};
    \end{tikzpicture}
    \quad=\quad
    \begin{tikzpicture}[anchorbase]
        \draw (-0.5,-0.5) -- (-0.5,0.5);
        \draw (0.5,-0.5) -- (0.5,0.5);
        \filldraw[fill=white,draw=black] (-0.5,0) circle (5pt);
        \filldraw[fill=white,draw=black] (0.5,0) circle (5pt);
        \node at (-0.5,0) {$\scriptstyle{f}$};
        \node at (0.5,0) {$\scriptstyle{g}$};
    \end{tikzpicture}
    \quad=\quad
    (-1)^{\bar f\bar g}\
    \begin{tikzpicture}[anchorbase]
        \draw (-0.5,-0.5) -- (-0.5,0.5);
        \draw (0.5,-0.5) -- (0.5,0.5);
        \filldraw[fill=white,draw=black] (-0.5,-0.15) circle (5pt);
        \filldraw[fill=white,draw=black] (0.5,0.15) circle (5pt);
        \node at (-0.5,-0.15) {$\scriptstyle{f}$};
        \node at (0.5,0.15) {$\scriptstyle{g}$};
    \end{tikzpicture}
    \ .
\end{equation}
A strict monoidal supercategory $ \mathcal{C} $ with one generating object $\go$ gives rise to a \emph{tower of algebras}
\begin{equation*}
\operatorname{End}_{\mathcal{C}}(\go^{\otimes n}), \quad n \in \mathbb{N}.
\end{equation*}
We use this idea to introduce various families of algebras in an extremely efficient way, giving a presentation of $ \mathcal{C} $ with a small number of generating morphisms and relations.  If we then wish to have a presentation of the endomorphism algebras of $ \mathcal{C} $ \emph{as superalgebras}, we use the following result.

\begin{prop} \label{zebra}
Suppose $ \mathcal{C} $ is a strict $ \Bbbk $-linear monoidal supercategory with one generating object $\go$ and generating morphisms $f_i \in \operatorname{End}_{\mathcal{C}}(\go^{\otimes n_{i}}) $, $i \in I$, subject to the relations $R_j \in \operatorname{End}_{\mathcal{C}}(\go^{\otimes m_{j}})$, $j \in J$.  Then $\operatorname{End}_{\mathcal{C}}(\go^{\otimes n})$ is generated as an algebra by the elements
\begin{equation}
 1_{\go}^{\otimes k} \otimes f_{i} \otimes 1_{\go}^{\otimes (n-n_{i}-k)},\quad i \in I,\ 0 \leq k \leq n-n_{i},
\end{equation}
    subject to the relations
    \begin{equation} \label{zebra1}
        1_{\go}^{\otimes k} \otimes R_j \otimes 1_{\go}^{\otimes (n-m_{j}-k)},\quad j \in J,\ 0 \leq k \leq n-m_j,
    \end{equation}
    and the relations
    \begin{equation} \label{zebra2}
        \left( 1_{\go}^{\otimes k_{1}} \otimes f_i \otimes 1_{\go}^{\otimes k_2} \right)
        \left( 1_{\go}^{\otimes l_{1}} \otimes f_j \otimes 1_{\go}^{\otimes l_2} \right)
        - (-1)^{\overline{f_i} \ \overline{f_j}}
        \left( 1_{\go}^{\otimes l_{1}} \otimes f_j \otimes 1_{\go}^{\otimes l_2} \right)
        \left( 1_{\go}^{\otimes k_{1}} \otimes f_i \otimes 1_{\go}^{\otimes k_2} \right)
    \end{equation}
    for $i,j \in I$, $k_1+n_i+k_2 = n = l_1 + n_j + l_2$, $k_2 \geq n_j + l_2$.
\end{prop}

\begin{proof}
This follows from a \emph{super} version of \cite[Thm.~5.2, Thm.~5.4]{Liu}. Note that the relations \cref{zebra2} correspond to the super interchange law \cref{intlaw}.
\end{proof}

The \emph{reversed supercategory} $ \mathcal{C}^{\operatorname{rev}} $ of $ \mathcal{C} $ is the strict monoidal supercategory with the same objects and morphisms as $ \mathcal{C} $, where the composition of morphisms is the same as in $ \mathcal{C} $, but the tensor product is given by $ X \otimes^{\operatorname{rev}} Y = Y \otimes X $, $ f \otimes^{\operatorname{rev}} g = (-1)^{\bar{f}\bar{g}}(g \otimes f) $, for objects $ X, Y $ and morphisms $ f,g $.

\subsection{Monoids}

For a given set \(X\), we define $ F(X) $ to be the free monoid on \(X\). That is, $ F(X) $ is the set of all words in \(X\), and the product of any two words is their concatenation. We denote the length of a word $ w \in F(X) $ by $ |w| $. A \emph{subword} of $ w $ is an element in $ F(X) $ obtained by deleting some (not necessarily consecutive) letters in \(w\).

\begin{defn} \label{snowing}
Let $ u,v \in F(X) $. Then a $ (u,v) $-\emph{shuffle} is a word $ w \in F(X) $ satisfying the following two properties:
\begin{enumerate}
\item[\textbullet] The length of $ w $ is $ |w| = |u| + |v| $.
\item[\textbullet] The element $ w $ contains both $ u $ and $ v $ as subwords.
\end{enumerate}
We define $ F(X)_{u,v} $ to be the set of all $ (u,v) $-shuffles.
\end{defn}

\begin{egg}
Let $ X = \{a,b,c\} $. Then the word $ abcb $ is an $ (ac,bb) $-shuffle. The words $ cbab $ and $ acbbc $ are not $ (ac,bb) $-shuffles.
\end{egg}

\section{Frobenius superalgebras and teleporters} \label{Frob_algebras}

In this section, we recall some facts about Frobenius superalgebras. We will also define the teleporter morphisms in this section.

From here onwards, we fix a Frobenius superalgebra $ A \label{magnesium} $ with homogeneous trace map $ \operatorname{tr} \colon A \rightarrow \Bbbk \label{hydrogen} $ of parity $ \varepsilon \label{traceparity} \in \mathbb{Z}_{2} $. Let $ \psi \colon A \rightarrow A \label{Nakayama} $ be the Nakayama automorphism of \(A\), which is uniquely determined by the condition
\begin{equation} \label{supersymmetric}
\operatorname{tr}(ab) = (-1)^{\overline{a}\overline{b}}\operatorname{tr}(b\psi(a)), \quad a,b \in A.
\end{equation}
We say that \(A\) is \emph{symmetric} if $ \psi = \operatorname{id} $. We assume that \(A\) is a free $ \Bbbk $-module, and we fix a homogeneous basis $ B \label{teller} $ of \(A\). The definition of a Frobenius superalgebra implies that \(A\) has a homogeneous basis $ B^{\vee} := \{b^{\vee} : b \in B\} \label{zellers} $, called the \emph{left dual basis} of $ B $, satisfying
\begin{equation*}
\operatorname{tr}(b^{\vee}c) = \delta_{b,c}, \quad b,c \in B.
\end{equation*}
It follows that
\begin{equation} \label{expansion}
\sum_{b \in B} \operatorname{tr}(b^{\vee}a)b = a = \sum_{b \in B} \operatorname{tr}(ab)b^{\vee}, \quad a \in A.
\end{equation}
Note that 
\begin{equation} \label{dualparity}
\overline{b} + \overline{b^{\vee}} = \varepsilon, \quad b \in B.
\end{equation}
The left dual basis $ (B^{\vee})^{\vee} $ of $ B^{\vee} $ is given by
\begin{equation} \label{dualdual}
(b^{\vee})^{\vee} = (-1)^{\overline{b} + \varepsilon\overline{b}}\psi^{-1}(b), \quad b \in B.
\end{equation}
We now give some examples of Frobenius algebras.

\begin{egg}[Ground ring]
The ground ring $ \Bbbk $ is trivially a symmetric Frobenius algebra, with trace map $ \operatorname{tr} \colon \Bbbk \rightarrow \Bbbk $ given by the identity map.
\end{egg}

\begin{egg}[Group algebra of a finite group] \label{Group_algebra}
Let $ \Bbbk G $ be the group algebra of a finite group \(G\). Then $ \Bbbk G $ is a symmetric Frobenius algebra with trace map $ \operatorname{tr} \colon \Bbbk G \rightarrow \Bbbk $ given by $ \operatorname{tr}(g) = \delta_{e_{G},g} $, $ g \in G $. Here, $ e_{G} $ is the identity element of \(G\).
\end{egg}

\begin{egg}[Clifford superalgebra with a parity-preserving trace map] \label{Cliff_parity_pres}
Let $ \operatorname{Cl} $ be the rank-two Clifford superalgebra, which is generated by a single odd element $ c \in \operatorname{Cl} $ satisfying $ c^{2} = 1 $, and has basis $ \{1,c\} $. Then $ \operatorname{Cl} $ is a Frobenius algebra with trace map $ \operatorname{tr} \colon \operatorname{Cl} \rightarrow \Bbbk $ given by $ 1 \mapsto 1 $, $ c \mapsto 0 $. Here, we have $ \varepsilon = \bar{0} $, and the Nakayama automorphism $ \psi \colon \operatorname{Cl} \rightarrow \operatorname{Cl} $ is given by $ c \mapsto -c $. 
\end{egg}

\begin{egg}[Clifford superalgebra with a parity-reversing trace map]
\label{Cliff_parity_rev}
As in the previous example, let $ \operatorname{Cl} $ be the rank-two Clifford superalgebra. Then $ \operatorname{Cl} $ is a symmetric Frobenius algebra with trace map $ \operatorname{tr} \colon \operatorname{Cl} \rightarrow \Bbbk $ given by $ 1 \mapsto 0 $, $ c \mapsto 1 $. Here, we have $ \varepsilon = \bar{1} $.
\end{egg}

\begin{egg}[Grassmann superalgebra]
Let $ \Lambda $ be the rank-two Grassmann superalgebra, which is generated by a single odd element $ x $ satisfying $ x^{2} = 0 $, and has basis $ \{1,x\} $. Then $ \Lambda $ is a symmetric Frobenius algebra, with trace map $ \operatorname{tr} \colon \Lambda \rightarrow \Bbbk $ given by $ 1 \mapsto 0 $, $ x \mapsto 1 $. Here, we have $ \varepsilon = \bar{1} $.
\end{egg}

Let $ n \in \mathbb{N}_{+} $. Then the symmetric group $ S_{n} $ acts on $ A^{\otimes n} $ by superpermutation of the factors. If $ \mathbf{a} \in A^{\otimes n} $ and $ w \in S_{n} $, then we denote the resulting action by $ w(\mathbf{a}) $. Explicitly, the simple transposition $ s_{i} $ acts on $ A^{\otimes n} $ by
\begin{equation*}
s_{i}(a_{(1)} \otimes \cdots \otimes a_{(n)}) = (-1)^{\overline{a_{(i)}} \ \overline{a_{(i+1)}}} a_{(1)} \otimes \cdots \otimes a_{(i-1)} \otimes a_{(i+1)} \otimes a_{(i)} \otimes a_{(i+2)} \otimes \cdots \otimes a_{(n)}.
\end{equation*}
For $ a \in A $ and $ 1 \leq i \leq n $, we define
\begin{align}
a_{i} &:= 1^{\otimes (i-1)} \otimes a \otimes 1^{\otimes (n-i)} \in A^{\otimes n}, \label{helium}
\\ \psi_{i} &:= \operatorname{id}^{\otimes (i-1)} \otimes \psi \otimes \operatorname{id}^{\otimes (n-i)} \colon A^{\otimes n} \rightarrow A^{\otimes n}. \label{lithium}
\end{align}
We also define 
\begin{equation} \label{wicked}
t_{i,j} := \sum_{b \in B} (-1)^{\varepsilon\overline{b}}b_{i}b_{j}^{\vee} \in A^{\otimes n}, \quad 1 \leq i,j \leq n, \ i \neq j. 
\end{equation}
The elements $ t_{i,j} $ are independent of the choice of basis \(B\) and have parity $ \varepsilon $.

\begin{defn}
The \emph{Frobenius tower category} $ \mathpzc{tower}(A) $ is the strict $ \Bbbk $-linear monoidal supercategory generated by one object $ \go $, with generating morphisms
\begin{equation*}
\begin{tikzpicture}[centerzero]
      \draw[-] (0,-0.2) -- (0,0.2);
      \token{0,0}{west}{a};
\end{tikzpicture}
\colon \go \rightarrow \go,
 \quad a \in A,
\end{equation*}
called \emph{tokens}, and relations
\begin{equation} 
\label{tokrel1}
\begin{tikzpicture}[anchorbase, thick]
        \draw[-] (0,0) -- (0,0.7);
        \token{0,0.35}{west}{1};
\end{tikzpicture}
=
\begin{tikzpicture}[anchorbase, thick]
        \draw[-] (0,0) -- (0,0.7);
\end{tikzpicture}
\ ,\quad \lambda\
\begin{tikzpicture}[anchorbase, thick]
        \draw[-] (0,0) -- (0,0.7);
        \token{0,0.35}{west}{a};
\end{tikzpicture}
+ \mu\
\begin{tikzpicture}[anchorbase, thick]
        \draw[-] (0,0) -- (0,0.7);
        \token{0,0.35}{west}{b};
\end{tikzpicture}
=
\begin{tikzpicture}[anchorbase, thick]
        \draw[-] (0,0) -- (0,0.7);
        \token{0,0.35}{west}{\lambda a + \mu b};
\end{tikzpicture}
,\quad
\begin{tikzpicture}[anchorbase, thick]
        \draw[-] (0,0) -- (0,0.7);
        \token{0,0.2}{east}{b};
        \token{0,0.45}{east}{a};
\end{tikzpicture}
=
\begin{tikzpicture}[anchorbase, thick]
        \draw[-] (0,0) -- (0,0.7);
        \token{0,0.35}{west}{ab};
\end{tikzpicture}
\ ,\quad a,b \in A,\ \lambda,\mu \in \Bbbk.
\end{equation}
We declare the parity of a token $ \begin{tikzpicture}[centerzero]
      \draw[-] (0,-0.2) -- (0,0.2);
      \token{0,0}{west}{a};
\end{tikzpicture} $ to be equal to the parity of \(a\).
\end{defn}

The relation \cref{tokrel1} together with Proposition \ref{zebra} imply that we have an isomorphism of algebras 
\begin{equation} \label{tensorisomorp}
A^{\otimes n} \xrightarrow{\cong} \operatorname{End}_{\mathpzc{tower}(A)}(\go^{\otimes n}),
\end{equation}
sending $ a_{i} $, $ a \in A $, to the token labelled \(a\) on strand \(i\). Here and throughout this paper, we label strands from \emph{left to right}. We define the \emph{teleporters}
\begin{align}
    \begin{tikzpicture}[centerzero]
        \draw (-0.2,-0.3) -- (-0.2,0.3);
        \draw (0.2,-0.3) -- (0.2,0.3);
        \teleport{-0.2,0.1}{0.2,-0.1};
    \end{tikzpicture}
    := \sum_{b \in B} (-1)^{\varepsilon \bar{b}}
    \begin{tikzpicture}[centerzero]
        \draw (-0.2,-0.3) -- (-0.2,0.3);
        \draw (0.2,-0.3) -- (0.2,0.3);
        \token{-0.2,0.1}{east}{b};
        \token{0.2,-0.1}{west}{b^\vee};
    \end{tikzpicture} \ , \label{teleporter66} 
    \qquad
    \begin{tikzpicture}[centerzero]
        \draw (-0.2,-0.3) -- (-0.2,0.3);
        \draw (0.2,-0.3) -- (0.2,0.3);
        \teleport{-0.2,-0.1}{0.2,0.1};
    \end{tikzpicture}
    := \sum_{b \in B} (-1)^{\varepsilon \bar{b}}
    \begin{tikzpicture}[centerzero]
        \draw (-0.2,-0.3) -- (-0.2,0.3);
        \draw (0.2,-0.3) -- (0.2,0.3);
        \token{-0.2,-0.1}{east}{b^\vee};
        \token{0.2,0.1}{west}{b};
    \end{tikzpicture} 
    \ .
\end{align}
Again, the teleporters are independent of the choice of basis \(B\), and they have parity $ \varepsilon $. We have that the tokens ``teleport'' past the teleporters, in the sense of the following lemma.

\begin{lem} \label{teleportt}
Let $ a \in A $. Then
\begin{align}
    \begin{tikzpicture}[anchorbase]
        \draw[-] (0,-0.5) --(0,0.5);
        \draw[-] (0.5,-0.5) -- (0.5,0.5);
        \token{0,0.3}{east}{a};
        \teleport{0,0.1}{0.5,-0.1};
    \end{tikzpicture}
    &= (-1)^{\varepsilon \bar{a}}\
    \begin{tikzpicture}[anchorbase]
        \draw[-] (0,-0.5) --(0,0.5);
        \draw[-] (0.5,-0.5) -- (0.5,0.5);
        \token{0.5,-0.3}{west}{a};
        \teleport{0,0.1}{0.5,-0.1};
    \end{tikzpicture}
    \ ,&
    \begin{tikzpicture}[anchorbase]
        \draw[-] (0,-0.5) --(0,0.5);
        \draw[-] (0.5,-0.5) -- (0.5,0.5);
        \token{0,-0.3}{east}{\psi(a)};
        \teleport{0,0.1}{0.5,-0.1};
    \end{tikzpicture}
    &= (-1)^{\varepsilon \bar{a}}\
    \begin{tikzpicture}[anchorbase]
        \draw[-] (0,-0.5) --(0,0.5);
        \draw[-] (0.5,-0.5) -- (0.5,0.5);
        \token{0.5,0.3}{west}{a};
        \teleport{0,0.1}{0.5,-0.1};
    \end{tikzpicture}
    \ , \label{teleporter1}
    \\
    \begin{tikzpicture}[anchorbase]
        \draw[-] (0,-0.5) --(0,0.5);
        \draw[-] (0.5,-0.5) -- (0.5,0.5);
        \token{0,-0.3}{east}{a};
        \teleport{0,-0.1}{0.5,0.1};
    \end{tikzpicture}
    &= (-1)^{\varepsilon \bar{a}}\
    \begin{tikzpicture}[anchorbase]
        \draw[-] (0,-0.5) --(0,0.5);
        \draw[-] (0.5,-0.5) -- (0.5,0.5);
        \token{0.5,0.3}{west}{a};
        \teleport{0,-0.1}{0.5,0.1};
    \end{tikzpicture}
    \ ,&
    \begin{tikzpicture}[anchorbase]
        \draw[-] (0,-0.5) --(0,0.5);
        \draw[-] (0.5,-0.5) -- (0.5,0.5);
        \token{0,0.3}{east}{a};
        \teleport{0,-0.1}{0.5,0.1};
    \end{tikzpicture}
    &= (-1)^{\varepsilon \bar{a}}\
    \begin{tikzpicture}[anchorbase]
        \draw[-] (0,-0.5) --(0,0.5);
        \draw[-] (0.5,-0.5) -- (0.5,0.5);
        \token{0.5,-0.3}{west}{\psi(a)};
        \teleport{0,-0.1}{0.5,0.1};
    \end{tikzpicture}
    \ . \label{teleporter2}
\end{align}
\end{lem}

\begin{proof}
We compute that
\begin{equation*}
\begin{tikzpicture}[anchorbase]
        \draw[-] (0,-0.5) --(0,0.5);
        \draw[-] (0.5,-0.5) -- (0.5,0.5);
        \token{0,0.3}{east}{a};
        \teleport{0,0.1}{0.5,-0.1};
\end{tikzpicture}
\stackrel{\cref{expansion}}{=} \sum_{b,c \in B} (-1)^{\varepsilon \bar{b}}
\begin{tikzpicture}[centerzero]
        \draw (0,-0.5) -- (0,0.5);
        \draw (0.5,-0.5) -- (0.5,0.5);
        \token{0,0.1}{east}{\operatorname{tr}(c^{\vee}ab)c};
        \token{0.5,-0.1}{west}{b^\vee};
\end{tikzpicture}
= \sum_{b,c \in B} (-1)^{\varepsilon(\bar{a} + \bar{c})}
\begin{tikzpicture}[centerzero]
        \draw (0,-0.5) -- (0,0.5);
        \draw (0.5,-0.5) -- (0.5,0.5);
        \token{0,0.1}{east}{c};
        \token{0.5,-0.1}{west}{\operatorname{tr}(c^{\vee}ab)b^{\vee}};
\end{tikzpicture}
\stackrel{\cref{expansion}}{=} (-1)^{\varepsilon \bar{a}}\
\begin{tikzpicture}[anchorbase]
        \draw[-] (0,-0.5) --(0,0.5);
        \draw[-] (0.5,-0.5) -- (0.5,0.5);
        \token{0.5,-0.3}{west}{a};
        \teleport{0,0.1}{0.5,-0.1};
\end{tikzpicture} \ ,
\end{equation*}
where in the second equality, we used the fact that $ \operatorname{tr}(c^{\vee}ab) = 0 $ unless $ \bar{a} + \bar{b} + \bar{c} = \bar{0} $. Next, we have
\begin{align*}
\begin{tikzpicture}[anchorbase]
        \draw[-] (0,-0.5) --(0,0.5);
        \draw[-] (0.5,-0.5) -- (0.5,0.5);
        \token{0,-0.3}{east}{\psi(a)};
        \teleport{0,0.1}{0.5,-0.1};
\end{tikzpicture}
\stackrel{\cref{expansion}}{=} \sum_{b,c \in B} &(-1)^{\varepsilon\bar{b}+\varepsilon\bar{a}+\bar{a}\bar{b}}
\begin{tikzpicture}[centerzero]
        \draw (0,-0.5) -- (0,0.5);
        \draw (0.5,-0.5) -- (0.5,0.5);
        \token{0,0.1}{east}{\operatorname{tr}(c^{\vee}b\psi(a))c};
        \token{0.5,-0.1}{west}{b^{\vee}};
\end{tikzpicture}
\stackrel{\cref{supersymmetric}}{=} \sum_{b,c \in B} (-1)^{\varepsilon \bar{b} + \bar{a}\bar{c}}
\begin{tikzpicture}[centerzero]
        \draw (0,-0.5) -- (0,0.5);
        \draw (0.5,-0.5) -- (0.5,0.5);
        \token{0,0.1}{east}{c};
        \token{0.5,-0.1}{west}{\operatorname{tr}(ac^{\vee}b)b^{\vee}};
\end{tikzpicture}
\\ &= \sum_{b,c \in B} (-1)^{\varepsilon\bar{c}+\varepsilon{\bar{a}}+\bar{a}\bar{c}}
\begin{tikzpicture}[centerzero]
        \draw (0,-0.5) -- (0,0.5);
        \draw (0.5,-0.5) -- (0.5,0.5);
        \token{0,0.1}{east}{c};
        \token{0.5,-0.1}{west}{\operatorname{tr}(ac^{\vee}b)b^{\vee}};
\end{tikzpicture}
\stackrel{\cref{expansion}}{=} (-1)^{\varepsilon \bar{a}}\
\begin{tikzpicture}[anchorbase]
        \draw[-] (0,-0.5) --(0,0.5);
        \draw[-] (0.5,-0.5) -- (0.5,0.5);
        \token{0.5,0.3}{west}{a};
        \teleport{0,0.1}{0.5,-0.1};
\end{tikzpicture} \ ,
\end{align*}
where in the third equality, we used the fact that $ \operatorname{tr}(ac^{\vee}b) = 0 $ unless $ \bar{a}+\bar{b}+\bar{c} = \bar{0} $. The relation \cref{teleporter2} is then obtained by computing the image of \cref{teleporter1} under the isomorphism 
\begin{equation*}
\mathpzc{tower}(A) \rightarrow \mathpzc{tower}(A)^{\operatorname{rev}}, 
\quad 
\begin{tikzpicture}[centerzero]
      \draw[-] (0,-0.2) -- (0,0.2);
      \token{0,0}{west}{d};
\end{tikzpicture} 
\mapsto 
\begin{tikzpicture}[centerzero]
      \draw[-] (0,-0.2) -- (0,0.2);
      \token{0,0}{west}{d};
\end{tikzpicture} \ , \qquad d \in A. \qedhere
\end{equation*}
\end{proof}

\section{Affine wreath product algebras} \label{affine_wreath_product}

In this section, we provide a recap of the Frobenius polynomial algebras, the Frobenius Demazure operators, and the affine wreath product algebras. The affine wreath product algebras were first defined and studied by Savage in \cite{Savage}. In \cite{Savage}, it is assumed that the Frobenius algebra carries an even trace map. However, the affine wreath product algebras were generalized in \cite{Mendonca} to allow for trace maps of arbitrary parity. In this paper, since \(A\) can carry either an even or odd trace map, we will use the definition provided in \cite{Mendonca}.

\begin{defn} \label{summer}
We define the \emph{Frobenius polynomial category} $ \mathpzc{Pol}(A) $ to be the strict $ \Bbbk $-linear monoidal category generated by one object $ \go $, with generating morphisms
\begin{equation*}
\begin{tikzpicture}[centerzero]
      \draw[-] (0,-0.2) -- (0,0.2);
      \singdot{0,0};
\end{tikzpicture} \colon \go \rightarrow \go,
\ , \quad
\begin{tikzpicture}[centerzero]
      \draw[-] (0,-0.2) -- (0,0.2);
      \token{0,0}{west}{a};
\end{tikzpicture} \colon \go \rightarrow \go,
\quad a \in A,
\end{equation*}
subject to the relations \cref{tokrel1} and
\begin{equation} \label{fox}
\begin{tikzpicture}[centerzero]
      \draw[-] (0,-0.3) -- (0,0.3);
      \token{0,0.12}{east}{a};
      \singdot{0,-0.12};
\end{tikzpicture}
= (-1)^{\varepsilon\bar{a}} \
\begin{tikzpicture}[centerzero]
      \draw[-] (0,-0.3) -- (0,0.3);
      \token{0,-0.12}{west}{\psi(a)};
      \singdot{0,0.12};
\end{tikzpicture}
\ , \quad a \in A.
\end{equation}
We refer to the generator $ \begin{tikzpicture}[centerzero]
      \draw[-] (0,-0.2) -- (0,0.2);
      \singdot{0,0};
\end{tikzpicture} $ as the \emph{dot}, and declare it to be of parity $ \varepsilon $. The parity of the token $ \begin{tikzpicture}[centerzero]
      \draw[-] (0,-0.2) -- (0,0.2);
      \token{0,0}{west}{a};
\end{tikzpicture} $ is equal to the parity of $ a $. For $ n \in \mathbb{N} $, we define the \emph{Frobenius polynomial algebra}
\begin{equation*}
\operatorname{Pol}_{n}(A) := \operatorname{End}_{\mathpzc{Pol}(A)}(\go^{\otimes n}).
\end{equation*}
\end{defn}

\begin{prop} \label{Polrel}
As a superalgebra, $ \operatorname{Pol}_{n}(A) $ is isomorphic to the free product of $ A^{\otimes n} $ and the free associative superalgebra on the generators $ x_{1},\ldots, x_{n} $ of parity $ \varepsilon $, subject to the relations
\begin{align}
        x_{i}x_{j} &= (-1)^{\varepsilon}x_{j}x_{i}, & 1 \leq i,j \leq n, \ i \neq j  \label{F1_1} \\
        \mathbf{a}x_{i} &= (-1)^{\varepsilon\overline{\mathbf{a}}}x_{i}\psi_{i}(\mathbf{a}), & 1 \leq i \leq n, \ \mathbf{a} \in A^{\otimes n}. \label{F1_2}
\end{align}
Under this isomorphism, $ a_{i} $, for $ 1 \leq i \leq n $, corresponds to a token labelled $ a $ on strand $ i $, and $ x_{i} $, for $ 1 \leq i \leq n $, corresponds to a dot on strand $ i $. 
\end{prop}

\begin{proof}
This follows from Proposition \ref{zebra}.
\end{proof}

From here onwards, we identify $ \operatorname{Pol}_{n}(A) $ with the algebra given in Proposition \ref{Polrel}. The algebra $ \operatorname{Pol}_{n}(A) $ has basis
\begin{align} \label{Poly_base}
&\{\mathbf{a}x_{1}^{k_{1}}\cdots x_{n}^{k_{n}} : \mathbf{a} \in B^{\otimes n}, \ k_{1},\ldots, k_{n} \in \mathbb{N}\},
\end{align}
where recall that \(B\) is some basis of \(A\). We define $ x := x_{1} \in \operatorname{Pol}_{1}(A) $. Then the following lemma is easy to check.

\begin{lem} \label{identification}
We have an isomorphism of superalgebras $ \operatorname{Pol}_{n}(A) \rightarrow \operatorname{Pol}_{1}(A)^{\otimes n} $ that sends $ x_{i} $ to $ 1^{\otimes i-1} \otimes x \otimes 1^{\otimes n-i} $, and sends $ \mathbf{a} \in A^{\otimes n} $ to itself.
\end{lem}

The superpermutation action of the symmetric group $ S_{n} $ on $ A^{\otimes n} $ extends naturally to an action on $ \operatorname{Pol}_{n}(A) $ via superalgebra isomorphisms, where
\begin{equation*}
w(x_{i}) = x_{w(i)}, \quad 1 \leq i \leq n, \ w \in S_{n}.
\end{equation*} 
We denote this action by $ w(f) $, for $ w \in S_{n} $, $ f \in \operatorname{Pol}_{n}(A) $.

\begin{defn}[{\cite[§4.1]{Savage}}, {\cite[§4.1]{Mendonca}}] \label{autumn}
We define the \emph{Frobenius Demazure operators} (or the \emph{Frobenius divided difference operators}) to be the $ \Bbbk $-linear operators $ \partial_{i} \colon \operatorname{Pol}_{n}(A) \rightarrow  \operatorname{Pol}_{n}(A) $, $ 1 \leq i \leq n-1 $, determined by 
\begin{equation} \label{autumn22}
\partial_{i}(\mathbf{a}) = 0, \quad \partial_{i}(x_{j}) = 
\begin{cases}
t_{i,i+1} &\text{if} \ j = i,
\\ -t_{i+1,i} &\text{if} \ j = i+1,
\\ 0 &\text{otherwise},
\end{cases}
\end{equation}
for all $ \mathbf{a} \in A^{\otimes n} $, $ 1 \leq j \leq n $, and the twisted Leibniz rule
\begin{equation} \label{leibniz2}
\partial_{i}(fg) = \partial_{i}(f)g + s_{i}(f)\partial_{i}(g), \quad f,g \in \operatorname{Pol}_{n}(A).
\end{equation}
\end{defn}

We define $ Z(\operatorname{Pol}_{n}(A))^{S_{n}} $ to be the set of those elements in $ Z(\operatorname{Pol}_{n}(A)) $ that are invariant under the action of the symmetric group $ S_{n} $. 

\begin{lem} \label{treat2}
Let $ n \in \mathbb{N}_{+} $ and $ 1 \leq i \leq n-1 $. Then $ Z(\operatorname{Pol}_{n}(A))^{S_{n}} \subseteq \operatorname{ker}(\partial_{i}) $.
\end{lem}

\begin{proof}
Let $ f \in Z(\operatorname{Pol}_{n}(A))^{S_{n}} $, and let $ g \in \operatorname{Pol}_{n}(A) $ be an even element. Then since $ f $ satisfies $ s_{i}(f) = f $, we obtain by \cref{leibniz2} that $ \partial_{i}(fg) = \partial_{i}(f)g + f\partial_{i}(g) $. Furthermore, since \(f\) is central, we compute that
\begin{equation*}
\partial_{i}(fg) = \partial_{i}(gf) = \partial_{i}(g)f + s_{i}(g)\partial_{i}(f) = f\partial_{i}(g) + s_{i}(g)\partial_{i}(f).
\end{equation*}
Thus $ \partial_{i}(f)g = s_{i}(g)\partial_{i}(f) $. In particular, by taking $ g = x_{i}^{2} $, we obtain that
\begin{equation} \label{email}
\partial_{i}(f)x_{i}^{2} =  x_{i+1}^{2}\partial_{i}(f).
\end{equation}
Then, by using the basis \cref{Poly_base}, it is straightforward to see that the condition \cref{email} implies that $ \partial_{i}(f) = 0 $.
\end{proof}

\details{Suppose $ \partial_{i}(f) \neq 0 $. We define $ F $ to be the subalgebra of $ \operatorname{Pol}_{n}(A) $ generated by $ A^{\otimes n} $ and $ x_{1},\ldots,x_{i-1},x_{i+1},\ldots,x_{n} $. Then $ \operatorname{Pol}_{n}(A) $ is free as a left $ F $-module, with basis $ \{x_{i}^{k} : k \in \mathbb{N}\} $. We denote the degree of an element $ g \in \operatorname{Pol}_{n}(A) $, when written as a polynomial in $ x_{i} $ with coefficients in $ F $, by $ \operatorname{deg}(g) $. Then since $ \partial_{i}(f) \neq 0 $, we have by \cref{email} that 
\begin{equation*}
\operatorname{deg}(\partial_{i}(f)) + 2 = \operatorname{deg}(\partial_{i}(f)x_{i}^{2}) = \operatorname{deg}(x_{i+1}^{2}\partial_{i}(f)) = \operatorname{deg}(x_{i+1}^{2}) + \operatorname{deg}(\partial_{i}(f)) = \operatorname{deg}(\partial_{i}(f)),
\end{equation*}
which is a contradiction. Thus $ \partial_{i}(f) = 0 $.}

\begin{defn} \label{pluto2}
We define the \emph{affine wreath product category} $ \mathcal{AW}(A) $ to be the strict $ \Bbbk $-linear monoidal supercategory generated by one object $ \go $ and morphisms
\begin{equation*}
    \begin{tikzpicture}[centerzero, thick]
        \draw[-] (-0.2,-0.2) -- (0.2,0.2);
        \draw[-] (0.2,-0.2) -- (-0.2,0.2);
    \end{tikzpicture}
      \ \colon
    \go \otimes \go \to \go \otimes \go,
    \quad
    \begin{tikzpicture}[anchorbase, thick]
        \draw[-] (0,-0.3) -- (0,0.3);
        \token{0,0}{west}{a};
    \end{tikzpicture}
    \colon \go \to \go, \quad
 \begin{tikzpicture}[anchorbase, thick]
        \draw[-] (0,0) -- (0,0.6);
        \singdot{0,0.3};
\end{tikzpicture}
\ \colon
\go \to \go, \quad a \in A,
\end{equation*}
modulo the relations \cref{tokrel1} and
\begin{gather} 
\label{braid2}
\begin{tikzpicture}[anchorbase, thick]
        \draw[-] (0.2,-0.5) to[out=up,in=down] (-0.2,0) to[out=up,in=down] (0.2,0.5);
        \draw[-] (-0.2,-0.5) to[out=up,in=down] (0.2,0) to[out=up,in=down] (-0.2,0.5);
\end{tikzpicture}
\ =\
\begin{tikzpicture}[anchorbase, thick]
        \draw[-] (-0.2,-0.5) -- (-0.2,0.5);
        \draw[-] (0.2,-0.5) -- (0.2,0.5);
\end{tikzpicture}
\ ,\quad
\begin{tikzpicture}[anchorbase, thick]
        \draw[-] (0.4,-0.5) -- (-0.4,0.5);
        \draw[-] (0,-0.5) to[out=up, in=down] (-0.4,0) to[out=up,in=down] (0,0.5);
        \draw[-] (-0.4,-0.5) -- (0.4,0.5);
\end{tikzpicture}
\ =\
\begin{tikzpicture}[anchorbase, thick]
        \draw[-] (0.4,-0.5) -- (-0.4,0.5);
        \draw[-] (0,-0.5) to[out=up, in=down] (0.4,0) to[out=up,in=down] (0,0.5);
        \draw[-] (-0.4,-0.5) -- (0.4,0.5);
\end{tikzpicture}
\ ,\quad
\begin{tikzpicture}[centerzero, thick]
        \draw[-] (0.3,-0.4) -- (-0.3,0.4);
        \draw[-] (-0.3,-0.4) -- (0.3,0.4);
        \token{-0.15,-0.2}{east}{a};
\end{tikzpicture}
\ =\
\begin{tikzpicture}[centerzero, thick]
        \draw[-] (0.3,-0.4) -- (-0.3,0.4);
        \draw[-] (-0.3,-0.4) -- (0.3,0.4);
        \token{0.15,0.2}{west}{a};
\end{tikzpicture}
\ ,
\\ 
\label{affwreath}
\begin{tikzpicture}[centerzero, thick]
        \draw[-] (0.3,-0.4) -- (-0.3,0.4);
        \draw[-] (-0.3,-0.4) -- (0.3,0.4);
        \singdot{-0.15,-0.2};
\end{tikzpicture}
\ =\
\begin{tikzpicture}[centerzero, thick]
        \draw[-] (-0.3,-0.4) -- (0.3,0.4);
        \draw[-] (0.3,-0.4) -- (-0.3,0.4);
        \singdot{0.171,0.228};
\end{tikzpicture}
\ - \ 
\begin{tikzpicture}[centerzero]
        \draw (-0.2,-0.4) -- (-0.2,0.4);
        \draw (0.2,-0.4) -- (0.2,0.4);
        \teleport{-0.2,0.1}{0.2,-0.1};
\end{tikzpicture}
\ ,\quad
\begin{tikzpicture}[centerzero, thick]
        \draw[-] (0,-0.4) -- (0,0.4);
        \token{0,0.15}{west}{a};
        \singdot{0,-0.15};
\end{tikzpicture}
\ = (-1)^{\varepsilon \bar{a}} \
\begin{tikzpicture}[centerzero, thick]
        \draw[-] (0,-0.4) -- (0,0.4);
        \token{0,-0.15}{west}{\psi(a)};
        \singdot{0,0.15};
\end{tikzpicture}
\ ,\quad a \in A.
\end{gather}
The morphism $ \begin{tikzpicture}[anchorbase, thick]
        \draw[-] (0,0) -- (0,0.6);
        \singdot{0,0.3};
    \end{tikzpicture} $
is called the \emph{dot}, which is of parity $ \varepsilon $. The black crossing is even, and the parity of a token $ \begin{tikzpicture}[anchorbase, thick]
        \draw[-] (0,-0.2) -- (0,0.2);
        \token{0,0}{west}{a};
    \end{tikzpicture} $ is equal to the parity of \(a\). For $ n \in \mathbb{N} $, we define the \emph{affine wreath product algebra} to be
\begin{equation*}
W_{n}^{\operatorname{aff}}(A) := \operatorname{End}_{\mathcal{AW}(A)}(\go^{\otimes n}).
\end{equation*}
\end{defn}

It follows from \cref{braid2} that
\begin{equation} \label{tokencross2}
\begin{tikzpicture}[centerzero, thick]
        \draw[-] (0.3,-0.4) -- (-0.3,0.4);
        \draw[-] (-0.3,-0.4) -- (0.3,0.4);
        \token{0.15,-0.2}{west}{a};
\end{tikzpicture}
\ =\
\begin{tikzpicture}[centerzero, thick]
        \draw[-] (0.3,-0.4) -- (-0.3,0.4);
        \draw[-] (-0.3,-0.4) -- (0.3,0.4);
        \token{-0.15,0.2}{east}{a};
\end{tikzpicture} \ , \qquad a \in A.
\end{equation}
Furthermore, we have by \cref{braid2} and \cref{affwreath} that
\begin{equation} \label{affwreath223}
\begin{tikzpicture}[centerzero, thick]
        \draw[-] (0.3,-0.4) -- (-0.3,0.4);
        \draw[-] (-0.3,-0.4) -- (0.3,0.4);
        \singdot{-0.15,0.2};
\end{tikzpicture}
\ =\
\begin{tikzpicture}[centerzero, thick]
        \draw[-] (-0.3,-0.4) -- (0.3,0.4);
        \draw[-] (0.3,-0.4) -- (-0.3,0.4);
        \singdot{0.171,-0.228};
\end{tikzpicture}
\ - \
\begin{tikzpicture}[centerzero]
        \draw (-0.2,-0.4) -- (-0.2,0.4);
        \draw (0.2,-0.4) -- (0.2,0.4);
        \teleport{-0.2,-0.1}{0.2,0.1};
\end{tikzpicture} \ .
\end{equation}

\begin{prop} 
The algebra $ W_{n}^{\operatorname{aff}}(A) $ is isomorphic to the free product of $ A^{\otimes n} $ and the free associative superalgebra on the generators $ x_{1}, \ldots, x_{n}, \sigma_{1},\ldots,\sigma_{n-1} $, subject to the relations \cref{F1_1}, \cref{F1_2}, and
\begin{align}
\sigma_{i}\sigma_{j} &= \sigma_{j}\sigma_{i}, & 1 \leq i,j \leq n-1, |i-j| > 1, \label{mars1}
\\ \sigma_{i}\sigma_{i+1}\sigma_{i} &= \sigma_{i+1}\sigma_{i}\sigma_{i+1}, & 1 \leq i \leq n-2, \label{mars2}
\\ \sigma_{i}^{2} &= 1, & 1 \leq i \leq n-1, \label{mars3}
\\ \sigma_{i}\mathbf{a} &= s_{i}(\mathbf{a})\sigma_{i}, & \mathbf{a} \in A^{\otimes n}, 1 \leq i \leq n-1. \label{mars4}
\\ x_{j}\sigma_{i} &= \sigma_{i}x_{j}, & 1 \leq i \leq n-1, \ 1 \leq j \leq n, \ j \neq i,i+1, \label{ruler3} \\
        \sigma_{i}x_{i} &= x_{i+1}\sigma_{i} - t_{i,i+1}, & 1 \leq i \leq n-1. \label{ruler4}
\end{align}
Here, the $ x_{i} $ are of parity $ \varepsilon $ and the $ \sigma_{i} $ are even. Under this isomorphism, $ a_{i} $, for $ 1 \leq i \leq n $, corresponds to the token labelled $ a $ on the $ i $-th strand, $ x_{i} $, for $ 1 \leq i \leq n $, corresponds to the dot on the $ i $-th strand, and $ \sigma_{i} $, for $ 1 \leq i \leq n-1 $, corresponds to the black crossing of the $ i $-th and $ (i+1) $-th strands.
\end{prop}

\begin{proof}
This follows from Proposition \ref{zebra}.
\end{proof}

We have an algebra homomorphism $ \Bbbk S_{n} \rightarrow  W_{n}^{\operatorname{aff}}(A) $ given by $ s_{i} \mapsto \sigma_{i} $, and we denote the image of $ w \in S_{n} $ under this map by $ \sigma_{w} $. Recall that \(B\) is a basis of \(A\). Then we have the following basis of $ W_{n}^{\operatorname{aff}}(A) $; see \cite[Cor.~4.8]{Savage} and \cite[Thm.~4.2.3]{Mendonca}.

\begin{prop} \label{midnight2}
The algebra $ W_{n}^{\operatorname{aff}}(A) $ has basis
\begin{equation*}
\{\mathbf{a}x_{1}^{k_{1}}\cdots x_{n}^{k_{n}}\sigma_{w}: \mathbf{a} \in B^{\otimes n}, \ k_{1},\ldots, k_{n} \in \mathbb{N}, \ w \in S_{n}\}.
\end{equation*}
\end{prop}

Recall that we have the Frobenius polynomial algebra $ \operatorname{Pol}_{n}(A) $ from Definition \ref{summer}. There is a natural algebra homomorphism $ \operatorname{Pol}_{n}(A) \rightarrow W_{n}^{\operatorname{aff}}(A) $ that maps $ \mathbf{a} \in A^{\otimes n} $ and $ x_{i} $ to the elements in $ W_{n}^{\operatorname{aff}}(A) $ of the same name. By Proposition \ref{midnight2}, this homomorphism is injective, and so we may view $ \operatorname{Pol}_{n}(A) $ as a subalgebra of $ W_{n}^{\operatorname{aff}}(A) $.
 
\begin{lem}[{\cite[Lem.~4.1.1]{Mendonca}, \cite[Lem.~4.1]{Savage}}]
For all $ f \in \operatorname{Pol}_{n}(A) $ and $ 1 \leq i \leq n-1 $, we have in $ W_{n}^{\operatorname{aff}}(A) $ that
\begin{equation}
\sigma_{i}f = s_{i}(f)\sigma_{i} - \partial_{i}(f). \label{movein}
\end{equation}
\end{lem}

\begin{prop} \label{aff_center_result}
The center of $ W_{n}^{\operatorname{aff}}(A) $ is equal to $ Z(\operatorname{Pol}_{n}(A))^{S_{n}} $.
\end{prop}

\begin{proof}
Let $ f \in Z(\operatorname{Pol}_{n}(A))^{S_{n}} $. Then \(f\) commutes with the elements of $ \operatorname{Pol}_{n}(A) $. Furthermore, Lemma \ref{treat2} and \cref{movein} imply that $ \sigma_{i}f = f\sigma_{i} $ for all $ 1 \leq i \leq n-1 $, from which we obtain that $ f $ lies in the center of $ W_{n}^{\operatorname{aff}}(A) $.
\\ \indent For the reverse inclusion, let $ f \in Z(W_{n}^{\operatorname{aff}}(A)) $. Then, by a similar argument to \cite[Lem.~4.12]{Savage}, we have that $ f \in Z(\operatorname{Pol}_{n}(A)) $. Furthermore, \cref{movein} implies that $ f\sigma_{i} = \sigma_{i}f = s_{i}(f)\sigma_{i} - \partial_{i}(f) $ for all $ 1 \leq i \leq n-1 $. Thus, by Proposition \ref{midnight2}, we have $ s_{i}(f) = f $ as required.
\end{proof}

For a more explicit description of the center, we refer the reader to \cite[Thm.~4.14]{Savage}, \cite[Thm.~4.3.1]{Mendonca}, and \cite[Thm.~4.3.4]{Mendonca_thesis}.

\section{Higher-level affine wreath product algebras} \label{higher_level_section}

In this section, we define and study the higher-level affine wreath product category and the higher-level affine wreath product algebras. 

\subsection{Higher-level affine wreath product algebras}

We set $ x := x_{1} \in \operatorname{Pol}_{1}(A) $. Given $ n \in \mathbb{N}_{+} $ and $ 1 \leq i \leq n $, we have an injective algebra homomorphism
\begin{equation}
\operatorname{Pol}_{1}(A) \rightarrow \operatorname{Pol}_{n}(A), \quad a \mapsto a_{i}, \ x \mapsto x_{i}, \qquad a \in A.
\end{equation}
Here, $ a_{i} $ is defined as in \cref{helium}. We denote the image of an element $ f \in \operatorname{Pol}_{1}(A) $ under this homomorphism by $ f_{i} $. Explicitly, if $ f = \sum_{k \geq 0} a_{(k)}x^{k} $, $ a_{(k)} \in A $, then 
\begin{equation} \label{fog}
f_{i} = \sum_{k \geq 0} (a_{(k)})_{i}(x_{i})^{k} \in \operatorname{Pol}_{n}(A).
\end{equation}
Recall that, for an algebra $ F $, we have the sets $ \mathcal{E}_{F} $ and $ \widehat{\mathcal{E}}_{F} $ from Definition \ref{multi_pen}.

\begin{lem} \label{uncharted5678}
Let $ n \in \mathbb{N} $, $ 1 \leq i \leq n $, and $ f \in \mathcal{E}_{\operatorname{Pol}_{1}(A)} $. Then $ f_{i} \in \mathcal{E}_{\operatorname{Pol}_{n}(A)} $.
\end{lem}

\begin{proof}
We must show that $ f_{i} $ is even, central, and regular in $ \operatorname{Pol}_{n}(A) $. Since \(f\) is even and central, it is straightforward to see that $ f_{i} $ is also even and central. Next, to show that $ f_{i} $ is regular in $ \operatorname{Pol}_{n}(A) $, it suffices to show that $ f_{i} $ is not a left zero-divisor in $ \operatorname{Pol}_{n}(A) $ (since $ f_{i} $ is central). It follows by the regularity of \(f\) that the map 
\begin{equation*}
L \colon \operatorname{Pol}_{1}(A) \rightarrow \operatorname{Pol}_{1}(A), \ g \mapsto fg 
\end{equation*}
is injective. If \(M\) is a flat $ \Bbbk $-module, and $ \varphi \colon M \rightarrow M $ is an injective $ \Bbbk $-module homomorphism, then the map
\begin{equation} \label{water_bottle}
\operatorname{Id}_{M^{\otimes (i-1)}} \otimes \varphi \otimes \operatorname{Id}_{M^{\otimes (n-i)}} \colon M^{\otimes n} \rightarrow M^{\otimes n}
\end{equation}
is injective. Applying this with $ M = \operatorname{Pol}_{1}(A) $ (which is a free $ \Bbbk $-module and hence a flat $ \Bbbk $-module) and $ \varphi = L $, and using the identification of Lemma \ref{identification}, we see that the injective map \cref{water_bottle} corresponds to multiplication on the left by $ f_{i} $. Thus $ f_{i} $ is not a left zero-divisor in $ \operatorname{Pol}_{n}(A) $.
\end{proof}

\details{We prove the following assertion: If \(M\) is a flat $ \Bbbk $-module, and $ \varphi \colon M \rightarrow M $ is an injective $ \Bbbk $-module homomorphism, then the map \cref{water_bottle} is injective.
\\ \indent Since \(M\) is flat, it follows that $ M^{\otimes (i-1)} $ is also flat. Hence the map
\begin{equation} \label{calendar}
\operatorname{Id}_{M^{\otimes (i-1)}} \otimes \varphi \colon M^{\otimes i} \rightarrow M^{\otimes i}
\end{equation}
is injective. Furthermore, since $ M^{\otimes (n-i)} $ is flat, and the map \cref{calendar} is injective, it follows that the map
\begin{equation*}
(\operatorname{Id}_{M^{\otimes (i-1)}} \otimes \varphi) \otimes \operatorname{Id}_{M^{\otimes (n-i)}} \colon M^{\otimes n} \rightarrow M^{\otimes n}
\end{equation*}
is injective. Therefore, the map \cref{water_bottle} is injective.}

\begin{defn} \label{grab}
We define the \emph{higher-level affine wreath product category}  $ \mathcal{LAW}(A) $ to be the strict $ \Bbbk $-linear monoidal supercategory defined as follows. The objects are generated by the elements in the set $ \widehat{\mathcal{E}}_{\operatorname{Pol}_{1}(A)} $. The morphisms are generated by
\begin{gather}
    \begin{tikzpicture}[centerzero, thick]
        \draw[-] (-0.2,-0.2) -- (0.2,0.2);
        \draw[-] (0.2,-0.2) -- (-0.2,0.2);
    \end{tikzpicture}
      \ \colon
    \go \otimes \go \to \go \otimes \go,
    \quad
    \begin{tikzpicture}[anchorbase, thick]
        \draw[-] (0,-0.3) -- (0,0.3);
        \token{0,0}{west}{a};
    \end{tikzpicture}
    \colon \go \to \go, \quad
 \begin{tikzpicture}[anchorbase, thick]
        \draw[-] (0,0) -- (0,0.6);
        \singdot{0,0.3};
\end{tikzpicture}
\ \colon
\go \to \go, \label{Gen1} \\
\begin{tikzpicture}[centerzero, thick]
        \draw[-] (0.2,-0.2) -- (-0.2,0.2);
        \draw[-, wei] (-0.2,-0.2) -- (0.2,0.2);
        \node at (-0.2,-0.4) {\tiny $ Q $};
\end{tikzpicture}
\ \colon Q \otimes \go \to \go \otimes Q, \qquad
\begin{tikzpicture}[centerzero, thick]
        \draw[-] (-0.2,-0.2) -- (0.2,0.2);
        \draw[-,wei] (0.2,-0.2) -- (-0.2,0.2);
        \node at (0.2,-0.4) {\tiny $ Q $};
\end{tikzpicture}
\colon \go \otimes Q \to Q \otimes \go, \label{Gen2}
\end{gather}
for all $ a \in A $, $ Q \in \mathcal{E}_{\operatorname{Pol}_{1}(A)} $. The crossings are even, the dot is of parity $ \varepsilon $, and the parity of a token $ \begin{tikzpicture}[anchorbase, thick]
        \draw[-] (0,-0.2) -- (0,0.2);
        \token{0,0}{west}{a};
    \end{tikzpicture} $ is equal to the parity of \(a\). Before stating the relations, we say a bit more on our diagrammatic conventions. When a dot is labelled by a multiplicity, we mean that we are taking its power under vertical composition. Recall that we have the element $ x = x_{1} \in \operatorname{Pol}_{1}(A) $. If $ f = \sum_{r\geq 0} a_{(r)}x^{r} \in \operatorname{Pol}_{1}(A) $, $ a_{(r)} \in A $, then we pin \(f\) to a string by defining
\begin{equation}
    \begin{tikzpicture}[centerzero, thick]
        \draw[-] (0,-0.5) -- (0,0.5);
             \pin{0,0}{-.7,0}{f};
    \end{tikzpicture}
    = \begin{tikzpicture}[centerzero, thick]
        \draw[-] (0,-0.5) -- (0,0.5);
             \pin{0,0}{.7,0}{f};
    \end{tikzpicture}\
    :=
\sum_{r\geq 0} \
\begin{tikzpicture}[centerzero, thick]
        \draw[-] (-0.3,-0.5) -- (-0.3,0.5);
        \multdot{-0.3,-0.2}{west}{r};
        \token{-0.3,0.2}{west}{a_{(r)}};
\end{tikzpicture} \ .
\end{equation}
We set $ x := x_{1} \in \operatorname{Pol}_{2}(A) $ and $ y := x_{2} \in \operatorname{Pol}_{2}(A) $. If $ f = \sum_{b,c \in B} \sum_{r,s \geq 0} \lambda_{b,c,r,s} (b \otimes c)x^{r}y^{s} \in \operatorname{Pol}_{2}(A) $, $\lambda_{b,c,r,s} \in \Bbbk $, then we define
\begin{equation} \label{red}
\begin{tikzpicture}[centerzero, thick]
        \draw[-] (-0.3,-0.5) -- (-0.3,0.5);
        \draw[-] (0.3,-0.5) -- (0.3,0.5);
        \pinpin{-0.3,0}{0.3,0}{-1,0}{f};
\end{tikzpicture} 
=
\begin{tikzpicture}[centerzero, thick]
        \draw[-] (-0.3,-0.5) -- (-0.3,0.5);
        \draw[-] (0.3,-0.5) -- (0.3,0.5);
        \pinpin{0.3,0}{-0.3,0}{1,0}{f};
\end{tikzpicture} 
:=
\sum_{b,c \in B} \sum_{r,s \geq 0} \lambda_{b,c,r,s}
\begin{tikzpicture}[centerzero, thick]
        \draw[-] (-0.3,-0.5) -- (-0.3,0.5);
        \draw[-] (0.3,-0.5) -- (0.3,0.5);
        \token{-0.3,0.2}{east}{b};
        \token{0.3,0.2}{west}{c};
        \multdot{-0.3,-0.2}{east}{r};
        \multdot{0.3,-0.2}{west}{s};
\end{tikzpicture} \ .
\end{equation}
The relations on the morphisms are given by \cref{tokrel1}, \cref{braid2}, \cref{affwreath}, and
\begin{gather}
\begin{tikzpicture}[centerzero, thick]
        \draw[-] (-0.3,-0.4) -- (0.3,0.4);
        \draw[-, wei] (0.3,-0.4) -- (-0.3,0.4);
        \token{-0.15,-0.2}{east}{a};
        \node at (0.3,-0.6) {\tiny $ Q $};
\end{tikzpicture}
= 
\begin{tikzpicture}[centerzero, thick]
        \draw[-] (-0.3,-0.4) -- (0.3,0.4);
        \draw[-, wei] (0.3,-0.4) -- (-0.3,0.4);
        \token{0.15,0.2}{west}{a};
        \node at (0.3,-0.6) {\tiny $ Q $};
\end{tikzpicture}
\ , \quad
\begin{tikzpicture}[centerzero, thick]
        \draw[-] (0.3,-0.4) -- (-0.3,0.4);
        \draw[-, wei] (-0.3,-0.4) -- (0.3,0.4);
        \token{0.15,-0.2}{west}{a};
        \node at (-0.3,-0.6) {\tiny $ Q $};
\end{tikzpicture}
= 
\begin{tikzpicture}[centerzero, thick]
        \draw[-] (0.3,-0.4) -- (-0.3,0.4);
        \draw[-, wei] (-0.3,-0.4) -- (0.3,0.4);
        \token{-0.15,0.2}{east}{a};
        \node at (-0.3,-0.6) {\tiny $ Q $};
\end{tikzpicture} \ , \label{hw1}
\\ 
\begin{tikzpicture}[centerzero, thick]
        \draw[-] (-0.3,-0.4) -- (0.3,0.4);
        \draw[-, wei] (0.3,-0.4) -- (-0.3,0.4);
        \singdot{-0.15,-0.2};
        \node at (0.3,-0.6) {\tiny $ Q $};
\end{tikzpicture}
= 
\begin{tikzpicture}[centerzero, thick]
        \draw[-] (-0.3,-0.4) -- (0.3,0.4);
        \draw[-, wei] (0.3,-0.4) -- (-0.3,0.4);
        \singdot{0.15,0.2};
        \node at (0.3,-0.6) {\tiny $ Q $};
\end{tikzpicture}
\ , \quad
\begin{tikzpicture}[centerzero, thick]
        \draw[-] (0.3,-0.4) -- (-0.3,0.4);
        \draw[-, wei] (-0.3,-0.4) -- (0.3,0.4);
        \singdot{0.15,-0.2};
        \node at (-0.3,-0.6) {\tiny $ Q $};
\end{tikzpicture}
= 
\begin{tikzpicture}[centerzero, thick]
        \draw[-] (0.3,-0.4) -- (-0.3,0.4);
        \draw[-, wei] (-0.3,-0.4) -- (0.3,0.4);
        \singdot{-0.15,0.2};
        \node at (-0.3,-0.6) {\tiny $ Q $};
\end{tikzpicture} \ , \label{hw2}
\\ 
\begin{tikzpicture}[centerzero, thick]
        \draw[-] (-0.2,-0.5) to[out=up,in=down] (0.2,0) to[out=up,in=down] (-0.2,0.5);
        \draw[-, wei] (0.2,-0.5) to[out=up,in=down] (-0.2,0) to[out=up,in=down] (0.2,0.5);
        \node at (0.2,-0.7) {\tiny $ Q $};
\end{tikzpicture}
=
\begin{tikzpicture}[centerzero, thick]
        \pin{-0.2,0}{-1,0}{Q};
        \draw[-] (-0.2,-0.5) -- (-0.2,0.5);
        \draw[-, wei] (0.2,-0.5) -- (0.2,0.5);
        \node at (0.2,-0.7) {\tiny $ Q $};
\end{tikzpicture} \ , \label{hw3}
\quad
\begin{tikzpicture}[centerzero, thick]
        \draw[-] (0.2,-0.5) to[out=up,in=down] (-0.2,0) to[out=up,in=down] (0.2,0.5);
        \draw[-, wei] (-0.2,-0.5) to[out=up,in=down] (0.2,0) to[out=up,in=down] (-0.2,0.5);
        \node at (-0.2,-0.7) {\tiny $ Q $};
\end{tikzpicture}
\ = \
\begin{tikzpicture}[centerzero, thick]
        \draw[-, wei] (-0.2,-0.5) -- (-0.2,0.5);
        \draw[-] (0.2,-0.5) -- (0.2,0.5);
        \node at (-0.2,-0.7) {\tiny $ Q $};
        \pin{0.2,0}{1,0}{Q};
\end{tikzpicture} \ , 
\\ 
\begin{tikzpicture}[centerzero, thick]
        \draw[-] (0.4,-0.5) -- (-0.4,0.5);
        \draw[-] (0,-0.5) to[out=up, in=down] (-0.4,0) to[out=up,in=down] (0,0.5);
        \draw[-, wei] (-0.4,-0.5) -- (0.4,0.5);
        \node at (-0.4,-0.7) {\tiny $ Q $};
\end{tikzpicture}
\ =\
\begin{tikzpicture}[centerzero, thick]
        \draw[-] (0.4,-0.5) -- (-0.4,0.5);
        \draw[-] (0,-0.5) to[out=up, in=down] (0.4,0) to[out=up,in=down] (0,0.5);
        \draw[-, wei] (-0.4,-0.5) -- (0.4,0.5);
        \node at (-0.4,-0.7) {\tiny $ Q $};
\end{tikzpicture} \ , \label{hw5}
\quad
\begin{tikzpicture}[centerzero, thick]
        \draw[-] (0,-0.5) to[out=up, in=down] (-0.4,0) to[out=up,in=down] (0,0.5);
        \draw[-] (-0.4,-0.5) -- (0.4,0.5);
        \draw[-, wei] (0.4,-0.5) -- (-0.4,0.5);
        \node at (0.4,-0.7) {\tiny $ Q $};
\end{tikzpicture}
\ =\
\begin{tikzpicture}[centerzero, thick]
        \draw[wipe] (0,-0.5) to[out=up, in=down] (0.4,0) to[out=up,in=down] (0,0.5);
        \draw[-] (0,-0.5) to[out=up, in=down] (0.4,0) to[out=up,in=down] (0,0.5);
        \draw[-] (-0.4,-0.5) -- (0.4,0.5);
        \draw[-, wei] (0.4,-0.5) -- (-0.4,0.5);
        \node at (0.4,-0.7) {\tiny $ Q $};
\end{tikzpicture}
\ , 
\\ \label{green}
\begin{tikzpicture}[centerzero, thick]
        \draw[-] (0.4,-0.5) -- (-0.4,0.5);
        \draw[-] (-0.4,-0.5) -- (0.4,0.5);
        \node at (0,-0.7) {\tiny $ Q $};
        \draw[-, wei] (0,-0.5) to[out=up, in=down] (-0.4,0) to[out=up,in=down] (0,0.5);
\end{tikzpicture}
\ =\
\begin{tikzpicture}[centerzero, thick]
        \draw[-] (2,-0.5) -- (1.2,0.5);
        \draw[-] (1.2,-0.5) -- (2,0.5);
        \node at (1.6,-0.7) {\tiny $ Q $};
        \draw[-, wei] (1.6,-0.5) to[out=up, in=down] (2,0) to[out=up,in=down] (1.6,0.5);
\end{tikzpicture}
\ - \
\begin{tikzpicture}[centerzero, thick]
        \node at (0,-0.7) {\tiny $ Q $};
        \draw[-] (-0.6,-0.5) -- (-0.6,0.5);
        \draw[-, wei] (0,-0.5) -- (0,0.5);
        \draw[-] (0.6,-0.5) -- (0.6,0.5);
        \pinpin{0.6,0}{-0.6,0}{1.8,0}{\partial_{1}(Q_{1})};
\end{tikzpicture} \ ,
\end{gather}
for all $ a \in A $, $ Q \in \mathcal{E}_{\operatorname{Pol}_{1}(A)} $. In the relation \cref{green}, the element $ Q_{1} \in \operatorname{Pol}_{2}(A) $ is defined as in \cref{fog}, and the map $ \partial_{1} \colon \operatorname{Pol}_{2}(A) \rightarrow \operatorname{Pol}_{2}(A) $ is the Frobenius Demazure operator from Definition \ref{autumn}. This concludes the definition of $ \mathcal{LAW}(A) $.
\end{defn}

\begin{rmk} \label{int_rmk}
The relation \cref{green} is equivalent to
\begin{equation} \label{green_alternative}
\begin{tikzpicture}[centerzero, thick]
        \draw[-] (0.4,-0.5) -- (-0.4,0.5);
        \draw[-] (-0.4,-0.5) -- (0.4,0.5);
        \node at (0,-0.7) {\tiny $ Q $};
        \draw[-, wei] (0,-0.5) to[out=up, in=down] (-0.4,0) to[out=up,in=down] (0,0.5);
\end{tikzpicture}
\ =\
\begin{tikzpicture}[centerzero, thick]
        \draw[-] (2,-0.5) -- (1.2,0.5);
        \draw[-] (1.2,-0.5) -- (2,0.5);
        \node at (1.6,-0.7) {\tiny $ Q $};
        \draw[-, wei] (1.6,-0.5) to[out=up, in=down] (2,0) to[out=up,in=down] (1.6,0.5);
\end{tikzpicture}
\ + \
\begin{tikzpicture}[centerzero, thick]
        \node at (0,-0.7) {\tiny $ Q $};
        \draw[-] (-0.6,-0.5) -- (-0.6,0.5);
        \draw[-, wei] (0,-0.5) -- (0,0.5);
        \draw[-] (0.6,-0.5) -- (0.6,0.5);
        \pinpin{0.6,0}{-0.6,0}{1.8,0}{\partial_{1}(Q_{2})};
\end{tikzpicture} \ , \qquad Q \in \mathcal{E}_{\operatorname{Pol}_{1}(A)}.
\end{equation}
Indeed, let $ Q \in \mathcal{E}_{\operatorname{Pol}_{1}(A)} $. Then by Lemma \ref{uncharted5678}, the elements $ Q_{1} $ and $ Q_{2} $ lie in the center of $ \operatorname{Pol}_{2}(A) $. Hence their sum $ Q_{1} + Q_{2} $ also lies in the center of $ \operatorname{Pol}_{2}(A) $. Furthermore, we have that $ Q_{1} + Q_{2} $ is symmetric, and so Lemma \ref{treat2} yields that $ \partial_{1}(Q_{1}+Q_{2}) = 0 $. Thus $ \partial_{1}(Q_{1}) = -\partial_{1}(Q_{2}) $.
\end{rmk}

Recall that, for a given set \(X\), $ F(X) $ is the free monoid on \(X\). Then the set of objects in $ \mathcal{LAW}(A) $ is $ F(\widehat{\mathcal{E}}_{\operatorname{Pol}_{1}(A)}) $. 

\begin{defn} \label{fluffy}
Given a positive integer $ d \in \mathbb{N}_{+} $ and a word $ \mathbf{Q} \in F(\mathcal{E}_{\operatorname{Pol}_{1}(A)}) \subseteq F(\widehat{\mathcal{E}}_{\operatorname{Pol}_{1}(A)}) $, we define $ \Gamma_{d,\mathbf{Q}} $ to be the set of all $ (\go^{d},\mathbf{Q}) $-shuffles in $ F(\widehat{\mathcal{E}}_{\operatorname{Pol}_{1}(A)}) $. That is, 
\begin{equation} \label{fluffy2}
\Gamma_{d,\mathbf{Q}} = F(\widehat{\mathcal{E}}_{\operatorname{Pol}_{1}(A)})_{\go^{d},\mathbf{Q}},
\end{equation}
where $ F(\widehat{\mathcal{E}}_{\operatorname{Pol}_{1}(A)})_{\go^{d},\mathbf{Q}} $ is given in Definition \ref{snowing}.
\end{defn}

\begin{defn} \label{HLAWPA}
Let $ d \in \mathbb{N}_{+} $ and $ \mathbf{Q} \in F(\mathcal{E}_{\operatorname{Pol}_{1}(A)}) $. Then we define the $ (d,\mathbf{Q}) $-\emph{affine wreath product algebra} to be 
\begin{equation} \label{HLAWPA2}
W_{d,\mathbf{Q}}^{\operatorname{aff}}(A) := \operatorname{End}_{\operatorname{Add}(\mathcal{LAW}(A))}\left(\bigoplus_{\mathbf{i} \in \Gamma_{d,\mathbf{Q}}} \mathbf{i}\right).
\end{equation}
Here, $ \operatorname{Add}(\mathcal{LAW}(A)) $ is the additive envelope of $ \mathcal{LAW}(A) $. We say that $ |\mathbf{Q}| $ is the \emph{level} of $ W_{d,\mathbf{Q}}^{\operatorname{aff}}(A) $.
\end{defn}

Using language similar to that of \cite{Maksimau-Stroppel}, we will often refer to the algebra $ W_{d,\mathbf{Q}}^{\operatorname{aff}}(A) $ as a \emph{higher-level affine wreath product algebra}. An arbitrary element in $ W_{d,\mathbf{Q}}^{\operatorname{aff}}(A) $ is a $ \Bbbk $-linear combination of diagrams containing tokens, dots and crossings. The parameter \(d\) corresponds to the number of black strands in the diagrams of $ W_{d,\mathbf{Q}}^{\operatorname{aff}}(A) $, and the level $ |\mathbf{Q}| $ corresponds to the number of red strands. The product of two diagrams in $ W_{d,\mathbf{Q}}^{\operatorname{aff}}(A) $ is their vertical composition if this is defined, and if their vertical composition is not defined, then the product is zero.

\begin{egg}
Suppose $ d = 2 $, and suppose that $ \mathbf{Q} = QQ' \in F(\mathcal{E}_{\operatorname{Pol}_{1}(A)}) $ for some $ Q, Q' \in \mathcal{E}_{\operatorname{Pol}_{1}(A)} $. Then two elements of $ W_{d,\mathbf{Q}}^{\operatorname{aff}}(A) $ are given below:
\begin{equation*} 
D_{1} = 
\begin{tikzpicture}[centerzero, thick]
         \draw[-] (0.6,-0.7) to[out=up, in=down] (0.9,0) to[out=up, in=down] (0.3,0.7);
         \draw[-] (0.9,-0.7) to[out=up, in=down] (0,0.7);
         \draw[-,wei] (0,-0.7) to[out=up, in=down] (0.6,0.7);
         \draw[-,wei] (0.3,-0.7) to[out=up, in=down] (0.9,0.7);
         \singdot{0.9,0};
         \node at (0,-0.9) {\tiny $ Q $};
         \node at (0.3,-0.9) {\tiny $ Q' $};
\end{tikzpicture}
+ 3
\begin{tikzpicture}[centerzero, thick]
         \draw[-] (0,-0.7) to[out=up, in=down] (0.3,0.7);
         \draw[-,wei] (0.3,-0.7) to[out=up, in=down] (-0.3,0.7);
         \draw[-,wei] (0.6,-0.7) -- (0.6,0.7);
         \draw[-] (0.9,-0.7) -- (0.9,0.7);
         \token{0.9,0.3}{west}{a};
         \singdot{0.9,-0.2};
         \node at (0.3,-0.9) {\tiny $ Q $};
         \node at (0.6,-0.9) {\tiny $ Q' $};
\end{tikzpicture} \ , \quad 
D_{2} = 
\begin{tikzpicture}[centerzero, thick]
         \draw[-] (0.6,-0.7) to[out=up, in=down] (0,0.7);
         \draw[-] (0,-0.7) to[out=up, in=down] (0.9,0.7);
         \draw[-,wei] (0.9,-0.7) to[out=up, in=down] (0.6,0.7);
         \draw[-,wei] (0.3,-0.7) to[out=up, in=down] (0,0) to[out=up, in=down] (0.3,0.7);
         \node at (0.3,-0.9) {\tiny $ Q $};
         \node at (0.9,-0.9) {\tiny $ Q' $};
\end{tikzpicture} \ ,
\end{equation*}
where $ a \in A $. We then have in $ W_{d,\mathbf{Q}}^{\operatorname{aff}}(A) $ that
\begin{equation*}
D_{1}D_{2} = 3 \
\begin{tikzpicture}[centerzero, thick]
         \draw[-] (0.6,-0.7) to[out=up, in=down] (0,7/30) to[out=up, in=down] (0.3,0.7);
         \draw[-] (0,-0.7) to[out=up, in=down] (0.9,0.7);
         \draw[-,wei] (0.9,-0.7) to[out=up, in=down] (0.6,0.7);
         \draw[-,wei] (0.3,-0.7) to[out=up, in=down] (0,-7/30) to[out=up, in=down] (0.3,7/30) to[out=up, in=down] (0,0.7);
         \node at (0.3,-0.9) {\tiny $ Q $};
         \node at (0.9,-0.9) {\tiny $ Q' $};
         \token{0.88,0.5294}{west}{a};
         \singdot{0.8,0.35};
\end{tikzpicture}
, \quad
D_{2}D_{1} = 0.
\end{equation*}
\end{egg}

\begin{rmk} \label{pen}
Suppose that $ \mathbf{Q} \in F(\mathcal{E}_{\operatorname{Pol}_{1}(A)}) $ is the empty word. Then it will follow from the basis theorem (Theorem \ref{HL-basis}) that, in this case, $ W_{d,\mathbf{Q}}^{\operatorname{aff}}(A) $ is isomorphic to the affine wreath product algebra $ W_{d}^{\operatorname{aff}}(A)$ from Definition \ref{pluto2}.
\end{rmk}

We define the following algebra, which is obtained by taking $ A = \Bbbk $ in Definition \ref{HLAWPA}.

\begin{defn} \label{A=K_case}
Let $ d \in \mathbb{N}_{+} $ and $ \mathbf{Q} \in F(\mathcal{E}_{\operatorname{Pol}_{1}(\Bbbk)}) $. Then we define the $ (d,\mathbf{Q}) $-\emph{degenerate affine Hecke algebra} to be $ W_{d,\mathbf{Q}}^{\operatorname{aff}}(\Bbbk) $.
\end{defn}

The $ (d,\mathbf{Q}) $-degenerate affine Hecke algebra should be thought of as a degenerate version of the higher-level affine Hecke algebra, which was defined by Maksimau and Stroppel in \cite[§2]{Maksimau-Stroppel} and by Webster in \cite[§5]{Webster2}. The algebra $ W_{d,\mathbf{Q}}^{\operatorname{aff}}(\Bbbk) $ does not appear to have been explicitly defined before, but it does implicitly appear as a path algebra of the affine Schur category from \cite{Song-Wang2}.

\subsection{A monoidal functor} \label{monoidal_functor}

We have a natural functor $ \mathcal{AW}(A) \rightarrow \mathcal{LAW}(A) $ that sends the generating objects and generating morphisms in $ \mathcal{AW}(A) $ to the ones of the same name in $ \mathcal{LAW}(A) $. The following theorem provides a left inverse for this functor.

\begin{theo} \label{fruit}
There is a strict $ \Bbbk $-linear monoidal functor
\begin{equation*}
\Phi \colon \mathcal{LAW}(A) \rightarrow \mathcal{AW}(A)
\end{equation*}
which sends $ \go $ to $ \go $ and $ Q \in \mathcal{E}_{\operatorname{Pol}_{1}(A)} $ to $ \mathds{1} $ (the unit object of $ \mathcal{AW}(A) $), and sends the generating morphisms as follows:
\begin{gather*}
\Phi \left(\begin{tikzpicture}[centerzero, thick]
      \draw[-] (0.6,-0.4) -- (0.6,0.4);
      \token{0.6,0}{west}{a};
\end{tikzpicture}\right) = \begin{tikzpicture}[centerzero, thick]
      \draw[-] (0.6,-0.4) -- (0.6,0.4);
      \token{0.6,0}{west}{a};
\end{tikzpicture}, 
\quad 
\Phi \left(\begin{tikzpicture}[centerzero, thick]
      \draw[-] (0.6,-0.4) -- (0.6,0.4);
      \singdot{0.6,0};
\end{tikzpicture}\right) = 
\begin{tikzpicture}[centerzero, thick]
      \draw[-] (0.6,-0.4) -- (0.6,0.4);
      \singdot{0.6,0};
\end{tikzpicture} \ ,
\quad 
\Phi \left(\begin{tikzpicture}[centerzero, thick]
      \draw[-] (0.9,-0.4) -- (0.3,0.4);
      \draw[-] (0.3,-0.4) -- (0.9,0.4);
\end{tikzpicture}\right) 
=
\begin{tikzpicture}[centerzero, thick]
      \draw[-] (0.9,-0.4) -- (0.3,0.4);
      \draw[-] (0.3,-0.4) -- (0.9,0.4);
\end{tikzpicture} \ ,
\\ \Phi \left(\begin{tikzpicture}[centerzero, thick]
      \draw[-] (0.9,-0.4) -- (0.3,0.4);
      \draw[-,wei] (0.3,-0.4) -- (0.9,0.4);
      \node at (0.3,-0.6) {\tiny $ Q $};
\end{tikzpicture}\right) 
=
\begin{tikzpicture}[centerzero, thick]
      \draw[-] (0.6,-0.4) -- (0.6,0.4);
\end{tikzpicture} \ ,
\quad 
\Phi \left(\begin{tikzpicture}[centerzero, thick]
      \draw[-] (0.3,-0.4) -- (0.9,0.4);
      \draw[-,wei] (0.9,-0.4) -- (0.3,0.4);
      \node at (0.9,-0.6) {\tiny $ Q $};
\end{tikzpicture}\right) 
= 
\begin{tikzpicture}[centerzero, thick]
        \draw[-] (0,-0.5) -- (0,0.5);
             \pin{0,0}{.7,0}{Q};
\end{tikzpicture} \ ,
\end{gather*}
for all $ a \in A $, $ Q \in \mathcal{E}_{\operatorname{Pol}_{1}(A)} $. 
\end{theo}

\begin{proof}
To prove the existence of $ \Phi $, it suffices to verify the relations \cref{tokrel1}, \cref{braid2}, \cref{affwreath}, and \cref{hw1,hw2,hw3,hw5,green}. The relations \cref{tokrel1}, \cref{braid2} and \cref{affwreath} are immediate, and the relations \cref{hw1,hw2,hw3} are straightforward to verify. The first relation in \cref{hw5} is also straightforward to check. Note that for \cref{hw1} and \cref{hw2}, one must make use of the fact that the elements of $ \mathcal{E}_{\operatorname{Pol}_{1}(A)} $ lie in the even center of $ \operatorname{Pol}_{1}(A) $ (see Definition \ref{multi_pen}). In the following, we check \cref{green} and the second relation in \cref{hw5}. For the rest of the proof, we fix an element $ Q \in \mathcal{E}_{\operatorname{Pol}_{1}(A)} $.
\\ \indent To show the second relation of \cref{hw5}, we compute that
\begin{align*}
\Phi \left(\begin{tikzpicture}[centerzero, thick]
        \draw[wipe] (0,-0.5) to[out=up, in=down] (0.4,0) to[out=up,in=down] (0,0.5);
        \draw[-] (0,-0.5) to[out=up, in=down] (0.4,0) to[out=up,in=down] (0,0.5);
        \draw[-] (-0.4,-0.5) -- (0.4,0.5);
        \draw[-, wei] (0.4,-0.5) -- (-0.4,0.5);
        \node at (0.4,-0.7) {\tiny $ Q $};
\end{tikzpicture}\right)
= 
\begin{tikzpicture}[centerzero, thick]
      \draw[-] (0.9,-0.4) -- (0.3,0.4);
      \draw[-] (0.3,-0.4) -- (0.9,0.4);
      \pin{0.45,-0.2}{-0.2,-0.2}{Q};
      \pin{0.75,-0.2}{1.4,-0.2}{Q};
\end{tikzpicture}
= 
\begin{tikzpicture}[centerzero, thick]
      \draw[-] (0.9,-0.4) -- (0.3,0.4);
      \draw[-] (0.3,-0.4) -- (0.9,0.4);
      \pin{0.45,0.2}{-0.2,0.2}{Q};
      \pin{0.75,0.2}{1.4,0.2}{Q};
\end{tikzpicture}
= 
\Phi \left(\begin{tikzpicture}[centerzero, thick]
        \draw[-] (0,-0.5) to[out=up, in=down] (-0.4,0) to[out=up,in=down] (0,0.5);
        \draw[-] (-0.4,-0.5) -- (0.4,0.5);
        \draw[-, wei] (0.4,-0.5) -- (-0.4,0.5);
        \node at (0.4,-0.7) {\tiny $ Q $};
\end{tikzpicture}\right),
\end{align*}
where the second equality can be seen as follows. Since $ Q \in \mathcal{E}_{\operatorname{Pol}_{1}(A)} $, we have by Lemma \ref{uncharted5678} that the elements $ Q_{1}, Q_{2} \in \operatorname{Pol}_{2}(A) $ lie in the center of $ \operatorname{Pol}_{2}(A) $. Hence their product $ Q_{1}Q_{2} $ also lies in the center of $ \operatorname{Pol}_{2}(A) $. Furthermore, the element $ Q_{1}Q_{2} $ is symmetric, and so Proposition \ref{aff_center_result} yields that $ Q_{1}Q_{2} $ is central in $ W_{2}^{\operatorname{aff}}(A) $. In particular, $ Q_{1}Q_{2} $ and $ \sigma_{1} $ commute.
\\ \indent To show \cref{green}, we compute that
\begin{equation*}
\Phi \left(\begin{tikzpicture}[centerzero, thick]
        \draw[-] (0.4,-0.5) -- (-0.4,0.5);
        \draw[-] (-0.4,-0.5) -- (0.4,0.5);
        \node at (0,-0.7) {\tiny $ Q $};
        \draw[-, wei] (0,-0.5) to[out=up, in=down] (-0.4,0) to[out=up,in=down] (0,0.5);
\end{tikzpicture}\right)
= \begin{tikzpicture}[centerzero, thick]
      \draw[-] (0.9,-0.4) -- (0.3,0.4);
      \draw[-] (0.3,-0.4) -- (0.9,0.4);
      \pin{0.45,-0.2}{-0.2,-0.2}{Q};
\end{tikzpicture}
\ \stackrel{\mathclap{\cref{movein}}}{=} \ \begin{tikzpicture}[centerzero, thick]
      \draw[-] (0.9,-0.4) -- (0.3,0.4);
      \draw[-] (0.3,-0.4) -- (0.9,0.4);
      \pin{0.75,0.2}{1.4,0.2}{Q};
\end{tikzpicture}
- 
\begin{tikzpicture}[centerzero, thick]
      \draw[-] (0.7,-0.4) -- (0.7,0.4);
      \draw[-] (0.3,-0.4) -- (0.3,0.4);
        \pinpin{0.3,0}{0.7,0}{1.8,0}{\partial_{1}(Q_{1})};
\end{tikzpicture} 
\ = \Phi \left(\begin{tikzpicture}[centerzero, thick]
        \draw[-] (2,-0.5) -- (1.2,0.5);
        \draw[-] (1.2,-0.5) -- (2,0.5);
        \node at (1.6,-0.7) {\tiny $ Q $};
        \draw[-, wei] (1.6,-0.5) to[out=up, in=down] (2,0) to[out=up,in=down] (1.6,0.5);
\end{tikzpicture}
\ - \
\begin{tikzpicture}[centerzero, thick]
        \node at (0,-0.7) {\tiny $ Q $};
        \draw[-] (-0.6,-0.5) -- (-0.6,0.5);
        \draw[-, wei] (0,-0.5) -- (0,0.5);
        \draw[-] (0.6,-0.5) -- (0.6,0.5);
        \pinpin{0.6,0}{-0.6,0}{1.8,0}{\partial_{1}(Q_{1})};
\end{tikzpicture}\right). \qedhere
\end{equation*}
\end{proof}

\begin{rmk}
If $ \mathcal{C} $ is any category, and $ \mathcal{F} \colon \mathcal{AW}(A) \rightarrow \mathcal{C} $ is any functor, then composing $ \Phi $ with $ \mathcal{F} $ yields a functor $ \mathcal{LAW}(A) \rightarrow \mathcal{C} $. We will use this idea to obtain a polynomial representation of $ \mathcal{LAW}(A) $ in Corollary \ref{poly}.
\end{rmk}

\begin{rmk}
Using a similar proof to Theorem \ref{fruit}, one can show that there is a strict $ \Bbbk $-linear monoidal functor $ \Phi' \colon \mathcal{LAW}(A) \rightarrow \mathcal{AW}(A) $ which sends $ \go $ to $ \go $ and $ Q \in \mathcal{E}_{\operatorname{Pol}_{1}(A)} $ to $ \mathds{1} $, and sends the generating morphisms as follows:
\begin{gather*}
\Phi' \left(\begin{tikzpicture}[centerzero, thick]
      \draw[-] (0.6,-0.4) -- (0.6,0.4);
      \token{0.6,0}{west}{a};
\end{tikzpicture}\right) = \begin{tikzpicture}[centerzero, thick]
      \draw[-] (0.6,-0.4) -- (0.6,0.4);
      \token{0.6,0}{west}{a};
\end{tikzpicture}, 
\quad 
\Phi' \left(\begin{tikzpicture}[centerzero, thick]
      \draw[-] (0.6,-0.4) -- (0.6,0.4);
      \singdot{0.6,0};
\end{tikzpicture}\right) = 
\begin{tikzpicture}[centerzero, thick]
      \draw[-] (0.6,-0.4) -- (0.6,0.4);
      \singdot{0.6,0};
\end{tikzpicture} \ ,
\quad 
\Phi' \left(\begin{tikzpicture}[centerzero, thick]
      \draw[-] (0.9,-0.4) -- (0.3,0.4);
      \draw[-] (0.3,-0.4) -- (0.9,0.4);
\end{tikzpicture}\right) 
=
\begin{tikzpicture}[centerzero, thick]
      \draw[-] (0.9,-0.4) -- (0.3,0.4);
      \draw[-] (0.3,-0.4) -- (0.9,0.4);
\end{tikzpicture} \ ,
\\ \Phi' \left(\begin{tikzpicture}[centerzero, thick]
      \draw[-] (0.9,-0.4) -- (0.3,0.4);
      \draw[-,wei] (0.3,-0.4) -- (0.9,0.4);
      \node at (0.3,-0.6) {\tiny $ Q $};
\end{tikzpicture}\right) 
=
\begin{tikzpicture}[centerzero, thick]
        \draw[-] (0,-0.5) -- (0,0.5);
             \pin{0,0}{.7,0}{Q};
\end{tikzpicture} \ .
\quad 
\Phi' \left(\begin{tikzpicture}[centerzero, thick]
      \draw[-] (0.3,-0.4) -- (0.9,0.4);
      \draw[-,wei] (0.9,-0.4) -- (0.3,0.4);
      \node at (0.9,-0.6) {\tiny $ Q $};
\end{tikzpicture}\right) 
= 
\begin{tikzpicture}[centerzero, thick]
      \draw[-] (0.6,-0.4) -- (0.6,0.4);
\end{tikzpicture} \ ,
\end{gather*}
for all $ a \in A $, $ Q \in \mathcal{E}_{\operatorname{Pol}_{1}(A)} $. 
\end{rmk}

The functor $ \Phi $ is clearly surjective on objects, and we will obtain in Corollary \ref{herd} that $ \Phi $ is faithful. However, the functor $ \Phi $ is not full (and so is not an equivalence of categories). Indeed, let $ Q $ be any element in $ \mathcal{E}_{\operatorname{Pol}_{1}(A)} $ (for example, one could take $ Q = 1 $). Then it follows from the forms of the generating morphisms of $ \mathcal{LAW}(A) $ in \cref{Gen1} and \cref{Gen2} that $ \operatorname{Hom}_{\mathcal{LAW}(A)}(\mathds{1},Q) = 0 $. Thus the induced map
\begin{equation*}
\operatorname{Hom}_{\mathcal{LAW}(A)}(\mathds{1},Q) \rightarrow \operatorname{End}_{\mathcal{AW}(A)}(\mathds{1}) \cong \Bbbk, \quad r \mapsto \Phi(r)
\end{equation*}
is not surjective.

\subsection{Basis theorem} \label{basisthm}

In this section, we find bases of the morphism spaces of $ \mathcal{LAW}(A) $.

\begin{lem} \label{zeromostly2}
Let $ \mathbf{i}, \mathbf{j} \in F(\widehat{\mathcal{E}}_{\operatorname{Pol}_{1}(A)}) $ be two objects in $ \mathcal{LAW}(A) $. Then we have $ \operatorname{Hom}_{\mathcal{LAW}(A)}(\mathbf{i},\mathbf{j}) = 0 $ unless $ \mathbf{i}, \mathbf{j} \in \Gamma_{d,\mathbf{Q}} $ for some $ d \in \mathbb{N} $, $ \mathbf{Q} \in F(\mathcal{E}_{\operatorname{Pol}_{1}(A)}) $.
\end{lem}

\begin{proof}
This follows from the forms of the generating morphisms of $ \mathcal{LAW}(A) $ in \cref{Gen1} and \cref{Gen2}.
\end{proof}

In light of Lemma \ref{zeromostly2}, for the rest of this section, we fix an integer $ d \in \mathbb{N} $ and a word $ \mathbf{Q} \in F(\mathcal{E}_{\operatorname{Pol}_{1}(A)}) $. We aim to find a basis of $ \operatorname{Hom}_{\mathcal{LAW}(A)}(\mathbf{i}, \mathbf{j}) $ for all $ \mathbf{i}, \mathbf{j} \in \Gamma_{d,\mathbf{Q}} $.

\begin{defn} \label{Generalized_sigma}
For each $ \mathbf{i}, \mathbf{j} \in \Gamma_{d,\mathbf{Q}} $, and each permutation $ w \in S_{d} $, we choose a diagram $ \sigma_{\mathbf{j},w,\mathbf{i}} \in \operatorname{Hom}_{\mathcal{LAW}(A)}(\mathbf{i}, \mathbf{j}) $ with the following properties:
\begin{enumerate}
\item[\textbullet]  The sequence of transpositions on the black strands of $ \sigma_{\mathbf{j},w,\mathbf{i}} $ is a reduced expression for \(w\). 
\item[\textbullet] Each pair of red and black strands cross at most once.
\item[\textbullet] The diagram $ \sigma_{\mathbf{j},w,\mathbf{i}} $ is tokenless and dotless.
\end{enumerate}
\end{defn}

\begin{rmk} \label{greetings}
We note that such a choice of diagram with the above properties is not unique. However, we see no reason to prefer one choice over another.
\end{rmk}

\begin{egg}
Suppose $ d = 3 $ and $ \mathbf{Q} = QQ' $ for some $ Q,Q' \in \mathcal{E}_{\operatorname{Pol}_{1}(A)} $. Let $ \mathbf{i} := \go Q \go Q' \go \in \Gamma_{d,\mathbf{Q}} $ and $ \mathbf{j} := \go Q Q' \go \go \in \Gamma_{d,\mathbf{Q}} $. Set $ w = s_{2}s_{1} \in S_{3} $. Then two possible choices for the element $ \sigma_{\mathbf{j},w,\mathbf{i}} $ are given below:
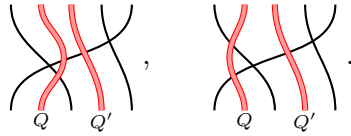
\begin{equation*}
\begin{tikzpicture}[centerzero, thick]
       \draw[-] (0,-0.7) to[out=up, in=down] (1.6,0.7);
       \draw[-] (0.8,-0.7) to[out=up, in=down] (0,0.7);
       \draw[-,wei] (0.4,-0.7) to[out=up, in=down] (0.7,0) to[out=up, in=down] (0.4,0.7);
       \draw[-,wei] (1.2,-0.7) to[out=up, in=down] (0.8,0.7);
       \draw[-] (1.6,-0.7) to[out=up, in=down] (1.2,0.7);
       \node at (0.4,-0.85) {\tiny $ Q $};
       \node at (1.2,-0.85) {\tiny $ Q' $};
\end{tikzpicture} \ , 
\qquad
\begin{tikzpicture}[centerzero, thick]
       \draw[-] (0,-0.7) to[out=up, in=down] (1.6,0.7);
       \draw[-] (0.8,-0.7) to[out=up, in=down] (0,0.7);
       \draw[-,wei] (0.4,-0.7) to[out=up, in=down] (0.2,0) to[out=up, in=down] (0.4,0.7);
       \draw[-,wei] (1.2,-0.7) to[out=up, in=down] (0.8,0.7);
       \draw[-] (1.6,-0.7) to[out=up, in=down] (1.2,0.7);
       \node at (0.4,-0.85) {\tiny $ Q $};
       \node at (1.2,-0.85) {\tiny $ Q' $};
\end{tikzpicture} \ .
\end{equation*}
\end{egg}

Given an element $ \mathbf{i} \in \Gamma_{d,\mathbf{Q}} $, we have an algebra homomorphism 
\begin{equation} \label{secure2}
A^{\otimes d} \rightarrow \operatorname{End}_{\mathcal{LAW}(A)}(\mathbf{i})
\end{equation}
that sends $ a_{i} $ ($ a \in A $, $ 1 \leq i \leq d $) to the token labelled \(a\) on the \(i\)-th black strand. If $ \mathbf{a} \in A^{\otimes d} $, then we define $ \mathbf{a}_{\mathbf{i}} $ to be the image of $ \mathbf{a} $ under the homomorphism of \cref{secure2}. Also, for $ 1 \leq i \leq d $, we set $ x_{i,\mathbf{i}} \in \operatorname{End}_{\mathcal{LAW}(A)}(\mathbf{i}) $ to be a dot on the $ i $-th black strand. Then if $ \alpha = (\alpha_{1},\ldots,\alpha_{d}) \in \mathbb{N}^{d} $, we define 
\begin{equation}
\mathbf{x}_{\mathbf{i}}^{\alpha} := (x_{1,\mathbf{i}})^{\alpha_{1}}\cdots (x_{d,\mathbf{i}})^{\alpha_{d}} \in \operatorname{End}_{\mathcal{LAW}(A)}(\mathbf{i}).
\end{equation}
We now define the following subset of $ \operatorname{Hom}_{\mathcal{LAW}(A)}(\mathbf{i}, \mathbf{j}) $:
\begin{equation}
_{\mathbf{j}} \mathcal{B}_{\mathbf{i}} := \{\mathbf{a}_{\mathbf{j}}\mathbf{x}_{\mathbf{j}}^{\alpha}\sigma_{\mathbf{j},w,\mathbf{i}} : \mathbf{a} \in B^{\otimes d}, \ \alpha \in \mathbb{N}^{d}, \ w \in S_{d} \}.
\end{equation}

\begin{lem} \label{k-span2}
Let $ \mathbf{i}, \mathbf{j} \in \Gamma_{d,\mathbf{Q}} $. Then the set $ _{\mathbf{j}} \mathcal{B}_{\mathbf{i}} $ spans $ \operatorname{Hom}_{\mathcal{LAW}(A)}(\mathbf{i}, \mathbf{j}) $.
\end{lem}

\begin{proof}
We prove this by induction on the total number of crossings. If \(D\) is a diagram with no crossings, then \(D\) contains only dots and tokens, from which it clearly follows that \(D\) lies in the span of $ _{\mathbf{j}} \mathcal{B}_{\mathbf{i}} $.
\\ \indent Now let $ k > 0 $, and let \(D\) be a diagram with \(k\) crossings. First, by using the relations \cref{braid2}, \cref{tokencross2}, \cref{affwreath}, \cref{affwreath223}, \cref{hw1}, and \cref{hw2}, one can move the dots and tokens to the top of \(D\) modulo terms with fewer crossings, which all lie in the span of $ _{\mathbf{j}} \mathcal{B}_{\mathbf{i}} $ by the induction hypothesis. Now, if there are two strands in \(D\) that cross more than once, then \(D\) can be written as a $ \Bbbk $-linear combination of diagrams with fewer than $ k $ crossings. The proof of this fact is analogous to the proof of \cite[Lem.~4.10(3)]{Webster}. So we may from here on assume that any two strands in \(D\) cross at most once. In this case, the sequence of black strands on \(D\) corresponds to some reduced expression for some element $ w \in S_{d} $. We can now apply the relations \cref{hw5,green} and the second relation in \cref{braid2} to get from \(D\) to $ \sigma_{\mathbf{j},w,\mathbf{i}} $. Note again that this process may create additional terms with fewer crossings (due to \cref{green}), but these terms all lie in the span of $ _{\mathbf{j}} \mathcal{B}_{\mathbf{i}} $ by the induction hypothesis. Thus we have that \(D\) lies in the span of $ _{\mathbf{j}} \mathcal{B}_{\mathbf{i}} $.
\end{proof}

Given elements $ \mathbf{i}, \mathbf{j} \in \Gamma_{d,\mathbf{Q}} $, the functor $ \Phi $ from Theorem \ref{fruit} restricts to a $ \Bbbk $-linear map 
\begin{equation*}
\Phi \colon \operatorname{Hom}_{\mathcal{LAW}(A)}(\mathbf{i}, \mathbf{j}) \rightarrow W_{d}^{\operatorname{aff}}(A).
\end{equation*}

\begin{lem} \label{cactus2}
Let $ \mathbf{i}, \mathbf{j} \in \Gamma_{d,\mathbf{Q}} $ and $ w \in S_{d} $. Then we have
\begin{equation} \label{cactus}
\Phi(\sigma_{\mathbf{j},w,\mathbf{i}}) = h_{w,w}\sigma_{w} + \sum_{\substack{u \in S_{d} \\ L(u) < L(w)}} h_{u,w}\sigma_{u}
\end{equation}
for some $ h_{u,w} \in \operatorname{Pol}_{d}(A) $, where $ h_{w,w} $ is regular in $ \operatorname{Pol}_{d}(A) $. Here, \(L\) is the length function on $ S_{d} $.
\end{lem}

\begin{proof}
Using the definition of $ \Phi $ from Theorem \ref{fruit}, we see that $ \Phi(\sigma_{\mathbf{j},w,\mathbf{i}}) $ is a diagram whose sequence of transpositions on the black strands is a reduced expression for \(w\), where the strands carry pins, and each such pin is labelled by an element of $ \mathcal{E}_{\operatorname{Pol}_{1}(A)} $. For example, if $ Q,Q' \in \mathcal{E}_{\operatorname{Pol}_{1}(A)} $, then
\begin{equation*}
\Phi\left(\begin{tikzpicture}[centerzero, thick]
       \draw[-] (0,-0.7) to[out=up, in=down] (1.6,0.7);
       \draw[-] (0.8,-0.7) to[out=up, in=down] (0,0.7);
       \draw[-,wei] (0.4,-0.7) to[out=up, in=down] (0.2,0) to[out=up, in=down] (0.4,0.7);
       \draw[-,wei] (1.2,-0.7) to[out=up, in=down] (0.8,0.7);
       \draw[-] (1.6,-0.7) to[out=up, in=down] (1.2,0.7);
       \node at (0.4,-0.85) {\tiny $ Q $};
       \node at (1.2,-0.85) {\tiny $ Q' $};
\end{tikzpicture}\right)
=
\begin{tikzpicture}[centerzero, thick]
       \draw[-] (0,-0.7) to[out=up, in=down] (1.6,0.7);
       \draw[-] (0.8,-0.7) to[out=up, in=down] (0,0.7);
       \draw[-] (1.6,-0.7) to[out=up, in=down] (1.2,0.7);
       \pin{0.2,-0.28}{-0.6,-0.28}{Q};
       \pin{1,0.1}{0.6,0.8}{Q'};
\end{tikzpicture} \ .
\end{equation*}
Next, we inductively apply \cref{movein} to slide the pins to the top of the diagram, at the cost of adding a linear combination of diagrams with fewer than $ L(w) $ crossings. It then follows that $ \Phi(\sigma_{\mathbf{j},w,\mathbf{i}}) $ is equal to an expansion of the form given in \cref{cactus}. Note that by this process, we obtain that $ h_{w,w} $ is some composition of pins, each labelled by an element of $ \mathcal{E}_{\operatorname{Pol}_{1}(A)} $. That is, $ h_{w,w} = D_{(1)}\cdots D_{(m)} $, where each $ D_{(k)} \in \operatorname{Pol}_{d}(A) $ is of the form
\begin{equation*}
D_{(k)} = 
\begin{tikzpicture}[centerzero, thick]
        \draw[-] (-2.1,-0.5) -- (-2.1,0.5);
        \node at (-1.5,0) {$ \cdots $};
        \draw[-] (-0.9,-0.5) -- (-0.9,0.5);
        \draw[-] (0.3,-0.5) -- (0.3,0.5);
        \draw[-] (0.8,-0.5) -- (0.8,0.5);
        \node at (1.4,0) {$ \cdots $};
        \pin{0.3,0}{-0.5,0}{Q};
        \draw[-] (2,-0.5) -- (2,0.5);
\end{tikzpicture} 
\end{equation*}
for some $ Q \in \mathcal{E}_{\operatorname{Pol}_{1}(A)} $. By Lemma \ref{uncharted5678}, each $ D_{(k)} $ is regular. Thus, since $ h_{w,w} = D_{(1)}\cdots D_{(m)} $, we obtain that $ h_{w,w} $ is also regular.
\end{proof}

\begin{theo} \label{HL-basis}
Let $ \mathbf{i}, \mathbf{j} \in \Gamma_{d,\mathbf{Q}} $. Then the set $ _{\mathbf{j}} \mathcal{B}_{\mathbf{i}} $ is a $ \Bbbk $-basis of $ \operatorname{Hom}_{\mathcal{LAW}(A)}(\mathbf{i}, \mathbf{j}) $.
\end{theo}

\begin{proof}
We have that $ \operatorname{Hom}_{\mathcal{LAW}(A)}(\mathbf{i}, \mathbf{j}) $ is a left $ \operatorname{Pol}_{d}(A) $-module, with action given by 
\begin{equation*}
\mathbf{a} \cdot y = \mathbf{a}_{\mathbf{j}}y, \quad x_{i} \cdot y = x_{i,\mathbf{j}}y,
\end{equation*}
for all $ \mathbf{a} \in A^{\otimes d} $, $ 1 \leq i \leq d $, $ y \in  \operatorname{Hom}_{\mathcal{LAW}(A)}(\mathbf{i}, \mathbf{j}) $. Then in order to prove the theorem, it suffices to show that $ \operatorname{Hom}_{\mathcal{LAW}(A)}(\mathbf{i}, \mathbf{j}) $ is free as a left $ \operatorname{Pol}_{d}(A) $-module, with basis $ \{\sigma_{\mathbf{j},w,\mathbf{i}} : w \in S_{d}\} $. By Lemma~\ref{k-span2}, the set $ \{\sigma_{\mathbf{j},w,\mathbf{i}} : w \in S_{d}\} $ spans $ \operatorname{Hom}_{\mathcal{LAW}(A)}(\mathbf{i}, \mathbf{j}) $ (as a left $ \operatorname{Pol}_{d}(A) $-module), and so it remains to show linear independence. Thus we assume that we have a nontrivial linear dependence equation
\begin{equation} \label{depequation3975}
\sum_{w \in S_{d}} f_{w}\sigma_{\mathbf{j},w,\mathbf{i}} = 0,
\end{equation}
where $ f_{w} \in \operatorname{Pol}_{d}(A) $. Choose an element $ v \in S_{d} $ of maximal length such that $ f_{v} \neq 0 $. Then \cref{depequation3975} becomes
\begin{equation} \label{depequation3976}
\sum_{\substack{w \in S_{d} \\ L(w) \leq L(v)}} f_{w}\sigma_{\mathbf{j},w,\mathbf{i}} = 0.
\end{equation}
Then computing the image of \cref{depequation3976} under $ \Phi $, and using \cref{cactus}, we have in $ W_{d}^{\operatorname{aff}}(A) $ that
\begin{equation} \label{nook3975}
\sum_{\substack{w \in S_{d} \\ L(w) \leq L(v)}} f_{w}\left( h_{w,w}\sigma_{w} + \sum_{\substack{u \in S_{d} \\ L(u) < L(w)}} h_{u,w}\sigma_{u}\right) = 0.
\end{equation}
Then \cref{nook3975} and Proposition \ref{midnight2} imply that $ f_{v}h_{v,v} = 0 $. Since $ h_{v,v} $ is regular (see Lemma \ref{cactus2}), we obtain that $ f_{v} = 0 $, which is a contradiction. The result now follows.
\end{proof}

\begin{cor} \label{herd}
The functor $ \Phi $ is faithful.
\end{cor}

\begin{proof}
This follows from the proof of Theorem \ref{HL-basis}.
\end{proof}

Given an integer $ n \in \mathbb{N}_{+} $ and a $ \Bbbk $-algebra \(F\), we define $ M_{n}(F) $ to be the algebra of $ n \times n $ matrices with entries in \(F\).

\begin{cor} \label{apple_tree}
Let $ n := |\Gamma_{d,\mathbf{Q}}| $. Then $ W_{d,\mathbf{Q}}^{\operatorname{aff}}(A) $ is isomorphic to a subalgebra of $ M_{n}(W_{d}^{\operatorname{aff}}(A)) $.
\end{cor}

\begin{proof}
The faithful functor $ \Phi \colon \mathcal{LAW}(A) \rightarrow \mathcal{AW}(A) $ induces an injective algebra homomorphism
\begin{equation*}
W_{d,\mathbf{Q}}^{\operatorname{aff}}(A) \quad \stackrel{\mathclap{\cref{HLAWPA2}}}{=} \quad \operatorname{End}_{\operatorname{Add}(\mathcal{LAW}(A))}\left(\bigoplus_{\mathbf{i} \in \Gamma_{d,\mathbf{Q}}}\mathbf{i}\right) \rightarrow \operatorname{End}_{\operatorname{Add}(\mathcal{AW}(A))}\left(\bigoplus_{i=1}^{n} \go^{\otimes d}\right) = M_{n}(W_{d}^{\operatorname{aff}}(A)).
\end{equation*}
The result now follows.
\end{proof}

\subsection{A polynomial representation} \label{poly_representation}

In this section, we give a polynomial representation of $ \mathcal{LAW}(A) $. We start by giving a polynomial representation of $ \mathcal{AW}(A) $. We define $ \Bbbk\text{-smod} $ to be the category of $ \Bbbk $-supermodules and $ \Bbbk $-linear maps (which need not be parity-preserving). In this subsection, we identify $ \operatorname{Pol}_{n}(A) $ with $ \operatorname{Pol}_{1}(A)^{\otimes n} $ as in Lemma \ref{identification}. 

\begin{prop} \label{Polyrepp-1}
We have a $ \Bbbk $-linear monoidal functor 
\begin{equation*}
P \colon \mathcal{AW}(A) \rightarrow \Bbbk\text{-smod}
\end{equation*}
which sends $ \go $ to $ \operatorname{Pol}_{1}(A) $, and sends the generating morphisms as follows:
\begin{align*}
P \left(\begin{tikzpicture}[centerzero, thick]
      \draw[-] (0.6,-0.4) -- (0.6,0.4);
      \token{0.6,0}{west}{a};
\end{tikzpicture}\right) &\colon \operatorname{Pol}_{1}(A) \rightarrow \operatorname{Pol}_{1}(A), \quad f \mapsto af, 
\\ P \left(\begin{tikzpicture}[centerzero, thick]
      \draw[-] (0.6,-0.4) -- (0.6,0.4);
      \singdot{0.6,0};
\end{tikzpicture}\right) &\colon \operatorname{Pol}_{1}(A) \rightarrow \operatorname{Pol}_{1}(A), \quad f \mapsto xf,
\\ P \left(\begin{tikzpicture}[centerzero, thick]
      \draw[-] (0.9,-0.4) -- (0.3,0.4);
      \draw[-] (0.3,-0.4) -- (0.9,0.4);
\end{tikzpicture}\right) &\colon \operatorname{Pol}_{2}(A) \rightarrow \operatorname{Pol}_{2}(A), \quad f \mapsto s_{1}(f) - \partial_{1}(f), 
\end{align*}
for all $ a \in A $. Furthermore, the functor \(P\) is faithful.
\end{prop}

\begin{proof}
The proof uses standard techniques, and so we only sketch the argument here. The existence of $ P $ can be shown by directly verifying the relations \cref{tokrel1}, \cref{braid2}, and \cref{affwreath}. To prove that \(P\) is faithful, it suffices to show that the induced algebra homomorphism
\begin{equation*}
W_{n}^{\operatorname{aff}}(A) \rightarrow \operatorname{End}_{\Bbbk}(\operatorname{Pol}_{n}(A)), \quad z \mapsto P(z) 
\end{equation*}
is injective for all $ n \in \mathbb{N}_{+} $. Here, $ \operatorname{End}_{\Bbbk}(\operatorname{Pol}_{n}(A)) $ is the set of all $ \Bbbk $-linear maps $ \operatorname{Pol}_{n}(A) \rightarrow \operatorname{Pol}_{n}(A) $. Thus we fix an element $ n \in \mathbb{N}_{+} $, and we assume that $ P(z) = 0 $ for some $ z \in W_{n}^{\operatorname{aff}}(A) $. Then, as in the proof of \cite[Thm.~3.2.2]{Kleshchev}, one can consider $ P(z)(x_{1}^{N}x_{2}^{2N}\cdots x_{n}^{nN}) $ for $ N \gg 0 $ to obtain that $ z = 0 $.
\end{proof}

\begin{cor} \label{Polyrepp}
Let $ n \in \mathbb{N}_{+} $. Then we have an action of $ W_{n}^{\operatorname{aff}}(A) $ on $ \operatorname{Pol}_{n}(A) $, given by 
\begin{equation} \label{Polyrepp2}
g \cdot f = gf, \quad \sigma_{i} \cdot f = s_{i}(f) - \partial_{i}(f),
\end{equation}
for all $ g,f \in \operatorname{Pol}_{n}(A) $ and $ 1 \leq i \leq n-1 $. Furthermore, this action is faithful.
\end{cor}

\begin{proof}
This follows from Proposition \ref{Polyrepp-1}. 
\end{proof}

\begin{cor} \label{poly}
We have a $ \Bbbk $-linear monoidal functor 
\begin{equation*}
P' \colon \mathcal{LAW}(A) \rightarrow \Bbbk\text{-smod}
\end{equation*}
which sends $ \go $ to $ \operatorname{Pol}_{1}(A) $ and $ Q \in \mathcal{E}_{\operatorname{Pol}_{1}(A)} $ to $ \Bbbk $, and sends the generating morphisms as follows:
\begin{align*}
P'\left(\begin{tikzpicture}[centerzero, thick]
      \draw[-] (0.6,-0.4) -- (0.6,0.4);
      \token{0.6,0}{west}{a};
\end{tikzpicture}\right) &\colon \operatorname{Pol}_{1}(A) \rightarrow \operatorname{Pol}_{1}(A), \quad f \mapsto af, 
\\ P' \left(\begin{tikzpicture}[centerzero, thick]
      \draw[-] (0.6,-0.4) -- (0.6,0.4);
      \singdot{0.6,0};
\end{tikzpicture}\right) &\colon \operatorname{Pol}_{1}(A) \rightarrow \operatorname{Pol}_{1}(A), \quad f \mapsto xf,
\\ P' \left(\begin{tikzpicture}[centerzero, thick]
      \draw[-] (0.9,-0.4) -- (0.3,0.4);
      \draw[-] (0.3,-0.4) -- (0.9,0.4);
\end{tikzpicture}\right) &\colon \operatorname{Pol}_{2}(A) \rightarrow \operatorname{Pol}_{2}(A), \quad f \mapsto s_{1}(f) - \partial_{1}(f), 
\\ P' \left(\begin{tikzpicture}[centerzero, thick]
      \draw[-] (0.9,-0.4) -- (0.3,0.4);
      \draw[-,wei] (0.3,-0.4) -- (0.9,0.4);
      \node at (0.3,-0.6) {\tiny $ Q $};
\end{tikzpicture}\right) &\colon \operatorname{Pol}_{1}(A) \rightarrow \operatorname{Pol}_{1}(A), \quad f \mapsto f, 
\\ P' \left(\begin{tikzpicture}[centerzero, thick]
      \draw[-] (0.3,-0.4) -- (0.9,0.4);
      \draw[-,wei] (0.9,-0.4) -- (0.3,0.4);
      \node at (0.9,-0.6) {\tiny $ Q $};
\end{tikzpicture}\right) &\colon \operatorname{Pol}_{1}(A) \rightarrow \operatorname{Pol}_{1}(A), \quad f \mapsto Qf,
\end{align*}
for all $ a \in A $, $ Q \in \mathcal{E}_{\operatorname{Pol}_{1}(A)} $. Furthermore, the functor $ P' $ is faithful.
\end{cor}

\begin{proof}
It is straightforward to see that $ P' $ is equal to the composition of $ \Phi $ and $ P $, where the functor $ \Phi \colon \mathcal{LAW}(A) \rightarrow \mathcal{AW}(A) $ is given in Theorem \ref{fruit}, and $ P \colon \mathcal{AW}(A) \rightarrow \Bbbk\text{-smod} $ is the functor of Proposition \ref{Polyrepp-1}. The fact that $ P' $ is faithful then immediately follows from the fact that both $ \Phi $ and $ P $ are faithful; see Corollary \ref{herd} and Proposition \ref{Polyrepp-1}.
\end{proof}

\subsection{Description of the center} \label{h-affwreathcenter}

In this section, we fix a positive integer $ d \in \mathbb{N}_{+} $ and a word $ \mathbf{Q} \in F(\mathcal{E}_{\operatorname{Pol}_{1}(A)}) $. We aim to find the center of the algebra $ W_{d,\mathbf{Q}}^{\operatorname{aff}}(A) $, in the sense of \cref{centerdef}. For any elements $ \mathbf{i}, \mathbf{j} \in \Gamma_{d,\mathbf{Q}} $, we define
\begin{equation*}
_{\mathbf{i}}W_{d,\mathbf{Q}}^{\operatorname{aff}}(A)_{\mathbf{j}} = 1_{\mathbf{i}}W_{d,\mathbf{Q}}^{\operatorname{aff}}(A)1_{\mathbf{j}},
\end{equation*}
where $ 1_{\mathbf{i}} \in \operatorname{End}_{\mathcal{LAW}(A)}(\mathbf{i}) $ and $ 1_{\mathbf{j}} \in \operatorname{End}_{\mathcal{LAW}(A)}(\mathbf{j}) $ are the identity morphisms. Define $ \boldsymbol{\omega} := \mathbf{Q}\go^{d} \in \Gamma_{d,\mathbf{Q}} $ to be the concatenation of the words $ \mathbf{Q} $ and $ \go^{d} $.

\begin{lem} \label{flour}
We have an isomorphism of algebras $ W_{d}^{\operatorname{aff}}(A) \cong {_{\boldsymbol{\omega}}}W_{d,\mathbf{Q}}^{\operatorname{aff}}(A)_{\boldsymbol{\omega}} $.
\end{lem}

\begin{proof}
There is an obvious algebra homomorphism $ W_{d}^{\operatorname{aff}}(A) \rightarrow {_{\boldsymbol{\omega}}}W_{d,\mathbf{Q}}^{\operatorname{aff}}(A)_{\boldsymbol{\omega}} $ that adds $ |\mathbf{Q}| $ red strands to the left of each diagram in $ W_{d}^{\operatorname{aff}}(A) $. It follows from Proposition \ref{midnight2} and Theorem \ref{HL-basis} that this map is an isomorphism.
\end{proof}

Given $ \mathbf{a} \in A^{\otimes d} $, $ \mathbf{i} \in \Gamma_{d,\mathbf{Q}} $, and $ 1 \leq i \leq d $, recall that we have the elements $ \mathbf{a}_{\mathbf{i}}, x_{i,\mathbf{i}} \in W_{d,\mathbf{Q}}^{\operatorname{aff}}(A) $ (see Section \ref{basisthm}). Then, by Theorem \ref{HL-basis}, we have an injective algebra homomorphism $ \operatorname{Pol}_{d}(A) \rightarrow W_{d,\mathbf{Q}}^{\operatorname{aff}}(A) $ given by
\begin{equation} \label{lifeguard}
\mathbf{a} \mapsto \sum_{\mathbf{i} \in \Gamma_{d,\mathbf{Q}}} \mathbf{a}_{\mathbf{i}}, \qquad x_{i} \mapsto \sum_{\mathbf{i} \in \Gamma_{d,\mathbf{Q}}} x_{i,\mathbf{i}}.
\end{equation}
For the rest of this section, we view $ \operatorname{Pol}_{d}(A) $ as a subalgebra of $ W_{d,\mathbf{Q}}^{\operatorname{aff}}(A) $ under this algebra homomorphism. We also view $ Z(\operatorname{Pol}_{d}(A))^{S_{d}} $ and $ Z(\operatorname{Pol}_{d}(A)) $ as subalgebras of $ W_{d,\mathbf{Q}}^{\operatorname{aff}}(A) $ under the inclusions 
\begin{equation}
Z(\operatorname{Pol}_{d}(A))^{S_{d}} \subseteq Z(\operatorname{Pol}_{d}(A)) \subseteq \operatorname{Pol}_{d}(A) \subseteq W_{d,\mathbf{Q}}^{\operatorname{aff}}(A).
\end{equation}

\begin{prop} \label{higher_center}
The center of $ W_{d,\mathbf{Q}}^{\operatorname{aff}}(A) $ is equal to $ Z(\operatorname{Pol}_{d}(A))^{S_{d}} $.
\end{prop}

\begin{proof}
Let $ z \in Z(\operatorname{Pol}_{d}(A))^{S_{d}} $. Then in order to show that \(z\) lies in the center of $ W_{d,\mathbf{Q}}^{\operatorname{aff}}(A) $, it suffices to show that \(z\) commutes with the identity morphisms $ 1_{\mathbf{i}} $, the elements of $ \operatorname{Pol}_{d}(A) $, the red-black crossings, and the black-black crossings (since $ W_{d,\mathbf{Q}}^{\operatorname{aff}}(A) $ is generated by these elements). It is clear that \(z\) commutes with the identity morphisms $ 1_{\mathbf{i}} $ and the elements of $ \operatorname{Pol}_{d}(A) $. Next, we have by the relations \cref{hw1} and \cref{hw2} that \(z\) commutes with red-black crossings. Finally, \(z\) commutes with black-black crossings by Proposition \ref{aff_center_result}.
\\ \indent Conversely, let $ z \in Z(W_{d,\mathbf{Q}}^{\operatorname{aff}}(A)) $. Then $ z1_{\boldsymbol{\omega}} \in Z({_{\boldsymbol{\omega}}}W_{d,\mathbf{Q}}^{\operatorname{aff}}(A)_{\boldsymbol{\omega}}) $, and so since the isomorphism of Lemma \ref{flour} restricts to an isomorphism between $ Z(W_{d}^{\operatorname{aff}}(A)) $ and $ Z({_{\boldsymbol{\omega}}}W_{d,\mathbf{Q}}^{\operatorname{aff}}(A)_{\boldsymbol{\omega}}) $, we obtain by Proposition \ref{aff_center_result} that $ z1_{\boldsymbol{\omega}} = f1_{\boldsymbol{\omega}} $ for some $ f \in Z(\operatorname{Pol}_{d}(A))^{S_{d}} $. Let $ \mathbf{i} \in \Gamma_{d,\mathbf{Q}} $, and set $ T = \sigma_{\mathbf{i},1,\boldsymbol{\omega}} $ (see Definition \ref{Generalized_sigma}). Then
\begin{equation*}
z1_{\mathbf{i}}T = Tz1_{\boldsymbol{\omega}} = Tf1_{\boldsymbol{\omega}} = f1_{\mathbf{i}}T.
\end{equation*}
This implies that $ z1_{\mathbf{i}} = f1_{\mathbf{i}} $, since the map $ {_{\mathbf{i}}}W_{d,\mathbf{Q}}^{\operatorname{aff}}(A)_{\mathbf{i}} \rightarrow {_{\mathbf{i}}}W_{d,\mathbf{Q}}^{\operatorname{aff}}(A)_{\boldsymbol{\omega}} $, $ y \mapsto yT $ is injective by Theorem \ref{HL-basis}. Therefore,
\begin{equation*}
z = \sum_{\mathbf{i} \in \Gamma_{d,\mathbf{Q}}} z1_{\mathbf{i}} = \sum_{\mathbf{i} \in \Gamma_{d,\mathbf{Q}}} f1_{\mathbf{i}} = f \in Z(\operatorname{Pol}_{d}(A))^{S_{d}}. \qedhere
\end{equation*}
\end{proof}

\begin{rmk}
Let $ \mathbf{Q}, \mathbf{Q}' \in F(\mathcal{E}_{\operatorname{Pol}_{1}(A)}) $. Then, by Proposition \ref{higher_center}, we have an isomorphism of algebras $ Z(W_{d,\mathbf{Q}}^{\operatorname{aff}}(A)) \cong Z(W_{d,\mathbf{Q}'}^{\operatorname{aff}}(A)) $.
\end{rmk}

\subsection{Independence of the trace map} \label{trace_independent}

We have by \cite[Lem.~3.1.3]{Mendonca} that, up to isomorphism, the affine wreath product category $ \mathcal{AW}(A) $ and the affine wreath product algebra $ W_{n}^{\operatorname{aff}}(A) $ depend only on the underlying superalgebra \(A\), and not on the choice of trace map. This result is also given in \cite[Lem.~3.2]{Savage} in the case when one restricts to even trace maps. In this section, we find an analogue of these results in our higher-level setting.
\\ \indent Recall that, throughout this paper, \(A\) is a Frobenius superalgebra with trace map $ \operatorname{tr} \colon A \rightarrow \Bbbk $ of parity $ \varepsilon $ and Nakayama automorphism $ \psi \colon A \rightarrow A $. Let $ \operatorname{tr}' \colon A \rightarrow \Bbbk $ be another homogeneous trace map. We denote the parity of $ \operatorname{tr}' $ by $ \varepsilon' $, and we denote the Nakayama automorphism associated to $ \operatorname{tr}' $ by $ \psi' \colon A \rightarrow A $. We set \(A'\) to be the Frobenius algebra which is obtained by equipping \(A\) with $ \operatorname{tr}' $ instead of $ \operatorname{tr} $. The trace maps $ \operatorname{tr} $ and $ \operatorname{tr}' $ are related as follows.

\begin{lem} \label{twotrace232}
There exists a homogeneous invertible $ u \in A $ of parity $ \overline{u} = \varepsilon + \varepsilon' $ such that $ \operatorname{tr}'(a) = \operatorname{tr}(au) $ for all $ a \in A $. Furthermore, we have $ \psi'(a) = (-1)^{\bar{u}\bar{a}}u\psi(a)u^{-1} $ for all $ a \in A $.
\end{lem}

\begin{proof}
The first statement is given in \cite[Prop.~4.7]{Pike-Savage}. For the second statement, we compute for all $ a, b \in A $ that
\begin{equation*}
\operatorname{tr}'(ab) = \operatorname{tr}(abu) = (-1)^{(\bar{b}+\bar{u})\bar{a}}\operatorname{tr}(bu\psi(a)) = (-1)^{\bar{b}\bar{a}}\operatorname{tr}'(b(-1)^{\bar{u}\bar{a}}u\psi(a)u^{-1}), 
\end{equation*}
and so the result follows.
\end{proof}

For the rest of this section, the element $ u \in A $ is as in Lemma \ref{twotrace232}. Recall that we have the categories $ \mathpzc{Pol}(A) $ and $ \mathpzc{Pol}(A') $ from Definition \ref{summer}. These categories are related as follows.

\begin{lem} \label{radiant}
We have an isomorphism of monoidal categories $ \Psi \colon  \mathpzc{Pol}(A) \rightarrow \mathpzc{Pol}(A') $ sending the generating object $ \go $ to $ \go $, and sending the generating morphisms as follows:
\begin{equation*}
\begin{tikzpicture}[centerzero, thick]
      \draw[-] (0.6,-0.4) -- (0.6,0.4);
      \token{0.6,0}{west}{a};
\end{tikzpicture}
\mapsto 
\begin{tikzpicture}[centerzero, thick]
      \draw[-] (0.6,-0.4) -- (0.6,0.4);
      \token{0.6,0}{west}{a};
\end{tikzpicture}, 
\quad 
\begin{tikzpicture}[centerzero, thick]
      \draw[-] (0.6,-0.4) -- (0.6,0.4);
      \singdot{0.6,0};
\end{tikzpicture}
\mapsto (-1)^{\bar{u}} \ 
\begin{tikzpicture}[centerzero, thick]
      \draw[-] (0.6,-0.4) -- (0.6,0.4);
      \singdot{0.6,0.15};
      \token{0.6,-0.15}{west}{u};
\end{tikzpicture} \ , \qquad a \in A.
\end{equation*}
Its inverse $ \Psi' \colon  \mathpzc{Pol}(A') \rightarrow \mathpzc{Pol}(A) $ sends $ \go $ to $ \go $, and sends the generating morphisms as follows:
\begin{equation*}
\begin{tikzpicture}[centerzero, thick]
      \draw[-] (0.6,-0.4) -- (0.6,0.4);
      \token{0.6,0}{west}{a};
\end{tikzpicture}
\mapsto 
\begin{tikzpicture}[centerzero, thick]
      \draw[-] (0.6,-0.4) -- (0.6,0.4);
      \token{0.6,0}{west}{a};
\end{tikzpicture}, 
\quad 
\begin{tikzpicture}[centerzero, thick]
      \draw[-] (0.6,-0.4) -- (0.6,0.4);
      \singdot{0.6,0};
\end{tikzpicture}
\mapsto (-1)^{\bar{u}} \ 
\begin{tikzpicture}[centerzero, thick]
      \draw[-] (0.6,-0.4) -- (0.6,0.4);
      \singdot{0.6,0.15};
      \token{0.6,-0.15}{west}{u^{-1}};
\end{tikzpicture} \ , \qquad a \in A.
\end{equation*}
\end{lem}

\begin{proof}
To prove the existence of $ \Psi $, it suffices to check the relations \cref{tokrel1} and \cref{fox}. The relation \cref{tokrel1} is immediate, and one can show that \cref{fox} holds by using Lemma \ref{twotrace232}. The existence of $ \Psi' $ then follows by symmetry. Finally, it is straightforward to check that $ \Psi $ and $ \Psi' $ are mutual inverses.
\details{Let $ a \in A $. Then using Lemma \ref{twotrace232}, we compute that
\begin{equation} \label{dotproof}
\Psi\left(\begin{tikzpicture}[centerzero, thick]
      \draw[-] (0.6,-0.4) -- (0.6,0.4);
      \singdot{0.6,-0.15};
      \token{0.6,0.15}{west}{a};
\end{tikzpicture}\right)
=
(-1)^{\bar{u}} \
\begin{tikzpicture}[centerzero, thick]
      \draw[-] (0.6,-0.4) -- (0.6,0.4);
      \singdot{0.6,0};
      \token{0.6,0.25}{west}{a};
      \token{0.6,-0.25}{west}{u};
\end{tikzpicture}
=
(-1)^{\bar{u}+\varepsilon' \bar{a}} \
\begin{tikzpicture}[centerzero, thick]
      \draw[-] (0.6,-0.4) -- (0.6,0.4);
      \singdot{0.6,0.15};
      \token{0.6,-0.15}{west}{\psi'(a)u};
\end{tikzpicture}
=
(-1)^{\bar{u}+\varepsilon \bar{a}} \
\begin{tikzpicture}[centerzero, thick]
      \draw[-] (0.6,-0.4) -- (0.6,0.4);
      \singdot{0.6,0.15};
      \token{0.6,-0.15}{west}{u\psi(a)};
\end{tikzpicture}
= 
\Psi\left( (-1)^{\varepsilon \bar{a}} \ 
\begin{tikzpicture}[centerzero, thick]
      \draw[-] (0.6,-0.4) -- (0.6,0.4);
      \singdot{0.6,0.15};
      \token{0.6,-0.15}{west}{\psi(a)};
\end{tikzpicture}\right).
\end{equation}}
\end{proof}

\details{We state some useful properties of the functor $ \Psi $ here, which we will apply in the proof of Proposition \ref{ind_trace_map}. The functor $ \Psi $ induces an algebra isomorphism $ \Psi \colon \operatorname{Pol}_{n}(A) \rightarrow \operatorname{Pol}_{n}(A') $ for each $ n \in \mathbb{N}_{+} $. If $ 1 \leq i \leq n-1 $, then we have
\begin{align} 
\Psi(s_{i}(f)) &= s_{i}(\Psi(f)), \qquad f \in \operatorname{Pol}_{n}(A). \label{jolly2}
\end{align}
Also, it is straightforward to see that, if $ n \in \mathbb{N}_{+} $ and $ 1 \leq i \leq n $, then we have in $ \operatorname{Pol}_{n}(A') $ that
\begin{equation} \label{crispy}
\Psi(f_{i}) = \Psi(f)_{i}, \qquad f \in \operatorname{Pol}_{1}(A).
\end{equation}
Recall that, for $ n \in \mathbb{N}_{+} $ and $ 1 \leq i \leq n-1 $, we have the Frobenius Demazure operators $ \partial_{i} \colon \operatorname{Pol}_{n}(A) \rightarrow \operatorname{Pol}_{n}(A) $ from Definition \ref{autumn}. We set $ \partial_{i,A} \colon \operatorname{Pol}_{n}(A) \rightarrow \operatorname{Pol}_{n}(A) $ to be the Frobenius Demazure operators associated to $ \operatorname{Pol}_{n}(A) $, and we set $ \partial_{i,A'} \colon \operatorname{Pol}_{n}(A') \rightarrow \operatorname{Pol}_{n}(A') $ to be the Frobenius Demazure operators associated to $ \operatorname{Pol}_{n}(A') $. This is to avoid confusion between the various Demazure operators. Then for $ n \in \mathbb{N}_{+} $ and $ 1 \leq i \leq n-1 $, one can show by induction that the induced algebra isomorphism $ \Psi \colon \operatorname{Pol}_{n}(A) \rightarrow \operatorname{Pol}_{n}(A') $ satisfies
\begin{align} 
\Psi(\partial_{i,A}(f)) &= \partial_{i,A'}(\Psi(f)), \qquad f \in \operatorname{Pol}_{n}(A). \label{jolly}
\end{align}
Indeed, first note that, if $ f \in A^{\otimes n} $ or $ f = x_{j} $ for some $ j \neq i,i+1 $, then \cref{jolly} holds, since in this case we have that both sides of \cref{jolly} are equal to zero. We next show that \cref{jolly} holds when $ f = x_{i} $. First note that the left dual basis of \(B\), with respect to $ \operatorname{tr}' $, is given by $ (Bu)^{\vee} = \{(bu)^{\vee} : b \in B\} $. We now compute that
\begin{multline*}
\partial_{i,A'}(\Psi(x_{i})) = (-1)^{\bar{u}}\partial_{i,A'}(x_{i}u_{i}) = (-1)^{\bar{u}}t_{i,i+1}u_{i} = (-1)^{\bar{u}} \sum_{b \in B} (-1)^{\varepsilon'\bar{b}}b_{i}((bu)^{\vee})_{i+1}u_{i} 
\\ = \sum_{b \in B} (-1)^{\varepsilon (\bar{u} + \bar{b})}(bu)_{i}((bu)^{\vee})_{i+1} = \Psi(t_{i,i+1}) = \Psi(\partial_{i,A}(x_{i})).
\end{multline*}
Here, the third equality follows from \cref{wicked} and the fact that the left dual basis of \(B\), with respect to $ \operatorname{tr}' $, is given by $ (Bu)^{\vee} = \{(bu)^{\vee} : b \in B\} $. The fourth equality follows from the fact that $ \varepsilon' = \bar{u} + \varepsilon $ (see Lemma \ref{twotrace232}) and $ \overline{(bu)^{\vee}} = \bar{b} + \bar{u} + \varepsilon $ (see \cref{dualparity}). By a similar computation, one can show that \cref{jolly} holds for $ f = x_{i+1} $. Now suppose that \cref{jolly} holds for some $ f,g \in \operatorname{Pol}_{n}(A) $. Then we have
\begin{align*}
\Psi(\partial_{i,A}(fg)) &= \Psi(\partial_{i,A}(f)g + s_{i}(f)\partial_{i,A}(g)) 
\\ &= \partial_{i,A'}(\Psi(f))\Psi(g) + s_{i}(\Psi(f))\partial_{i,A'}(\Psi(g))
\\ &= \partial_{i,A'}(\Psi(fg)),
\end{align*}
where the first and third equalities follow by \cref{leibniz2}, and the second equality is obtained by using \cref{jolly2}. Thus \cref{jolly} now follows by induction.} 

Recall that we have the sets $ \mathcal{E}_{\operatorname{Pol}_{1}(A)} $ and $ \mathcal{E}_{\operatorname{Pol}_{1}(A')} $ from Definition \ref{multi_pen}. Then the induced algebra isomorphism $ \Psi \colon \operatorname{Pol}_{1}(A) \rightarrow \operatorname{Pol}_{1}(A') $ satisfies
\begin{equation} \label{grasp}
\Psi(\mathcal{E}_{\operatorname{Pol}_{1}(A)}) = \mathcal{E}_{\operatorname{Pol}_{1}(A')}.
\end{equation}

\begin{prop} \label{ind_trace_map}
We have an isomorphism of monoidal categories $ \mathcal{F} \colon \mathcal{LAW}(A) \rightarrow \mathcal{LAW}(A') $ which maps $ \go $ to $ \go $ and $ Q \in \mathcal{E}_{\operatorname{Pol}_{1}(A)} $ to $ \Psi(Q) \in \mathcal{E}_{\operatorname{Pol}_{1}(A')} $, and sends the generating morphisms as follows:
\begin{gather*}
\mathcal{F} \left(\begin{tikzpicture}[centerzero, thick]
      \draw[-] (0.6,-0.4) -- (0.6,0.4);
      \token{0.6,0}{west}{a};
\end{tikzpicture}\right) = \begin{tikzpicture}[centerzero, thick]
      \draw[-] (0.6,-0.4) -- (0.6,0.4);
      \token{0.6,0}{west}{a};
\end{tikzpicture}, 
\quad 
\mathcal{F} \left(\begin{tikzpicture}[centerzero, thick]
      \draw[-] (0.6,-0.4) -- (0.6,0.4);
      \singdot{0.6,0};
\end{tikzpicture}\right) = (-1)^{\bar{u}} \
\begin{tikzpicture}[centerzero, thick]
      \draw[-] (0.6,-0.4) -- (0.6,0.4);
      \singdot{0.6,0.15};
      \token{0.6,-0.15}{west}{u};
\end{tikzpicture} \ ,
\quad 
\mathcal{F} \left(\begin{tikzpicture}[centerzero, thick]
      \draw[-] (0.9,-0.4) -- (0.3,0.4);
      \draw[-] (0.3,-0.4) -- (0.9,0.4);
\end{tikzpicture}\right) 
=
\begin{tikzpicture}[centerzero, thick]
      \draw[-] (0.9,-0.4) -- (0.3,0.4);
      \draw[-] (0.3,-0.4) -- (0.9,0.4);
\end{tikzpicture} \ ,
\\ \mathcal{F} \left(\begin{tikzpicture}[centerzero, thick]
      \draw[-] (0.9,-0.4) -- (0.3,0.4);
      \draw[-,wei] (0.3,-0.4) -- (0.9,0.4);
      \node at (0.3,-0.6) {\tiny $ Q $};
\end{tikzpicture}\right) 
=
\begin{tikzpicture}[centerzero, thick]
      \draw[-] (0.9,-0.4) -- (0.3,0.4);
      \draw[-,wei] (0.3,-0.4) -- (0.9,0.4);
      \node at (0.3,-0.6) {\tiny $ \Psi(Q) $};
\end{tikzpicture} \ ,
\quad 
\mathcal{F}\left(\begin{tikzpicture}[centerzero, thick]
      \draw[-] (0.3,-0.4) -- (0.9,0.4);
      \draw[-,wei] (0.9,-0.4) -- (0.3,0.4);
      \node at (0.9,-0.6) {\tiny $ Q $};
\end{tikzpicture}\right) 
= 
\begin{tikzpicture}[centerzero, thick]
      \draw[-] (0.3,-0.4) -- (0.9,0.4);
      \draw[-,wei] (0.9,-0.4) -- (0.3,0.4);
      \node at (0.9,-0.6) {\tiny $ \Psi(Q) $};
\end{tikzpicture} \ ,
\end{gather*}
for all $ a \in A $. Its inverse $ \mathcal{G} \colon \mathcal{LAW}(A') \rightarrow \mathcal{LAW}(A) $ maps $ \go $ to $ \go $ and $ Q \in \mathcal{E}_{\operatorname{Pol}_{1}(A')} $ to $ \Psi^{-1}(Q) \in \mathcal{E}_{\operatorname{Pol}_{1}(A)} $, and sends the generating morphisms as follows:
\begin{gather*}
\mathcal{G} \left(\begin{tikzpicture}[centerzero, thick]
      \draw[-] (0.6,-0.4) -- (0.6,0.4);
      \token{0.6,0}{west}{a};
\end{tikzpicture}\right) = \begin{tikzpicture}[centerzero, thick]
      \draw[-] (0.6,-0.4) -- (0.6,0.4);
      \token{0.6,0}{west}{a};
\end{tikzpicture}, 
\quad 
\mathcal{G} \left(\begin{tikzpicture}[centerzero, thick]
      \draw[-] (0.6,-0.4) -- (0.6,0.4);
      \singdot{0.6,0};
\end{tikzpicture}\right) = (-1)^{\bar{u}} \
\begin{tikzpicture}[centerzero, thick]
      \draw[-] (0.6,-0.4) -- (0.6,0.4);
      \singdot{0.6,0.15};
      \token{0.6,-0.15}{west}{u^{-1}};
\end{tikzpicture} \ ,
\quad 
\mathcal{G} \left(\begin{tikzpicture}[centerzero, thick]
      \draw[-] (0.9,-0.4) -- (0.3,0.4);
      \draw[-] (0.3,-0.4) -- (0.9,0.4);
\end{tikzpicture}\right) 
=
\begin{tikzpicture}[centerzero, thick]
      \draw[-] (0.9,-0.4) -- (0.3,0.4);
      \draw[-] (0.3,-0.4) -- (0.9,0.4);
\end{tikzpicture} \ ,
\\ \mathcal{G} \left(\begin{tikzpicture}[centerzero, thick]
      \draw[-] (0.9,-0.4) -- (0.3,0.4);
      \draw[-,wei] (0.3,-0.4) -- (0.9,0.4);
      \node at (0.3,-0.6) {\tiny $ Q $};
\end{tikzpicture}\right) 
=
\begin{tikzpicture}[centerzero, thick]
      \draw[-] (0.9,-0.4) -- (0.3,0.4);
      \draw[-,wei] (0.3,-0.4) -- (0.9,0.4);
      \node at (0.3,-0.6) {\tiny $ \Psi^{-1}(Q) $};
\end{tikzpicture} \ ,
\quad 
\mathcal{G}\left(\begin{tikzpicture}[centerzero, thick]
      \draw[-] (0.3,-0.4) -- (0.9,0.4);
      \draw[-,wei] (0.9,-0.4) -- (0.3,0.4);
      \node at (0.9,-0.6) {\tiny $ Q $};
\end{tikzpicture}\right) 
= 
\begin{tikzpicture}[centerzero, thick]
      \draw[-] (0.3,-0.4) -- (0.9,0.4);
      \draw[-,wei] (0.9,-0.4) -- (0.3,0.4);
      \node at (0.9,-0.6) {\tiny $ \Psi^{-1}(Q) $};
\end{tikzpicture} \ ,
\end{gather*}
for all $ a \in A $. 
\end{prop}

\begin{proof}
To prove the existence of $ \mathcal{F} $, we need to check the relations \cref{tokrel1}, \cref{braid2}, \cref{affwreath}, and \cref{hw1,hw2,hw3,hw5,green}. These are straightforward, so we omit them here. It then follows by symmetry that the functor $ \mathcal{G} \colon \mathcal{LAW}(A') \rightarrow \mathcal{LAW}(A) $ exists. Finally, it is straightforward to see that $ \mathcal{F} $ and $ \mathcal{G} $ are mutual inverses.
\details{\emph{Relation} \cref{affwreath}: We compute that
\begin{align*}
\mathcal{F}\left(\begin{tikzpicture}[centerzero, thick]
        \draw[-] (0.3,-0.4) -- (-0.3,0.4);
        \draw[-] (-0.3,-0.4) -- (0.3,0.4);
        \singdot{-0.15,-0.2};
\end{tikzpicture}
-
\begin{tikzpicture}[centerzero, thick]
        \draw[-] (-0.3,-0.4) -- (0.3,0.4);
        \draw[-] (0.3,-0.4) -- (-0.3,0.4);
        \singdot{0.171,0.228};
\end{tikzpicture}\right)
&=
(-1)^{\bar{u}}
\begin{tikzpicture}[centerzero, thick]
        \draw[-] (0.3,-0.4) -- (-0.3,0.4);
        \draw[-] (-0.3,-0.4) -- (0.3,0.4);
        \token{-0.2,-0.26666666}{east}{u};
        \singdot{-0.1,-0.133333333};
\end{tikzpicture}
- (-1)^{\bar{u}}
\begin{tikzpicture}[centerzero, thick]
        \draw[-] (0.3,-0.4) -- (-0.3,0.4);
        \draw[-] (-0.3,-0.4) -- (0.3,0.4);
        \singdot{0.2,0.26666666};
        \token{0.1,0.133333333}{west}{u};
\end{tikzpicture}
\\ &= - (-1)^{\bar{u}} \ 
\begin{tikzpicture}[centerzero]
        \draw (-0.2,-0.4) -- (-0.2,0.4);
        \draw (0.2,-0.4) -- (0.2,0.4);
        \teleport{-0.2,0.1}{0.2,-0.1};
        \token{-0.2,-0.2}{east}{u};
\end{tikzpicture}
\\ &= -(-1)^{\bar{u}}\sum_{b \in B} (-1)^{\varepsilon' \bar{b}}
\begin{tikzpicture}[centerzero, thick]
        \draw[-] (0.2,-0.4) -- (0.2,0.4);
        \draw[-] (-0.2,-0.4) -- (-0.2,0.4);
        \token{-0.2,0.2}{east}{b};
        \token{0.2,0}{west}{(bu)^{\vee}};
        \token{-0.2,-0.2}{east}{u};
\end{tikzpicture}
\\ &= -\sum_{b \in B} (-1)^{\varepsilon(\bar{u}+\bar{b})}
\begin{tikzpicture}[centerzero, thick]
        \draw[-] (0.2,-0.4) -- (0.2,0.4);
        \draw[-] (-0.2,-0.4) -- (-0.2,0.4);
        \token{-0.2,0.2}{east}{bu};
        \token{0.2,-0.2}{west}{(bu)^{\vee}};
\end{tikzpicture}
\\ &= \mathcal{F}\left( - \ \begin{tikzpicture}[centerzero]
        \draw (-0.2,-0.4) -- (-0.2,0.4);
        \draw (0.2,-0.4) -- (0.2,0.4);
        \teleport{-0.2,0.1}{0.2,-0.1};
\end{tikzpicture}\right).
\end{align*}
Here, the third equality follows from \cref{teleporter66}, and the fact that the left dual basis of \(B\), with respect to $ \operatorname{tr}' $, is $ (Bu)^{\vee} = \{(bu)^{\vee} : b \in B\} $. The fourth equality follows from the fact that $ \varepsilon' = \bar{u} + \varepsilon $ (see Lemma \ref{twotrace232}) and $ \overline{(bu)^{\vee}} = \bar{b} + \bar{u} + \varepsilon $ (see \cref{dualparity}). For the second relation in \cref{affwreath}, we compute using Lemma \ref{twotrace232} that, for all $ a \in A $,
\begin{equation*} 
\mathcal{F}\left(\begin{tikzpicture}[centerzero, thick]
      \draw[-] (0.6,-0.4) -- (0.6,0.4);
      \singdot{0.6,-0.15};
      \token{0.6,0.15}{west}{a};
\end{tikzpicture}\right)
=
(-1)^{\bar{u}} \
\begin{tikzpicture}[centerzero, thick]
      \draw[-] (0.6,-0.4) -- (0.6,0.4);
      \singdot{0.6,0};
      \token{0.6,0.25}{west}{a};
      \token{0.6,-0.25}{west}{u};
\end{tikzpicture}
=
(-1)^{\bar{u}+\varepsilon' \bar{a}} \
\begin{tikzpicture}[centerzero, thick]
      \draw[-] (0.6,-0.4) -- (0.6,0.4);
      \singdot{0.6,0.15};
      \token{0.6,-0.15}{west}{\psi'(a)u};
\end{tikzpicture}
=
(-1)^{\bar{u}+\varepsilon \bar{a}} \
\begin{tikzpicture}[centerzero, thick]
      \draw[-] (0.6,-0.4) -- (0.6,0.4);
      \singdot{0.6,0.15};
      \token{0.6,-0.15}{west}{u\psi(a)};
\end{tikzpicture}
= 
\mathcal{F}\left( (-1)^{\varepsilon \bar{a}} \ 
\begin{tikzpicture}[centerzero, thick]
      \draw[-] (0.6,-0.4) -- (0.6,0.4);
      \singdot{0.6,0.15};
      \token{0.6,-0.15}{west}{\psi(a)};
\end{tikzpicture}\right).
\end{equation*}
\emph{Relation} \cref{hw3}: We only prove the first relation in \cref{hw3}, since the second relation is similar. We compute that 
\begin{equation*}
\mathcal{F}\left(\begin{tikzpicture}[centerzero, thick]
        \draw[-] (-0.2,-0.5) to[out=up,in=down] (0.2,0) to[out=up,in=down] (-0.2,0.5);
        \draw[-, wei] (0.2,-0.5) to[out=up,in=down] (-0.2,0) to[out=up,in=down] (0.2,0.5);
        \node at (0.2,-0.7) {\tiny $ Q $};
\end{tikzpicture}\right) 
= 
\begin{tikzpicture}[centerzero, thick]
        \draw[-] (-0.2,-0.5) to[out=up,in=down] (0.2,0) to[out=up,in=down] (-0.2,0.5);
        \draw[-, wei] (0.2,-0.5) to[out=up,in=down] (-0.2,0) to[out=up,in=down] (0.2,0.5);
        \node at (0.2,-0.7) {\tiny $ \Psi(Q) $};
\end{tikzpicture}
= 
\begin{tikzpicture}[centerzero, thick]
        \pin{-0.2,0}{-1,0}{\Psi(Q)};
        \draw[-] (-0.2,-0.5) -- (-0.2,0.5);
        \draw[-, wei] (0.2,-0.5) -- (0.2,0.5);
        \node at (0.2,-0.7) {\tiny $ \Psi(Q) $};
\end{tikzpicture}.
\end{equation*}
Thus the relation \cref{hw3} reduces to checking that 
\begin{equation} \label{light_switch}
\mathcal{F}\left(\begin{tikzpicture}[centerzero, thick]
        \pin{-0.2,0}{-1,0}{Q};
        \draw[-] (-0.2,-0.5) -- (-0.2,0.5);
        \draw[-, wei] (0.2,-0.5) -- (0.2,0.5);
        \node at (0.2,-0.7) {\tiny $ Q $};
\end{tikzpicture}\right)
= 
\begin{tikzpicture}[centerzero, thick]
        \pin{-0.2,0}{-1,0}{\Psi(Q)};
        \draw[-] (-0.2,-0.5) -- (-0.2,0.5);
        \draw[-, wei] (0.2,-0.5) -- (0.2,0.5);
        \node at (0.2,-0.7) {\tiny $ \Psi(Q) $};
\end{tikzpicture}.
\end{equation}
The equation \cref{light_switch} will follow if we are able to show that
\begin{equation} \label{flakes}
\mathcal{F}\left(\begin{tikzpicture}[centerzero, thick]
        \pin{-0.2,0}{-1,0}{f};
        \draw[-] (-0.2,-0.5) -- (-0.2,0.5);
        \draw[-, wei] (0.2,-0.5) -- (0.2,0.5);
        \node at (0.2,-0.7) {\tiny $ Q $};
\end{tikzpicture}\right)
= 
\begin{tikzpicture}[centerzero, thick]
        \pin{-0.2,0}{-1,0}{\Psi(f)};
        \draw[-] (-0.2,-0.5) -- (-0.2,0.5);
        \draw[-, wei] (0.2,-0.5) -- (0.2,0.5);
        \node at (0.2,-0.7) {\tiny $ \Psi(Q) $};
\end{tikzpicture}, \qquad f \in \operatorname{Pol}_{1}(A).
\end{equation}
The equation \cref{flakes} is straightforward to see.
\\ \indent \emph{Relation} \cref{green}: We compute that 
\begin{equation*}
\mathcal{F}\left(\begin{tikzpicture}[centerzero, thick]
        \draw[-] (0.4,-0.5) -- (-0.4,0.5);
        \draw[-] (-0.4,-0.5) -- (0.4,0.5);
        \node at (0,-0.7) {\tiny $ Q $};
        \draw[-, wei] (0,-0.5) to[out=up, in=down] (-0.4,0) to[out=up,in=down] (0,0.5);
\end{tikzpicture}\right)
= 
\begin{tikzpicture}[centerzero, thick]
        \draw[-] (0.4,-0.5) -- (-0.4,0.5);
        \draw[-] (-0.4,-0.5) -- (0.4,0.5);
        \node at (0,-0.7) {\tiny $ \Psi(Q) $};
        \draw[-, wei] (0,-0.5) to[out=up, in=down] (-0.4,0) to[out=up,in=down] (0,0.5);
\end{tikzpicture}
\quad \stackrel{\mathclap{\cref{green}}}{=} \quad
\begin{tikzpicture}[centerzero, thick]
        \draw[-] (2,-0.5) -- (1.2,0.5);
        \draw[-] (1.2,-0.5) -- (2,0.5);
        \node at (1.6,-0.7) {\tiny $ \Psi(Q) $};
        \draw[-, wei] (1.6,-0.5) to[out=up, in=down] (2,0) to[out=up,in=down] (1.6,0.5);
\end{tikzpicture}
\ - \
\begin{tikzpicture}[centerzero, thick]
        \node at (0,-0.7) {\tiny $ \Psi(Q) $};
        \draw[-] (-0.6,-0.5) -- (-0.6,0.5);
        \draw[-, wei] (0,-0.5) -- (0,0.5);
        \draw[-] (0.6,-0.5) -- (0.6,0.5);
        \pinpin{0.6,0}{-0.6,0}{2.2,0}{\partial_{1,A'}(\Psi(Q)_{1})};
\end{tikzpicture}
\end{equation*}
and
\begin{equation*}
\mathcal{F}\left(\begin{tikzpicture}[centerzero, thick]
        \draw[-] (2,-0.5) -- (1.2,0.5);
        \draw[-] (1.2,-0.5) -- (2,0.5);
        \node at (1.6,-0.7) {\tiny $ Q $};
        \draw[-, wei] (1.6,-0.5) to[out=up, in=down] (2,0) to[out=up,in=down] (1.6,0.5);
\end{tikzpicture}
\ - \
\begin{tikzpicture}[centerzero, thick]
        \node at (0,-0.7) {\tiny $ Q $};
        \draw[-] (-0.6,-0.5) -- (-0.6,0.5);
        \draw[-, wei] (0,-0.5) -- (0,0.5);
        \draw[-] (0.6,-0.5) -- (0.6,0.5);
        \pinpin{0.6,0}{-0.6,0}{1.8,0}{\partial_{1,A}(Q_{1})};
\end{tikzpicture}\right)
=
\begin{tikzpicture}[centerzero, thick]
        \draw[-] (2,-0.5) -- (1.2,0.5);
        \draw[-] (1.2,-0.5) -- (2,0.5);
        \node at (1.6,-0.7) {\tiny $ \Psi(Q) $};
        \draw[-, wei] (1.6,-0.5) to[out=up, in=down] (2,0) to[out=up,in=down] (1.6,0.5);
\end{tikzpicture}
-
\mathcal{F}\left(
\begin{tikzpicture}[centerzero, thick]
        \node at (0,-0.7) {\tiny $ Q $};
        \draw[-] (-0.6,-0.5) -- (-0.6,0.5);
        \draw[-, wei] (0,-0.5) -- (0,0.5);
        \draw[-] (0.6,-0.5) -- (0.6,0.5);
        \pinpin{0.6,0}{-0.6,0}{1.8,0}{\partial_{1,A}(Q_{1})};
\end{tikzpicture}\right).
\end{equation*}
Thus to prove \cref{green}, we must show that 
\begin{equation} \label{charger0.5}
\mathcal{F}\left(
\begin{tikzpicture}[centerzero, thick]
        \node at (0,-0.7) {\tiny $ Q $};
        \draw[-] (-0.6,-0.5) -- (-0.6,0.5);
        \draw[-, wei] (0,-0.5) -- (0,0.5);
        \draw[-] (0.6,-0.5) -- (0.6,0.5);
        \pinpin{0.6,0}{-0.6,0}{1.8,0}{\partial_{1,A}(Q_{1})};
\end{tikzpicture}\right)
= 
\begin{tikzpicture}[centerzero, thick]
        \node at (0,-0.7) {\tiny $ \Psi(Q) $};
        \draw[-] (-0.6,-0.5) -- (-0.6,0.5);
        \draw[-, wei] (0,-0.5) -- (0,0.5);
        \draw[-] (0.6,-0.5) -- (0.6,0.5);
        \pinpin{0.6,0}{-0.6,0}{2.2,0}{\partial_{1,A'}(\Psi(Q)_{1})};
\end{tikzpicture} \ .
\end{equation}
Due to \cref{jolly} and the fact that $ \Psi(Q_{1}) = \Psi(Q)_{1} $ (see \cref{crispy}), we have that \cref{charger0.5} will follow if we are able to show that
\begin{equation} \label{charger2}
\mathcal{F}\left(
\begin{tikzpicture}[centerzero, thick]
        \node at (0,-0.7) {\tiny $ Q $};
        \draw[-] (-0.6,-0.5) -- (-0.6,0.5);
        \draw[-, wei] (0,-0.5) -- (0,0.5);
        \draw[-] (0.6,-0.5) -- (0.6,0.5);
        \pinpin{0.6,0}{-0.6,0}{1.4,0}{f};
\end{tikzpicture}\right)
= 
\begin{tikzpicture}[centerzero, thick]
        \node at (0,-0.7) {\tiny $ \Psi(Q) $};
        \draw[-] (-0.6,-0.5) -- (-0.6,0.5);
        \draw[-, wei] (0,-0.5) -- (0,0.5);
        \draw[-] (0.6,-0.5) -- (0.6,0.5);
        \pinpin{0.6,0}{-0.6,0}{1.4,0}{\Psi(f)};
\end{tikzpicture} \ , \qquad f \in \operatorname{Pol}_{2}(A).
\end{equation}
The equation \cref{charger2} is straightforward to see.}
\end{proof}

By restricting to objects, $ \mathcal{F} $ induces a monoid isomorphism $ \mathcal{F} \colon F(\widehat{\mathcal{E}}_{\operatorname{Pol}_{1}(A)}) \rightarrow F(\widehat{\mathcal{E}}_{\operatorname{Pol}_{1}(A')}) $. This monoid isomorphism satisfies $ \mathcal{F}(\mathbf{Q}) \in F(\mathcal{E}_{\operatorname{Pol}_{1}(A')}) $ for all $ \mathbf{Q} \in F(\mathcal{E}_{\operatorname{Pol}_{1}(A)}) $.

\begin{cor} \label{ind_trace_map_2}
Let $ d \in \mathbb{N}_{+} $ and $ \mathbf{Q} \in F(\mathcal{E}_{\operatorname{Pol}_{1}(A)}) $. Then we have an isomorphism of algebras 
\begin{equation*}
W_{d,\mathbf{Q}}^{\operatorname{aff}}(A) \cong W_{d,\mathcal{F}(\mathbf{Q})}^{\operatorname{aff}}(A').
\end{equation*}
\end{cor}

\begin{proof}
This follows from Proposition \ref{ind_trace_map}.
\end{proof}

\details{We have that 
\begin{equation} \label{pencil_sharp}
\mathcal{F}(\Gamma_{d,\mathbf{Q}}) = \Gamma_{d,\mathcal{F}(\mathbf{Q})}.
\end{equation}
Indeed, we compute that 
\begin{gather*}
\mathcal{F}(\Gamma_{d,\mathbf{Q}}) \ \stackrel{\mathclap{\cref{fluffy2}}}{=} \ \mathcal{F}(F(\widehat{\mathcal{E}}_{\operatorname{Pol}_{1}(A)})_{\go^{d},\mathbf{Q}}) = F(\widehat{\mathcal{E}}_{\operatorname{Pol}_{1}(A')})_{\mathcal{F}(\go^{d}),\mathcal{F}(\mathbf{Q})} = F(\widehat{\mathcal{E}}_{\operatorname{Pol}_{1}(A')})_{\go^{d},\mathcal{F}(\mathbf{Q})} \ \stackrel{\mathclap{\cref{fluffy2}}}{=} \ \Gamma_{d,\mathcal{F}(\mathbf{Q})}.
\end{gather*}
Here, the third equality follows from the fact that the functor $ \mathcal{F} $ fixes $ \go $ (see Proposition \ref{ind_trace_map}), and the second equality follows from the fact that $ \mathcal{F} \colon F(\widehat{\mathcal{E}}_{\operatorname{Pol}_{1}(A)}) \rightarrow F(\widehat{\mathcal{E}}_{\operatorname{Pol}_{1}(A')}) $ is a monoid isomorphism. 
\\ \indent Now, the functor $ \mathcal{F} $ induces an isomorphism between the additive envelopes $ \operatorname{Add}(\mathcal{LAW}(A)) $ and $ \operatorname{Add}(\mathcal{LAW}(A')) $. This isomorphism in turn induces an algebra isomorphism 
\begin{gather*}
W_{d,\mathbf{Q}}^{\operatorname{aff}}(A) \quad \stackrel{\mathclap{\cref{HLAWPA2}}}{=} \quad \operatorname{End}_{\operatorname{Add}(\mathcal{LAW}(A))}\left(\bigoplus_{\mathbf{i} \in \Gamma_{d,\mathbf{Q}}} \mathbf{i} \right) \cong \operatorname{End}_{\operatorname{Add}(\mathcal{LAW}(A'))}\left(\bigoplus_{\mathbf{i} \in \Gamma_{d,\mathbf{Q}}} \mathcal{F}(\mathbf{i}) \right)
\\ \stackrel{\mathclap{\cref{pencil_sharp}}}{=} \quad \operatorname{End}_{\operatorname{Add}(\mathcal{LAW}(A'))}\left(\bigoplus_{\mathbf{j} \in \Gamma_{d,\mathcal{F}(\mathbf{Q})}} \mathbf{j}\right) \quad \stackrel{\mathclap{\cref{HLAWPA2}}}{=} \quad W_{d,\mathcal{F}(\mathbf{Q})}^{\operatorname{aff}}(A')
\end{gather*}
as required.}

\subsection{Cyclotomic quotients} \label{Cyclotomic_quotients}

In this section, we define cyclotomic quotients of higher-level affine wreath product algebras. These are natural analogues of the cyclotomic quotients for tensor product algebras; see \cite[Def.~4.7]{Webster}. Also, they are analogues of the cyclotomic quotients for higher-level affine Hecke algebras; see \cite[Def.~6.1]{Maksimau-Stroppel} and \cite[Def.~5.6]{Webster2}. Given an element $ \mathbf{i} \in F(\widehat{\mathcal{E}}_{\operatorname{Pol}_{1}(A)}) $, we define $ \mathbf{i}_{1} \in \widehat{\mathcal{E}}_{\operatorname{Pol}_{1}(A)} $ to be the first entry in the word $ \mathbf{i} $.

\begin{defn} \label{Highercycl}
Let $ d \in \mathbb{N}_{+} $ and $ \mathbf{Q} \in F(\mathcal{E}_{\operatorname{Pol}_{1}(A)}) $, and assume that $ |\mathbf{Q}| \geq 1 $. Then we define the \emph{cyclotomic $ (d,\mathbf{Q}) $-wreath product algebra} $ W_{d,\mathbf{Q}}^{\operatorname{cyc}}(A) $ to be the quotient of $ W_{d,\mathbf{Q}}^{\operatorname{aff}}(A) $ by the two-sided ideal generated by the elements of the set $ \{1_{\mathbf{i}} : \mathbf{i} \in \Gamma_{d,\mathbf{Q}}, \ \mathbf{i}_{1} = \go \} $.
\end{defn}

In other words, any diagram in $ W_{d,\mathbf{Q}}^{\operatorname{aff}}(A) $ that has a piece of black strand to the left of all the red strands is equal to \(0\) in the quotient $ W_{d,\mathbf{Q}}^{\operatorname{cyc}}(A) $. 

\begin{egg}
Suppose $ d = 2 $ and $ \mathbf{Q} = QQ' $ for some $ Q, Q' \in \mathcal{E}_{\operatorname{Pol}_{1}(A)} $. Let $ a \in A $, and consider the elements $ D_{1}, D_{2} \in W_{d,\mathbf{Q}}^{\operatorname{aff}}(A) $ given below:
\begin{equation*} 
D_{1} = 
\begin{tikzpicture}[centerzero, thick]
         \draw[-] (0.6,-0.7) to[out=up, in=down] (0.9,0) to[out=up, in=down] (0.3,0.7);
         \draw[-] (0.9,-0.7) to[out=up, in=down] (0,0.7);
         \draw[-,wei] (0,-0.7) to[out=up, in=down] (0.6,0.7);
         \draw[-,wei] (0.3,-0.7) to[out=up, in=down] (0.9,0.7);
         \singdot{0.9,0};
         \node at (0,-0.9) {\tiny $ Q $};
         \node at (0.3,-0.9) {\tiny $ Q' $};
\end{tikzpicture}
+ 3
\begin{tikzpicture}[centerzero, thick]
         \draw[-] (0,-0.7) to[out=up, in=down] (0.3,0.7);
         \draw[-,wei] (0.3,-0.7) to[out=up, in=down] (-0.3,0.7);
         \draw[-,wei] (0.6,-0.7) -- (0.6,0.7);
         \draw[-] (0.9,-0.7) -- (0.9,0.7);
         \token{0.9,0.3}{west}{a};
         \singdot{0.9,-0.2};
         \node at (0.3,-0.9) {\tiny $ Q $};
         \node at (0.6,-0.9) {\tiny $ Q' $};
\end{tikzpicture} \ , \quad 
D_{2} = 
\begin{tikzpicture}[centerzero, thick]
         \draw[-] (0.6,-0.7) to[out=up, in=down] (-0.3,0) to[out=up, in=down] (0.6,0.7);
         \draw[-] (0.9,-0.7) -- (0.9,0.7);
         \draw[-,wei] (0,-0.7) -- (0,0.7);
         \draw[-,wei] (0.3,-0.7) -- (0.3,0.7);
         \node at (0,-0.9) {\tiny $ Q $};
         \node at (0.3,-0.9) {\tiny $ Q' $};
\end{tikzpicture} \ .
\end{equation*}
Then both $ D_{1} $ and $ D_{2} $ lie in the kernel of the projection map $ W_{d,\mathbf{Q}}^{\operatorname{aff}}(A) \rightarrow W_{d,\mathbf{Q}}^{\operatorname{cyc}}(A) $.
\end{egg}

We can also define cyclotomic quotients of the affine wreath product algebras $ W_{n}^{\operatorname{aff}}(A) $ from Definition \ref{pluto2}.

\begin{defn}[{\cite[§6]{Savage}}]
Let $ n \in \mathbb{N}_{+} $ and $ Q \in \mathcal{E}_{\operatorname{Pol}_{1}(A)} $. Then we define the \emph{cyclotomic wreath product algebra} $ W_{n}^{Q}(A) $ to be the quotient of $ W_{n}^{\operatorname{aff}}(A) $ by the two-sided ideal generated by the element $ Q_{1} \in \operatorname{Pol}_{n}(A) $ (see \cref{fog}).
\end{defn}

The following theorem shows that, when $ \mathbf{Q} \in F(\mathcal{E}_{\operatorname{Pol}_{1}(A)}) $ is a word of length one, the cyclotomic $ (d,\mathbf{Q}) $-wreath product algebra is isomorphic to a cyclotomic wreath product algebra. It can be viewed as an analogue of \cite[Thm.~4.18]{Webster}, \cite[Prop.~5.7]{Webster2}, and \cite[Lem.~6.9]{Maksimau-Stroppel}.

\begin{theo} \label{cyclotomic_aff_thm}
Let $ d \in \mathbb{N}_{+} $ and $ Q \in \mathcal{E}_{\operatorname{Pol}_{1}(A)} $. Then we have an isomorphism of algebras 
\begin{equation*}
W_{d,Q}^{\operatorname{cyc}}(A) \cong W_{d}^{Q}(A).
\end{equation*}
\end{theo}

\begin{proof}
We have a non-unital injective algebra homomorphism $ \varphi \colon W_{d}^{\operatorname{aff}}(A) \rightarrow W_{d,Q}^{\operatorname{aff}}(A) $ that adds a single red strand to the left of each diagram. Composing with the projection map $ p \colon W_{d,Q}^{\operatorname{aff}}(A) \rightarrow W_{d,Q}^{\operatorname{cyc}}(A) $ yields a unital algebra homomorphism $ \phi \colon W_{d}^{\operatorname{aff}}(A) \rightarrow W_{d,Q}^{\operatorname{cyc}}(A) $. So we have the following commutative diagram of maps:
\begin{equation*}
\begin{tikzpicture}[auto]
  \node (A) at (0,0) {$ W_{d}^{\operatorname{aff}}(A) $};
  \node (B) at (2.5,0) {$ W_{d,Q}^{\operatorname{aff}}(A) $};
  \node (C) at (5,0) {$ W_{d,Q}^{\operatorname{cyc}}(A) $};
  \draw[->] (A) -- node[above] {$ \varphi $} (B);
  \draw[->] (B) -- node[above] {$ p $} (C);
  \draw[->, bend left=40] (A) to node[above] {$ \phi $} (C);
\end{tikzpicture}
\end{equation*}
By Theorem \ref{HL-basis}, there is a spanning set of $ W_{d,Q}^{\operatorname{cyc}}(A) $ that lies in the image of $ \phi $. This gives that $ \phi $ is surjective. Set $ \mathcal{I} $ to be the two-sided ideal of $ W_{d}^{\operatorname{aff}}(A) $ generated by $ Q_{1} $. Then in order to prove the result, it only remains to show that the kernel of $ \phi $ is equal to $ \mathcal{I} $. That the latter is contained in the former follows by an application of \cref{hw3}. Namely, we have
\begin{equation*}
\phi\left(\begin{tikzpicture}[centerzero, thick]
        \draw[-] (0.3,-0.5) -- (0.3,0.5);
        \draw[-] (0.8,-0.5) -- (0.8,0.5);
        \node at (1.4,0) {$ \cdots $};
        \pin{0.3,0}{-0.5,0}{Q};
        \draw[-] (2,-0.5) -- (2,0.5);
\end{tikzpicture}\ \right)
\ =
\begin{tikzpicture}[centerzero, thick]
        \draw[wei] (-0.9,-0.5) -- (-0.9,0.5);
        \draw[-] (0.3,-0.5) -- (0.3,0.5);
        \draw[-] (0.8,-0.5) -- (0.8,0.5);
        \node at (1.4,0) {$ \cdots $};
        \node at (-0.9,-0.7) {\tiny $ Q $};
        \pin{0.3,0}{-0.5,0}{Q};
        \draw[-] (2,-0.5) -- (2,0.5);
\end{tikzpicture}
\ =
\begin{tikzpicture}[centerzero, thick]
        \draw[-] (0.2,-0.5) to[out=up,in=down] (-0.2,0) to[out=up,in=down] (0.2,0.5);
        \draw[-, wei] (-0.2,-0.5) to[out=up,in=down] (0.2,0) to[out=up,in=down] (-0.2,0.5);
        \draw[-] (0.7,-0.5) -- (0.7,0.5);
        \draw[-] (1.9,-0.5) -- (1.9,0.5);
        \node at (1.3,0) {$ \cdots $};        
        \node at (-0.2,-0.7) {\tiny $ Q $};
\end{tikzpicture}
\ = 0.
\end{equation*}
On the other hand, suppose $ y \in \operatorname{ker}(\phi) $. Then by the injectivity of $ \varphi $, it suffices to show that $ \varphi(y) \in \varphi(\mathcal{I}) $.
\\ \indent Note first that since $ y \in \operatorname{ker}(\phi) $, and $ \phi = p \circ \varphi $, we have that $ \varphi(y) $ is in the kernel of the projection map \(p\). So $ \varphi(y) $ is a $ \Bbbk $-linear combination of diagrams with a piece of black strand to the left of the red strand. If $ D $ is such a diagram, then we prove by induction that $ D \in \varphi(\mathcal{I}) $, from which it will follow by $ \Bbbk $-linearity that $ \varphi(y) \in \varphi(\mathcal{I}) $. The value $ c $ that we induct on is half the number of red-black crossings plus the number of black-black crossings to the left of the red strand in \(D\). For the base case, if $ c = 1 $, then there is a single black strand in $ D $ which crosses over the red strand and immediately crosses back:
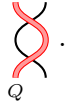
\begin{equation*} 
\begin{tikzpicture}[centerzero, thick]
        \draw[-] (0.2,-0.5) to[out=up,in=down] (-0.2,0) to[out=up,in=down] (0.2,0.5);
        \draw[-, wei] (-0.2,-0.5) to[out=up,in=down] (0.2,0) to[out=up,in=down] (-0.2,0.5);
        \node at (-0.2,-0.7) {\tiny $ Q $};
\end{tikzpicture} \ .
\end{equation*}
Note that this double crossing may contain some tokens or dots on the black strand, but these can be moved out by the relations \cref{hw1} and \cref{hw2}. Then an application of \cref{hw3} shows that $ D \in \varphi(\mathcal{I}) $, giving that the base case holds. Next, if $ c > 1 $, then \(D\) either has a red-black double crossing or a red-black triangle, where the red strand is on the right:
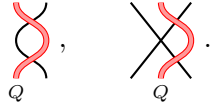
\begin{equation*}
\begin{tikzpicture}[centerzero, thick]
        \draw[-] (0.2,-0.5) to[out=up,in=down] (-0.2,0) to[out=up,in=down] (0.2,0.5);
        \draw[-, wei] (-0.2,-0.5) to[out=up,in=down] (0.2,0) to[out=up,in=down] (-0.2,0.5);
        \node at (-0.2,-0.7) {\tiny $ Q $};
\end{tikzpicture} \ ,
\qquad 
\begin{tikzpicture}[centerzero, thick]
        \draw[-] (2,-0.5) -- (1.2,0.5);
        \draw[-] (1.2,-0.5) -- (2,0.5);
        \node at (1.6,-0.7) {\tiny $ Q $};
        \draw[-, wei] (1.6,-0.5) to[out=up, in=down] (2,0) to[out=up,in=down] (1.6,0.5);
\end{tikzpicture} \ .
\end{equation*}
Again, this double crossing or triangle may contain tokens or dots on the black strands, but these can be moved out by the relations \cref{hw1} and \cref{hw2}. Then applying \cref{hw3} if \(D\) has a double crossing, or applying \cref{green} if \(D\) has a triangle, shows that \(D\) can be written as a $ \Bbbk $-linear combination of diagrams with smaller \(c\). Thus by the induction hypothesis, we have $ D \in \varphi(\mathcal{I}) $.
\end{proof}

\section{Affine Frobenius Hecke algebras}

In this section, we provide a recap of the affine Frobenius Hecke algebras. These algebras were first defined and studied in \cite{Rosso-Savage}. In this section, we fix a scalar $ z \in \Bbbk $, and we assume that the Frobenius algebra \(A\) is symmetric and carries a parity-preserving trace map. In other words, we have $ \psi = 1 $ and $ \varepsilon = \bar{0} $. 

\begin{defn} 
We define the category $ \mathcal{P}(A) $ to be the strict $ \Bbbk $-linear monoidal supercategory generated by one object $ \go $ and morphisms
\begin{equation*}
    \begin{tikzpicture}[anchorbase, thick]
        \draw[-] (0,-0.3) -- (0,0.3);
        \token{0,0}{west}{a};
    \end{tikzpicture}
    \colon \go \to \go, \quad
 \begin{tikzpicture}[anchorbase, thick]
        \draw[-] (0,0) -- (0,0.6);
        \singdot{0,0.3};
\end{tikzpicture}
\ \colon
\go \to \go, \quad
\begin{tikzpicture}[anchorbase, thick]
    \draw[-] (0,0) -- (0,0.6);
    \multdot{0,0.3}{west}{-1};
\end{tikzpicture}
\ \colon
\go \to \go, \quad a \in A,
\end{equation*}
modulo the relations \cref{tokrel1} and
\begin{gather} \label{P(A)_relations}
\begin{tikzpicture}[centerzero, thick]
        \draw[-] (0,-0.4) -- (0,0.4);
        \singdot{0,-0.15};
        \multdot{0,0.15}{west}{-1};
\end{tikzpicture}
\ = \
\begin{tikzpicture}[centerzero, thick]
        \draw[-] (0,-0.4) -- (0,0.4);
        \multdot{0,-0.15}{west}{-1};
        \singdot{0,0.15};
\end{tikzpicture}
\ = \
\begin{tikzpicture}[centerzero, thick]
        \draw[-] (0,-0.4) -- (0,0.4);
\end{tikzpicture}
\ ,\quad
\begin{tikzpicture}[centerzero, thick]
        \draw[-] (0,-0.4) -- (0,0.4);
        \token{0,0.15}{west}{a};
        \singdot{0,-0.15};
\end{tikzpicture}
\ =\
\begin{tikzpicture}[anchorbase, thick]
        \draw[-] (0,-0.4) -- (0,0.4);
        \token{0,-0.15}{west}{a};
        \singdot{0,0.15};
\end{tikzpicture}
\ ,\quad a \in A.
\end{gather}
The morphism $ \begin{tikzpicture}[anchorbase, thick]
        \draw[-] (0,-0.2) -- (0,0.2);
        \token{0,0}{west}{a};
    \end{tikzpicture} $ is called a \emph{token}, the morphism $ \begin{tikzpicture}[anchorbase, thick]
        \draw[-] (0,0) -- (0,0.6);
        \singdot{0,0.3};
    \end{tikzpicture} $
is called the \emph{dot}, and the morphism 
$ \begin{tikzpicture}[anchorbase, thick]
        \draw[-] (0,-0.3) -- (0,0.3);
        \multdot{0,0}{west}{-1};
\end{tikzpicture} $
is called the \emph{inverse dot}. The dot and inverse dot are even, and the parity of a token $ \begin{tikzpicture}[anchorbase, thick]
        \draw[-] (0,-0.2) -- (0,0.2);
        \token{0,0}{west}{a};
    \end{tikzpicture} $ is equal to the parity of \(a\). For $ n \in \mathbb{N}_{+} $, we define
\begin{equation*}
P_{n}(A) := \operatorname{End}_{\mathcal{P}(A)}(\go^{\otimes n}).
\end{equation*}
\end{defn}

Before stating the following proposition, recall that, throughout this paper, we number strands from \emph{left to right}.

\begin{prop} \label{verma}
Let $ n \in \mathbb{N}_{+} $. Then $ P_{n}(A) $ is isomorphic to the free product of $ \Bbbk[X_{1}^{\pm 1},\ldots,X_{n}^{\pm 1}] $ and $ A^{\otimes n} $ (where the $ X_{i} $ are understood to be even), subject to the relations
\begin{equation} \label{verma_2}
\mathbf{a}X_{i} = X_{i}\mathbf{a}, \qquad \mathbf{a} \in A^{\otimes n}, \ 1 \leq i \leq n.
\end{equation}
Under this isomorphism, $ a_{i} $, for $ 1 \leq i \leq n $, corresponds to the token labelled \(a\) on the \(i\)-th strand, and $ X_{i} $, for $ 1 \leq i \leq n $, corresponds to the dot on the \(i\)-th strand.
\end{prop}

\begin{proof}
This follows from Proposition \ref{zebra}.
\end{proof}

In what follows, we identify $ P_{n}(A) $ with the algebra given in Proposition \ref{verma}. We have an isomorphism of algebras $ P_{n}(A) \cong A^{\otimes n} \otimes \Bbbk[X_{1}^{\pm 1},\ldots,X_{n}^{\pm 1}] $, and $ P_{n}(A) $ has basis
\begin{equation*}
\{\mathbf{a}X_{1}^{k_{1}}\cdots X_{n}^{k_{n}} : \mathbf{a} \in B^{\otimes n}, \ k_{1},\ldots, k_{n} \in \mathbb{Z}\},
\end{equation*}
where recall that \(B\) is some basis of \(A\).

\begin{defn} \label{television}
Let $ n \in \mathbb{N}_{+} $. Then we set $ \Delta_{i} \colon P_{n}(A) \rightarrow P_{n}(A) $, $ 1 \leq i \leq n-1 $, to be the $ \Bbbk $-linear operators given by
\begin{equation} \label{television_2}
\Delta_{i}(\mathbf{a}p) = t_{i,i+1}\mathbf{a}\frac{X_{i+1}(s_{i}(p)-p)}{X_{i}-X_{i+1}}, \quad  \mathbf{a} \in A^{\otimes n}, \ p \in \Bbbk[X_{1}^{\pm 1}, \ldots, X_{n}^{\pm 1}].
\end{equation}
\end{defn}

%
%

\begin{defn} \label{quantum affine wreath product category}
We define the \emph{affine Frobenius Hecke category} (or \emph{quantum affine wreath category}) $ \mathcal{AH}(A,z) $ to be the strict $ \Bbbk $-linear monoidal supercategory generated by one object $ \go $ and morphisms
\begin{equation} \label{quantum affine wreath product category_ob}
    \begin{tikzpicture}[centerzero, thick]
        \draw[-] (0.2,-0.2) -- (-0.2,0.2);
        \draw[wipe] (-0.2,-0.2) -- (0.2,0.2);
        \draw[-] (-0.2,-0.2) -- (0.2,0.2);
     \end{tikzpicture}
     \ ,\
    \begin{tikzpicture}[centerzero, thick]
        \draw[-] (-0.2,-0.2) -- (0.2,0.2);
        \draw[wipe] (0.2,-0.2) -- (-0.2,0.2);
        \draw[-] (0.2,-0.2) -- (-0.2,0.2);
    \end{tikzpicture}
      \ \colon
    \go \otimes \go \to \go \otimes \go,
    \quad
    \begin{tikzpicture}[anchorbase, thick]
        \draw[-] (0,-0.3) -- (0,0.3);
        \token{0,0}{west}{a};
    \end{tikzpicture}
    \colon \go \to \go, \quad
 \begin{tikzpicture}[anchorbase, thick]
        \draw[-] (0,0) -- (0,0.6);
        \singdot{0,0.3};
\end{tikzpicture}
\ \colon
\go \to \go, \quad
\begin{tikzpicture}[anchorbase, thick]
    \draw[-] (0,0) -- (0,0.6);
    \multdot{0,0.3}{west}{-1};
\end{tikzpicture}
\ \colon
\go \to \go, \quad a \in A.
\end{equation}
We refer to the generators above as a \emph{positive black-black crossing}, a \emph{negative black-black crossing}, a \emph{token}, a \emph{dot}, and an \emph{inverse dot}. The relations are given by \cref{tokrel1}, \cref{P(A)_relations}, and
\begin{gather}
\label{braid}
\begin{tikzpicture}[anchorbase, thick]
        \draw[-] (0.2,-0.5) to[out=up,in=down] (-0.2,0) to[out=up,in=down] (0.2,0.5);
        \draw[wipe] (-0.2,-0.5) to[out=up,in=down] (0.2,0) to[out=up,in=down] (-0.2,0.5);
        \draw[-] (-0.2,-0.5) to[out=up,in=down] (0.2,0) to[out=up,in=down] (-0.2,0.5);
\end{tikzpicture}
\ =\
\begin{tikzpicture}[anchorbase, thick]
        \draw[-] (-0.2,-0.5) -- (-0.2,0.5);
        \draw[-] (0.2,-0.5) -- (0.2,0.5);
\end{tikzpicture}
\ =\
\begin{tikzpicture}[anchorbase, thick]
        \draw[-] (-0.2,-0.5) to[out=up,in=down] (0.2,0) to[out=up,in=down] (-0.2,0.5);
        \draw[wipe] (0.2,-0.5) to[out=up,in=down] (-0.2,0) to[out=up,in=down] (0.2,0.5);
        \draw[-] (0.2,-0.5) to[out=up,in=down] (-0.2,0) to[out=up,in=down] (0.2,0.5);
\end{tikzpicture}
\ ,\quad
\begin{tikzpicture}[anchorbase, thick]
        \draw[-] (0.4,-0.5) -- (-0.4,0.5);
        \draw[wipe] (0,-0.5) to[out=up, in=down] (-0.4,0) to[out=up,in=down] (0,0.5);
        \draw[-] (0,-0.5) to[out=up, in=down] (-0.4,0) to[out=up,in=down] (0,0.5);
        \draw[wipe] (-0.4,-0.5) -- (0.4,0.5);
        \draw[-] (-0.4,-0.5) -- (0.4,0.5);
\end{tikzpicture}
\ =\
\begin{tikzpicture}[anchorbase, thick]
        \draw[-] (0.4,-0.5) -- (-0.4,0.5);
        \draw[wipe] (0,-0.5) to[out=up, in=down] (0.4,0) to[out=up,in=down] (0,0.5);
        \draw[-] (0,-0.5) to[out=up, in=down] (0.4,0) to[out=up,in=down] (0,0.5);
        \draw[wipe] (-0.4,-0.5) -- (0.4,0.5);
        \draw[-] (-0.4,-0.5) -- (0.4,0.5);
\end{tikzpicture}
\ ,\quad
\begin{tikzpicture}[centerzero, thick]
        \draw[-] (0.3,-0.4) -- (-0.3,0.4);
        \draw[wipe] (-0.3,-0.4) -- (0.3,0.4);
        \draw[-] (-0.3,-0.4) -- (0.3,0.4);
        \token{-0.15,-0.2}{east}{a};
\end{tikzpicture}
\ =\
\begin{tikzpicture}[centerzero, thick]
        \draw[-] (0.3,-0.4) -- (-0.3,0.4);
        \draw[wipe] (-0.3,-0.4) -- (0.3,0.4);
        \draw[-] (-0.3,-0.4) -- (0.3,0.4);
        \token{0.15,0.2}{west}{a};
\end{tikzpicture}
\ , \quad a \in A,
\\ \label{preskein}
\begin{tikzpicture}[centerzero, thick]
        \draw[-] (0.3,-0.3) -- (-0.3,0.3);
        \draw[wipe] (-0.3,-0.3) -- (0.3,0.3);
        \draw[-] (-0.3,-0.3) -- (0.3,0.3);
\end{tikzpicture}
\ -\
\begin{tikzpicture}[centerzero, thick]
        \draw[-] (-0.3,-0.3) -- (0.3,0.3);
        \draw[wipe] (0.3,-0.3) -- (-0.3,0.3);
        \draw[-] (0.3,-0.3) -- (-0.3,0.3);
\end{tikzpicture}
= z \
\begin{tikzpicture}[centerzero, thick]
        \draw[-] (-0.2,-0.3) -- (-0.2,0.3);
        \draw[-] (0.2,-0.3) -- (0.2,0.3);
        \teleport{-0.2,0.1}{0.2,-0.1};
\end{tikzpicture} \ , 
\quad
\begin{tikzpicture}[centerzero, thick]
        \draw[-] (0.3,-0.4) -- (-0.3,0.4);
        \draw[wipe] (-0.3,-0.4) -- (0.3,0.4);
        \draw[-] (-0.3,-0.4) -- (0.3,0.4);
        \singdot{-0.15,-0.2};
\end{tikzpicture}
\ =\
\begin{tikzpicture}[centerzero, thick]
        \draw[-] (-0.3,-0.4) -- (0.3,0.4);
        \draw[wipe] (0.3,-0.4) -- (-0.3,0.4);
        \draw[-] (0.3,-0.4) -- (-0.3,0.4);
        \singdot{0.171,0.228};
\end{tikzpicture}
\ .
\end{gather}
Here, the morphism $ \begin{tikzpicture}[centerzero, thick]
        \draw[-] (-0.2,-0.3) -- (-0.2,0.3);
        \draw[-] (0.2,-0.3) -- (0.2,0.3);
        \teleport{-0.2,0.1}{0.2,-0.1};
\end{tikzpicture} $ is defined as in \cref{teleporter66}. The dot, inverse dot, and crossings are even, and the parity of a token $ \begin{tikzpicture}[anchorbase, thick]
        \draw[-] (0,-0.2) -- (0,0.2);
        \token{0,0}{west}{a};
    \end{tikzpicture} $ is equal to the parity of \(a\). For $ n \in \mathbb{N} $, we define the \emph{affine Frobenius Hecke algebra} (or \emph{quantum affine wreath algebra}) to be
\begin{equation*}
H_{n}^{\operatorname{aff}}(A,z) := \operatorname{End}_{\mathcal{AH}(A,z)}(\go^{\otimes n}).
\end{equation*}
\end{defn}

By \cref{braid}, \cref{preskein}, and \cref{teleporter1} (recall that $ \psi = 1 $ and $ \varepsilon = \bar{0} $ in this section), we have that
\begin{equation} \label{token crossing}
\begin{tikzpicture}[centerzero, thick]
        \draw[-] (-0.3,-0.4) -- (0.3,0.4);
        \draw[wipe] (0.3,-0.4) -- (-0.3,0.4);
        \draw[-] (0.3,-0.4) -- (-0.3,0.4);
        \token{-0.15,-0.2}{east}{a};
\end{tikzpicture}
\ =\
\begin{tikzpicture}[centerzero, thick]
        \draw[-] (-0.3,-0.4) -- (0.3,0.4);
        \draw[wipe] (0.3,-0.4) -- (-0.3,0.4);
        \draw[-] (0.3,-0.4) -- (-0.3,0.4);
        \token{0.15,0.2}{west}{a};
\end{tikzpicture}
\ , \quad
\begin{tikzpicture}[centerzero, thick]
        \draw[-] (0.3,-0.4) -- (-0.3,0.4);
        \draw[wipe] (-0.3,-0.4) -- (0.3,0.4);
        \draw[-] (-0.3,-0.4) -- (0.3,0.4);
        \token{0.15,-0.2}{west}{a};
\end{tikzpicture}
\ =\
\begin{tikzpicture}[centerzero, thick]
        \draw[-] (0.3,-0.4) -- (-0.3,0.4);
        \draw[wipe] (-0.3,-0.4) -- (0.3,0.4);
        \draw[-] (-0.3,-0.4) -- (0.3,0.4);
        \token{-0.15,0.2}{east}{a};
\end{tikzpicture}
\ , \quad
\begin{tikzpicture}[centerzero, thick]
        \draw[-] (-0.3,-0.4) -- (0.3,0.4);
        \draw[wipe] (0.3,-0.4) -- (-0.3,0.4);
        \draw[-] (0.3,-0.4) -- (-0.3,0.4);
        \token{0.15,-0.2}{west}{a};
\end{tikzpicture}
\ =\
\begin{tikzpicture}[centerzero, thick]
        \draw[-] (-0.3,-0.4) -- (0.3,0.4);
        \draw[wipe] (0.3,-0.4) -- (-0.3,0.4);
        \draw[-] (0.3,-0.4) -- (-0.3,0.4);
        \token{-0.15,0.2}{east}{a};
\end{tikzpicture}
\ .
\end{equation}
Furthermore, we have that
\begin{equation} \label{dot slide general}
\begin{tikzpicture}[centerzero, thick]
        \draw[-] (0.3,-0.4) -- (-0.3,0.4);
        \draw[wipe] (-0.3,-0.4) -- (0.3,0.4);
        \draw[-] (-0.3,-0.4) -- (0.3,0.4);
        \singdot{-0.15,0.2};
\end{tikzpicture}
\ =\
\begin{tikzpicture}[centerzero, thick]
        \draw[-] (-0.3,-0.4) -- (0.3,0.4);
        \draw[wipe] (0.3,-0.4) -- (-0.3,0.4);
        \draw[-] (0.3,-0.4) -- (-0.3,0.4);
        \singdot{0.171,-0.228};
\end{tikzpicture}
\ , \quad
\begin{tikzpicture}[centerzero, thick]
        \draw[-] (0.3,-0.4) -- (-0.3,0.4);
        \draw[wipe] (-0.3,-0.4) -- (0.3,0.4);
        \draw[-] (-0.3,-0.4) -- (0.3,0.4);
        \multdot{0.171,0.228}{west}{-1};
\end{tikzpicture}
\ =\
\begin{tikzpicture}[centerzero, thick]
        \draw[-] (-0.3,-0.4) -- (0.3,0.4);
        \draw[wipe] (0.3,-0.4) -- (-0.3,0.4);
        \draw[-] (0.3,-0.4) -- (-0.3,0.4);
        \multdot{-0.15,-0.2}{east}{-1};
\end{tikzpicture}
\ ,\quad
\begin{tikzpicture}[centerzero, thick]
        \draw[-] (0.3,-0.4) -- (-0.3,0.4);
        \draw[wipe] (-0.3,-0.4) -- (0.3,0.4);
        \draw[-] (-0.3,-0.4) -- (0.3,0.4);
        \multdot{0.171,-0.228}{west}{-1};
\end{tikzpicture}
\ =\
\begin{tikzpicture}[centerzero, thick]
        \draw[-] (-0.3,-0.4) -- (0.3,0.4);
        \draw[wipe] (0.3,-0.4) -- (-0.3,0.4);
        \draw[-] (0.3,-0.4) -- (-0.3,0.4);
        \multdot{-0.15,0.2}{east}{-1};
\end{tikzpicture}
\ .
\end{equation}
Note that, in \cite[Def.~2.2]{Rosso-Savage}, the affine Frobenius Hecke algebra was defined via generators and relations rather than as an endomorphism algebra of a category. However, it follows from Proposition~\ref{zebra} that the two algebras are indeed isomorphic.

\section{Higher-level affine Frobenius Hecke algebras} \label{HLAtitle}

In this section, we define the higher-level affine Frobenius Hecke algebras. Throughout this section, we fix a scalar $ z \in \Bbbk $, and we assume that the Frobenius algebra \(A\) is symmetric and carries a parity-preserving trace map. In other words, we have $ \psi = 1 $ and $ \varepsilon = \bar{0} $. The main case of interest that these assumptions on \(A\) exclude is the Clifford superalgebra from Example \ref{Cliff_parity_pres} and Example \ref{Cliff_parity_rev}, where the higher-level affine Frobenius Hecke algebra defined below would correspond to a higher-level version of the affine Hecke-Clifford superalgebra. We hope to explore the Clifford superalgebra case in the future.

\subsection{Higher-level affine Frobenius Hecke algebras}

Recall that, for an algebra \(F\), we have the sets $ \mathcal{E}_{F}  $ and $ \widehat{\mathcal{E}}_{F} $ from Definition \ref{multi_pen}.

\begin{defn} \label{HLAFHC}
We define the \emph{higher-level affine Frobenius Hecke category}  $ \mathcal{LAH}(A,z) $ to be the strict $ \Bbbk $-linear monoidal supercategory defined as follows. The objects are generated by the elements in the set $ \widehat{\mathcal{E}}_{P_{1}(A)} $. The morphisms are generated by
\begin{gather}
\begin{tikzpicture}[centerzero, thick]
        \draw[-] (0.2,-0.2) -- (-0.2,0.2);
        \draw[wipe] (-0.2,-0.2) -- (0.2,0.2);
        \draw[-] (-0.2,-0.2) -- (0.2,0.2);
     \end{tikzpicture}
     \ ,\
    \begin{tikzpicture}[centerzero, thick]
        \draw[-] (-0.2,-0.2) -- (0.2,0.2);
        \draw[wipe] (0.2,-0.2) -- (-0.2,0.2);
        \draw[-] (0.2,-0.2) -- (-0.2,0.2);
    \end{tikzpicture}
      \ \colon
    \go \otimes \go \to \go \otimes \go,
    \quad
    \begin{tikzpicture}[anchorbase, thick]
        \draw[-] (0,-0.3) -- (0,0.3);
        \token{0,0}{west}{a};
    \end{tikzpicture}
    \colon \go \to \go, \quad
 \begin{tikzpicture}[anchorbase, thick]
        \draw[-] (0,0) -- (0,0.6);
        \singdot{0,0.3};
\end{tikzpicture}
\ \colon
\go \to \go, \label{Gen1_quantum} 
\quad
\begin{tikzpicture}[anchorbase, thick]
    \draw[-] (0,0) -- (0,0.6);
    \multdot{0,0.3}{west}{-1};
\end{tikzpicture}
\ \colon
\go \to \go,
\\
\begin{tikzpicture}[centerzero, thick]
        \draw[-] (0.2,-0.2) -- (-0.2,0.2);
        \draw[-, wei] (-0.2,-0.2) -- (0.2,0.2);
        \node at (-0.2,-0.4) {\tiny $ Q $};
\end{tikzpicture}
\ \colon Q \otimes \go \to \go \otimes Q \qquad
\begin{tikzpicture}[centerzero, thick]
        \draw[-] (-0.2,-0.2) -- (0.2,0.2);
        \draw[-,wei] (0.2,-0.2) -- (-0.2,0.2);
        \node at (0.2,-0.4) {\tiny $ Q $};
\end{tikzpicture}
\colon \go \otimes Q \to Q \otimes \go, \label{Gen2_quantum}
\end{gather}
for all $ a \in A $, $ Q \in \mathcal{E}_{P_{1}(A)} $. The crossings and dot are even, and the parity of a token $ \begin{tikzpicture}[anchorbase, thick]
        \draw[-] (0,-0.2) -- (0,0.2);
        \token{0,0}{west}{a};
    \end{tikzpicture} $ is equal to the parity of \(a\). The relations on the morphisms are given by \cref{tokrel1}, \cref{braid}, \cref{preskein}, and
\begin{gather}
\begin{tikzpicture}[centerzero, thick]
        \draw[-] (-0.3,-0.4) -- (0.3,0.4);
        \draw[-, wei] (0.3,-0.4) -- (-0.3,0.4);
        \token{-0.15,-0.2}{east}{a};
        \node at (0.3,-0.6) {\tiny $ Q $};
\end{tikzpicture}
= 
\begin{tikzpicture}[centerzero, thick]
        \draw[-] (-0.3,-0.4) -- (0.3,0.4);
        \draw[-, wei] (0.3,-0.4) -- (-0.3,0.4);
        \token{0.15,0.2}{west}{a};
        \node at (0.3,-0.6) {\tiny $ Q $};
\end{tikzpicture}
\ , \quad
\begin{tikzpicture}[centerzero, thick]
        \draw[-] (0.3,-0.4) -- (-0.3,0.4);
        \draw[-, wei] (-0.3,-0.4) -- (0.3,0.4);
        \token{0.15,-0.2}{west}{a};
        \node at (-0.3,-0.6) {\tiny $ Q $};
\end{tikzpicture}
= 
\begin{tikzpicture}[centerzero, thick]
        \draw[-] (0.3,-0.4) -- (-0.3,0.4);
        \draw[-, wei] (-0.3,-0.4) -- (0.3,0.4);
        \token{-0.15,0.2}{east}{a};
        \node at (-0.3,-0.6) {\tiny $ Q $};
\end{tikzpicture} \ , \label{hw11}
\\
\begin{tikzpicture}[centerzero, thick]
        \draw[-] (-0.3,-0.4) -- (0.3,0.4);
        \draw[-, wei] (0.3,-0.4) -- (-0.3,0.4);
        \singdot{-0.15,-0.2};
        \node at (0.3,-0.6) {\tiny $ Q $};
\end{tikzpicture}
= 
\begin{tikzpicture}[centerzero, thick]
        \draw[-] (-0.3,-0.4) -- (0.3,0.4);
        \draw[-, wei] (0.3,-0.4) -- (-0.3,0.4);
        \singdot{0.15,0.2};
        \node at (0.3,-0.6) {\tiny $ Q $};
\end{tikzpicture}
\ , \quad
\begin{tikzpicture}[centerzero, thick]
        \draw[-] (0.3,-0.4) -- (-0.3,0.4);
        \draw[-, wei] (-0.3,-0.4) -- (0.3,0.4);
        \singdot{0.15,-0.2};
        \node at (-0.3,-0.6) {\tiny $ Q $};
\end{tikzpicture}
= 
\begin{tikzpicture}[centerzero, thick]
        \draw[-] (0.3,-0.4) -- (-0.3,0.4);
        \draw[-, wei] (-0.3,-0.4) -- (0.3,0.4);
        \singdot{-0.15,0.2};
        \node at (-0.3,-0.6) {\tiny $ Q $};
\end{tikzpicture} \ , \label{hw21}
\\ 
\begin{tikzpicture}[centerzero, thick]
        \draw[-] (-0.2,-0.5) to[out=up,in=down] (0.2,0) to[out=up,in=down] (-0.2,0.5);
        \draw[-, wei] (0.2,-0.5) to[out=up,in=down] (-0.2,0) to[out=up,in=down] (0.2,0.5);
        \node at (0.2,-0.7) {\tiny $ Q $};
\end{tikzpicture}
=
\begin{tikzpicture}[centerzero, thick]
        \pin{-0.3,0}{-1.1,0}{Q};
        \draw[-] (-0.3,-0.5) -- (-0.3,0.5);
        \draw[-, wei] (0.3,-0.5) -- (0.3,0.5);
        \node at (0.3,-0.7) {\tiny $ Q $};
\end{tikzpicture} \ , \label{hw31}
\quad
\begin{tikzpicture}[centerzero, thick]
        \draw[-] (0.2,-0.5) to[out=up,in=down] (-0.2,0) to[out=up,in=down] (0.2,0.5);
        \draw[-, wei] (-0.2,-0.5) to[out=up,in=down] (0.2,0) to[out=up,in=down] (-0.2,0.5);
        \node at (-0.2,-0.7) {\tiny $ Q $};
\end{tikzpicture}
\ = \
\begin{tikzpicture}[centerzero, thick]
        \draw[-, wei] (-0.3,-0.5) -- (-0.3,0.5);
        \draw[-] (0.3,-0.5) -- (0.3,0.5);
        \node at (-0.3,-0.7) {\tiny $ Q $};
        \pin{0.3,0}{1.1,0}{Q};
\end{tikzpicture} \ , 
\\ 
\begin{tikzpicture}[centerzero, thick]
        \draw[-] (0.4,-0.5) -- (-0.4,0.5);
        \draw[wipe] (0,-0.5) to[out=up, in=down] (-0.4,0) to[out=up,in=down] (0,0.5);
        \draw[-] (0,-0.5) to[out=up, in=down] (-0.4,0) to[out=up,in=down] (0,0.5);
        \draw[wipe] (-0.4,-0.5) -- (0.4,0.5);
        \draw[-, wei] (-0.4,-0.5) -- (0.4,0.5);
        \node at (-0.4,-0.7) {\tiny $ Q $};
\end{tikzpicture}
\ =\
\begin{tikzpicture}[centerzero, thick]
        \draw[-] (0.4,-0.5) -- (-0.4,0.5);
        \draw[wipe] (0,-0.5) to[out=up, in=down] (0.4,0) to[out=up,in=down] (0,0.5);
        \draw[-] (0,-0.5) to[out=up, in=down] (0.4,0) to[out=up,in=down] (0,0.5);
        \draw[wipe] (-0.4,-0.5) -- (0.4,0.5);
        \draw[-, wei] (-0.4,-0.5) -- (0.4,0.5);
        \node at (-0.4,-0.7) {\tiny $ Q $};
\end{tikzpicture} \ , \label{hw51}
\quad
\begin{tikzpicture}[centerzero, thick]
        \draw[-] (0,-0.5) to[out=up, in=down] (-0.4,0) to[out=up,in=down] (0,0.5);
        \draw[wipe] (-0.4,-0.5) -- (0,0);
        \draw[-] (-0.4,-0.5) -- (0.4,0.5);
        \draw[-, wei] (0.4,-0.5) -- (-0.4,0.5);
        \node at (0.4,-0.7) {\tiny $ Q $};
\end{tikzpicture}
\ =\
\begin{tikzpicture}[centerzero, thick]
        \draw[wipe] (0,-0.5) to[out=up, in=down] (0.4,0) to[out=up,in=down] (0,0.5);
        \draw[-] (0,-0.5) to[out=up, in=down] (0.4,0) to[out=up,in=down] (0,0.5);
        \draw[wipe] (-0.4,-0.5) -- (0.4,0.5);
        \draw[-] (-0.4,-0.5) -- (0.4,0.5);
        \draw[-, wei] (0.4,-0.5) -- (-0.4,0.5);
        \node at (0.4,-0.7) {\tiny $ Q $};
\end{tikzpicture}
\ , 
\\ \label{green1}
\begin{tikzpicture}[centerzero, thick]
        \draw[-] (0.4,-0.5) -- (-0.4,0.5);
        \draw[wipe] (-0.4,-0.5) -- (0.4,0.5);
        \draw[-] (-0.4,-0.5) -- (0.4,0.5);
        \node at (0,-0.7) {\tiny $ Q $};
        \draw[-, wei] (0,-0.5) to[out=up, in=down] (-0.4,0) to[out=up,in=down] (0,0.5);
\end{tikzpicture}
\ =\
\begin{tikzpicture}[centerzero, thick]
        \draw[-] (2,-0.5) -- (1.2,0.5);
        \draw[wipe] (1.2,-0.5) -- (2,0.5);
        \draw[-] (1.2,-0.5) -- (2,0.5);
        \node at (1.6,-0.7) {\tiny $ Q $};
        \draw[-, wei] (1.6,-0.5) to[out=up, in=down] (2,0) to[out=up,in=down] (1.6,0.5);
\end{tikzpicture}
\ + z \
\begin{tikzpicture}[centerzero, thick]
        \node at (0,-0.7) {\tiny $ Q $};
        \draw[-] (-0.6,-0.5) -- (-0.6,0.5);
        \draw[-, wei] (0,-0.5) -- (0,0.5);
        \draw[-] (0.6,-0.5) -- (0.6,0.5);
        \pinpin{0.6,0}{-0.6,0}{1.8,0}{\Delta_{1}(Q_{1})};
\end{tikzpicture} \ ,
\end{gather}
for all $ a \in A $, $ Q \in \mathcal{E}_{P_{1}(A)} $. 
\end{defn}

Recall that, for a given set \(Z\), $ F(Z) $ is the free monoid on \(Z\). Then the set of objects in $ \mathcal{LAH}(A,z) $ is $ F(\widehat{\mathcal{E}}_{P_{1}(A)}) $. 

\begin{defn} \label{fluffy}
Given a positive integer $ d \in \mathbb{N}_{+} $ and a word $ \mathbf{Q} \in F(\mathcal{E}_{P_{1}(A)}) \subseteq F(\widehat{\mathcal{E}}_{P_{1}(A)}) $, we define $ \Lambda_{d,\mathbf{Q}} $ to be the set of all $ (\go^{d},\mathbf{Q}) $-shuffles in $ F(\widehat{\mathcal{E}}_{P_{1}(A)}) $. That is, 
\begin{equation} \label{fluffy2_quantum}
\Lambda_{d,\mathbf{Q}} = F(\widehat{\mathcal{E}}_{P_{1}(A)})_{\go^{d},\mathbf{Q}},
\end{equation}
where $ F(\widehat{\mathcal{E}}_{P_{1}(A)})_{\go^{d},\mathbf{Q}} $ is given in Definition \ref{snowing}.
\end{defn}

\begin{defn} \label{HLAFHA2}
Let $ d \in \mathbb{N} $ and $ \mathbf{Q} \in F(\mathcal{E}_{P_{1}(A)}) $. Then we define the $ (d,\mathbf{Q}) $-\emph{affine Frobenius Hecke algebra} to be 
\begin{equation} \label{HLAFHA222}
H_{d,\mathbf{Q}}^{\operatorname{aff}}(A,z) := \operatorname{End}_{\operatorname{Add}(\mathcal{LAH}(A,z))}\left(\bigoplus_{\mathbf{i} \in \Lambda_{d,\mathbf{Q}} }\mathbf{i}\right).
\end{equation}
Here, $ \operatorname{Add}(\mathcal{LAH}(A,z)) $ is the additive envelope of $ \mathcal{LAH}(A,z) $. We say that $ |\mathbf{Q}| $ is the \emph{level} of $ H_{d,\mathbf{Q}}^{\operatorname{aff}}(A,z) $.
\end{defn}

Using language similar to that of \cite{Maksimau-Stroppel}, we will often refer to the algebra $ H_{d,\mathbf{Q}}^{\operatorname{aff}}(A,z) $ as a \emph{higher-level affine Frobenius Hecke algebra}. 

\begin{rmk} \label{A=K_case_quantum}
If $ \mathbf{Q} \in F(\mathcal{E}_{P_{1}(A)}) $ is the empty word, then the diagrams in $ H_{d,\mathbf{Q}}^{\operatorname{aff}}(A,z) $ contain only black strands. It will follow from the basis theorem (Theorem \ref{HL-basis_quantum}) that, in this case, $ H_{d,\mathbf{Q}}^{\operatorname{aff}}(A,z) $ is isomorphic to the affine Frobenius Hecke algebra $ H_{d}^{\operatorname{aff}}(A,z)$ from Definition \ref{quantum affine wreath product category}.
\end{rmk}

\begin{rmk} \label{usualaff}
Suppose $ A = \Bbbk $, and assume that $ z = \mathbf{q}-\mathbf{q}^{-1} $ for some $ \mathbf{q} \in \Bbbk^{\times} $. Then under certain assumptions on the word $ \mathbf{Q} $, the algebra $ H_{d,\mathbf{Q}}^{\operatorname{aff}}(\Bbbk,z) $ is isomorphic to the higher-level affine Hecke algebra from \cite[Def.~2.7]{Maksimau-Stroppel}. Here, one should take $ q = \mathbf{q}^{2} $ in \cite[Def.~2.7]{Maksimau-Stroppel}, and one should identify their $ \begin{tikzpicture}[centerzero, thick]
        \draw[-] (0.2,-0.2) -- (-0.2,0.2);
        \draw[-] (-0.2,-0.2) -- (0.2,0.2);
     \end{tikzpicture} $ 
with our 
$ \mathbf{q} \begin{tikzpicture}[centerzero, thick]
        \draw[-] (0.2,-0.2) -- (-0.2,0.2);
        \draw[wipe] (-0.2,-0.2) -- (0.2,0.2);
        \draw[-] (-0.2,-0.2) -- (0.2,0.2);
     \end{tikzpicture} $.
\end{rmk}

\begin{rmk}
As seen in Proposition \ref{ind_trace_map} and Corollary \ref{ind_trace_map_2}, the higher-level affine wreath product category and higher-level affine wreath product algebras depend only, up to isomorphism, on the underlying algebra \(A\), and not on the choice of trace map. However, in the quantum setting of the current section, there do not seem to be obvious isomorphisms between higher-level affine Frobenius Hecke algebras corresponding to the same algebra but with different trace maps. This is even true for the affine Frobenius Hecke algebras; see \cite[Rmk.~2.2]{Rosso-Savage}.
\end{rmk}

\subsection{Basis theorem} \label{basisthm_quantum}

In this section, we find bases of the morphism spaces of $ \mathcal{LAH}(A,z) $.

\begin{lem} \label{zeromostly2_quantum}
Let $ \mathbf{i}, \mathbf{j} \in F(\widehat{\mathcal{E}}_{P_{1}(A)}) $ be two objects in $ \mathcal{LAH}(A,z) $. Then we have $ \operatorname{Hom}_{\mathcal{LAH}(A,z)}(\mathbf{i},\mathbf{j}) = 0 $ unless $ \mathbf{i}, \mathbf{j} \in \Lambda_{d,\mathbf{Q}} $ for some $ d \in \mathbb{N} $, $ \mathbf{Q} \in F(\mathcal{E}_{P_{1}(A)}) $.
\end{lem}

\begin{proof}
This follows from the forms of the generating morphisms of $ \mathcal{LAH}(A,z) $ in \cref{Gen1_quantum} and \cref{Gen2_quantum}.
\end{proof}

In light of Lemma \ref{zeromostly2_quantum}, for the rest of this section, we fix an integer $ d \in \mathbb{N} $ and a word $ \mathbf{Q} \in F(\mathcal{E}_{P_{1}(A)}) $. We will give a basis of $ \operatorname{Hom}_{\mathcal{LAH}(A,z)}(\mathbf{i}, \mathbf{j}) $ for all $ \mathbf{i}, \mathbf{j} \in \Lambda_{d,\mathbf{Q}} $.

\begin{defn} \label{Generalized_sigma_quantum}
For each $ \mathbf{i}, \mathbf{j} \in \Lambda_{d,\mathbf{Q}} $, and each permutation $ w \in S_{d} $, we choose a diagram $ T_{\mathbf{j},w,\mathbf{i}} \in \operatorname{Hom}_{\mathcal{LAH}(A,z)}(\mathbf{i}, \mathbf{j}) $ with the following properties:
\begin{enumerate}
\item[\textbullet]  The sequence of transpositions on the black strands of $ T_{\mathbf{j},w,\mathbf{i}} $ is a reduced expression for \(w\). 
\item[\textbullet] Each pair of red and black strands cross at most once.
\item[\textbullet] The diagram $ T_{\mathbf{j},w,\mathbf{i}} $ is tokenless and dotless, and contains no negative black-black crossings.
\end{enumerate}
\end{defn}

\begin{egg}
Suppose $ d = 3 $ and $ \mathbf{Q} = QQ' $ for some $ Q,Q' \in \mathcal{E}_{P_{1}(A)} $. Let $ \mathbf{i} := \go Q \go Q' \go \in \Lambda_{d,\mathbf{Q}} $ and $ \mathbf{j} := \go Q Q' \go \go \in \Lambda_{d,\mathbf{Q}} $. Set $ w = s_{2}s_{1} \in S_{3} $. Then two possible choices for the element $ T_{\mathbf{j},w,\mathbf{i}} $ are given below:
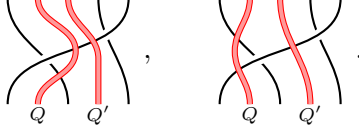
\begin{equation*}
\begin{tikzpicture}[centerzero, thick]
       \draw[-] (0.8,-0.7) to[out=up, in=down] (0,0.7);
       \draw[wipe] (0,-0.7) to[out=up, in=down] (1.6,0.7);
       \draw[-] (1.6,-0.7) to[out=up, in=down] (1.2,0.7);
       \draw[wipe] (0,-0.7) to[out=up, in=down] (1.6,0.7);       
       \draw[-] (0,-0.7) to[out=up, in=down] (1.6,0.7);
       \draw[-,wei] (0.4,-0.7) to[out=up, in=down] (0.9,0) to[out=up, in=down] (0.4,0.7);
       \draw[-,wei] (1.2,-0.7) to[out=up, in=down] (1.2,-0.1) to[out=up, in=down] (0.8,0.7);
       \node at (0.4,-0.85) {\tiny $ Q $};
       \node at (1.2,-0.85) {\tiny $ Q' $};
\end{tikzpicture} \ , 
\qquad
\begin{tikzpicture}[centerzero, thick]
       \draw[-] (0.8,-0.7) to[out=up, in=down] (0,0.7);
       \draw[wipe] (0,-0.7) to[out=up, in=down] (1.6,0.7);
       \draw[-] (1.6,-0.7) to[out=up, in=down] (1.2,0.7);
       \draw[wipe] (0,-0.7) to[out=up, in=down] (1.6,0.7);  
       \draw[-] (0,-0.7) to[out=up, in=down] (1.6,0.7);
       \draw[-,wei] (0.4,-0.7) to[out=up, in=down] (0.2,0) to[out=up, in=down] (0.4,0.7);
       \draw[-,wei] (1.2,-0.7) to[out=up, in=down] (0.8,0.7);
       \node at (0.4,-0.85) {\tiny $ Q $};
       \node at (1.2,-0.85) {\tiny $ Q' $};
\end{tikzpicture} \ .
\end{equation*}
\end{egg}

We now define the following subset of $ \operatorname{Hom}_{\mathcal{LAH}(A,z)}(\mathbf{i}, \mathbf{j}) $:
\begin{equation}
_{\mathbf{j}} \mathcal{C}_{\mathbf{i}} := \{\mathbf{a}_{\mathbf{j}}\mathbf{x}_{\mathbf{j}}^{\alpha}T_{\mathbf{j},w,\mathbf{i}} : \mathbf{a} \in B^{\otimes d}, \ \alpha \in \mathbb{Z}^{d}, \ w \in S_{d} \}.
\end{equation}

\begin{theo} \label{HL-basis_quantum}
Let $ \mathbf{i}, \mathbf{j} \in \Lambda_{d,\mathbf{Q}} $. Then the set $ _{\mathbf{j}} \mathcal{C}_{\mathbf{i}} $ is a $ \Bbbk $-basis of $ \operatorname{Hom}_{\mathcal{LAH}(A,z)}(\mathbf{i}, \mathbf{j}) $.
\end{theo}

\begin{proof}
The proof of this theorem is analogous to the proof of Theorem~\ref{HL-basis}. Hence we omit the details here.
\end{proof}

\details{We explain here how one obtains linear independence of the elements in $ _{\mathbf{j}} \mathcal{C}_{\mathbf{i}} $. There is a strict $ \Bbbk $-linear monoidal functor
\begin{equation*}
\Omega \colon \mathcal{LAH}(A,z) \rightarrow \mathcal{AH}(A,z)
\end{equation*}
which sends $ \go $ to $ \go $ and $ Q \in \mathcal{E}_{P_{1}(A)}  $ to $ \mathds{1} $ (the unit object of $ \mathcal{AH}(A,z) $), and sends the generating morphisms as follows:
\begin{gather*}
\Omega \left(\begin{tikzpicture}[centerzero, thick]
      \draw[-] (0.6,-0.4) -- (0.6,0.4);
      \token{0.6,0}{west}{a};
\end{tikzpicture}\right) = \begin{tikzpicture}[centerzero, thick]
      \draw[-] (0.6,-0.4) -- (0.6,0.4);
      \token{0.6,0}{west}{a};
\end{tikzpicture}, 
\quad 
\Omega \left(\begin{tikzpicture}[centerzero, thick]
      \draw[-] (0.6,-0.4) -- (0.6,0.4);
      \singdot{0.6,0};
\end{tikzpicture}\right) = 
\begin{tikzpicture}[centerzero, thick]
      \draw[-] (0.6,-0.4) -- (0.6,0.4);
      \singdot{0.6,0};
\end{tikzpicture} \ ,
\quad 
\Omega \left(\begin{tikzpicture}[centerzero, thick]
      \draw[-] (0.9,-0.4) -- (0.3,0.4);
       \draw[wipe] (0.3,-0.4) -- (0.9,0.4);
      \draw[-] (0.3,-0.4) -- (0.9,0.4);
\end{tikzpicture}\right) 
=
\begin{tikzpicture}[centerzero, thick]
      \draw[-] (0.9,-0.4) -- (0.3,0.4);
      \draw[wipe] (0.3,-0.4) -- (0.9,0.4);
      \draw[-] (0.3,-0.4) -- (0.9,0.4);
\end{tikzpicture} \ ,
\quad 
\Omega \left(\begin{tikzpicture}[centerzero, thick]
      \draw[-] (0.3,-0.4) -- (0.9,0.4);
      \draw[wipe] (0.9,-0.4) -- (0.3,0.4);
      \draw[-] (0.9,-0.4) -- (0.3,0.4);
\end{tikzpicture}\right) 
=
\begin{tikzpicture}[centerzero, thick]
      \draw[-] (0.3,-0.4) -- (0.9,0.4);
      \draw[wipe] (0.9,-0.4) -- (0.3,0.4);
      \draw[-] (0.9,-0.4) -- (0.3,0.4);
\end{tikzpicture} \ ,
\\ \Omega \left(\begin{tikzpicture}[centerzero, thick]
      \draw[-] (0.9,-0.4) -- (0.3,0.4);
      \draw[-,wei] (0.3,-0.4) -- (0.9,0.4);
      \node at (0.3,-0.6) {\tiny $ Q $};
\end{tikzpicture}\right) 
=
\begin{tikzpicture}[centerzero, thick]
      \draw[-] (0.6,-0.4) -- (0.6,0.4);
\end{tikzpicture} \ ,
\quad 
\Omega \left(\begin{tikzpicture}[centerzero, thick]
      \draw[-] (0.3,-0.4) -- (0.9,0.4);
      \draw[-,wei] (0.9,-0.4) -- (0.3,0.4);
      \node at (0.9,-0.6) {\tiny $ Q $};
\end{tikzpicture}\right) 
= 
\begin{tikzpicture}[centerzero, thick]
        \draw[-] (0,-0.5) -- (0,0.5);
             \pin{0,0}{.7,0}{Q};
\end{tikzpicture} \ ,
\end{gather*}
for all $ a \in A $, $ Q \in \mathcal{E}_{P_{1}(A)} $. This functor restricts to a $ \Bbbk $-linear map 
\begin{equation*}
\Omega \colon \operatorname{Hom}_{\mathcal{LAH}(A,z)}(\mathbf{i}, \mathbf{j}) \rightarrow H_{d}^{\operatorname{aff}}(A,z).
\end{equation*}
Then
\begin{equation} \label{cactus_quantum}
\Omega(T_{\mathbf{j},w,\mathbf{i}}) = h_{w,w}T_{w} + \sum_{\substack{u \in S_{d} \\ L(u) < L(w)}} h_{u,w}T_{u}
\end{equation}
for some $ h_{u,w} \in P_{d}(A) $, where $ h_{w,w} $ is regular in $ P_{d}(A) $. Now assume that there is a linear dependence equation
\begin{equation} \label{depequation3975_quantum}
\sum_{w \in S_{d}} f_{w}T_{\mathbf{j},w,\mathbf{i}} = 0,
\end{equation}
where $ f_{w} \in P_{d}(A) $. Then, by using \cref{cactus_quantum} and \cite[Cor~3.11]{Rosso-Savage}, one can compute the image of \cref{depequation3975_quantum} under $ \Omega $ and then deduce $ f_{w} = 0 $ for all $ w \in S_{d} $.}

\details{We prove here that the set $ _{\mathbf{j}} \mathcal{C}_{\mathbf{i}} $ spans $ \operatorname{Hom}_{\mathcal{LAH}(A,z)}(\mathbf{i}, \mathbf{j}) $. We prove this by induction on the total number of crossings. If \(D\) is a diagram with no crossings, then \(D\) contains only dots and tokens, from which it clearly follows that \(D\) lies in the span of $ _{\mathbf{j}} \mathcal{C}_{\mathbf{i}} $.
\\ \indent Now let $ k > 0 $, and let \(D\) be a diagram with \(k\) crossings. First, by using the relations \cref{braid,preskein,token crossing,dot slide general}, one can move the dots and tokens to the top of \(D\) modulo terms with fewer crossings, which all lie in the span of $ _{\mathbf{j}} \mathcal{C}_{\mathbf{i}} $ by the induction hypothesis. Furthermore, one can apply \cref{preskein} to remove any negative black-black crossings in \(D\). Now, if there are two strands in \(D\) that cross more than once, then \(D\) can be written as a $ \Bbbk $-linear combination of diagrams with fewer than $ k $ crossings. The proof of this fact is analogous to the proof of \cite[Lem.~4.10(3)]{Webster}. So we may from here on assume that any two strands in \(D\) cross at most once. In this case, the sequence of black strands on \(D\) corresponds to some reduced expression for some element $ w \in S_{d} $. We can now apply the relations \cref{hw51}, \cref{green1}, and the second relation in \cref{braid} to get from \(D\) to $ T_{\mathbf{j},w,\mathbf{i}} $. Note again that this process may create additional terms with fewer crossings (due to \cref{green1}), but these terms all lie in the span of $ _{\mathbf{j}} \mathcal{C}_{\mathbf{i}} $ by the induction hypothesis. Thus we have that \(D\) lies in the span of $ _{\mathbf{j}} \mathcal{C}_{\mathbf{i}} $.}

Recall that, for an integer $ n \in \mathbb{N}_{+} $ and a $ \Bbbk $-algebra \(F\), we define $ M_{n}(F) $ to be the algebra of $ n \times n $ matrices with entries in \(F\). Then we have the following analogue of Corollary~\ref{apple_tree}.

\begin{prop}
Let $ n := |\Lambda_{d,\mathbf{Q}}| $. Then the algebra $ H_{d,\mathbf{Q}}^{\operatorname{aff}}(A,z) $ is isomorphic to a subalgebra of $ M_{n}(H_{d}^{\operatorname{aff}}(A,z)) $.
\end{prop}

\details{The faithful functor $ \Omega \colon \mathcal{LAH}(A,z) \rightarrow \mathcal{AH}(A,z) $ given above induces an injective algebra homomorphism
\begin{equation*}
H_{d,\mathbf{Q}}^{\operatorname{aff}}(A,z) \quad \stackrel{\mathclap{\cref{HLAFHA222}}}{=} \quad \operatorname{End}_{\operatorname{Add}(\mathcal{LAH}(A,z))}\left(\bigoplus_{\mathbf{i} \in \Lambda_{d,\mathbf{Q}}}\mathbf{i}\right) \rightarrow \operatorname{End}_{\operatorname{Add}(\mathcal{AH}(A,z))}\left(\bigoplus_{i=1}^{n} \go^{\otimes d}\right) = M_{n}(H_{d}^{\operatorname{aff}}(A,z)).
\end{equation*}
The result now follows.}

\subsection{Description of the center} 

In this section, we fix a positive integer $ d \in \mathbb{N}_{+} $ and a word $ \mathbf{Q} \in F(\mathcal{E}_{P_{1}(A)}) $. We will state the center of the algebra $ H_{d,\mathbf{Q}}^{\operatorname{aff}}(A,z) $, in the sense of \cref{centerdef}. By Theorem \ref{HL-basis_quantum}, we have an injective algebra homomorphism $ P_{d}(A) \rightarrow H_{d,\mathbf{Q}}^{\operatorname{aff}}(A,z) $ given by
\begin{equation} \label{lifeguard_quantum}
\mathbf{a} \mapsto \sum_{\mathbf{i} \in \Lambda_{d,\mathbf{Q}}} \mathbf{a}_{\mathbf{i}}, \qquad x_{i} \mapsto \sum_{\mathbf{i} \in \Lambda_{d,\mathbf{Q}}} x_{i,\mathbf{i}}.
\end{equation}
We view $ P_{d}(A) $ as a subalgebra of $ H_{d,\mathbf{Q}}^{\operatorname{aff}}(A,z) $ under this algebra homomorphism. We also view $ Z(P_{d}(A))^{S_{d}} $ and $ Z(P_{d}(A)) $ as subalgebras of $ H_{d,\mathbf{Q}}^{\operatorname{aff}}(A,z) $ under the inclusions 
\begin{equation}
Z(P_{d}(A))^{S_{d}} \subseteq Z(P_{d}(A)) \subseteq P_{d}(A) \subseteq H_{d,\mathbf{Q}}^{\operatorname{aff}}(A,z).
\end{equation}

\begin{prop} \label{higher_center_quantum}
The center of $ H_{d,\mathbf{Q}}^{\operatorname{aff}}(A,z) $ is equal to $ Z(P_{d}(A))^{S_{d}} = P_{d}(Z(A))^{S_{d}} $.
\end{prop}

\begin{proof}
The proof of this result is analogous to the proof of Proposition \ref{higher_center}. Hence it will be omitted.
\end{proof}

\details{For $ \mathbf{i}, \mathbf{j} \in \Lambda_{d,\mathbf{Q}} $, we define
\begin{equation*}
_{\mathbf{i}}H_{d,\mathbf{Q}}^{\operatorname{aff}}(A,z)_{\mathbf{j}} = 1_{\mathbf{i}}H_{d,\mathbf{Q}}^{\operatorname{aff}}(A,z)1_{\mathbf{j}},
\end{equation*}
where $ 1_{\mathbf{i}} \in \operatorname{End}_{\mathcal{LAH}(A,z)}(\mathbf{i}) $ and $ 1_{\mathbf{j}} \in \operatorname{End}_{\mathcal{LAH}(A,z)}(\mathbf{j}) $ are the identity morphisms. Define $ \boldsymbol{\omega} := \mathbf{Q}\go^{d} \in \Lambda_{d,\mathbf{Q}} $ to be the concatenation of the words $ \mathbf{Q} $ and $ \go^{d} $. Then, as in Lemma \ref{flour}, we have an isomorphism of algebras
\begin{equation} \label{flour_quantum}
H_{d}^{\operatorname{aff}}(A,z) \cong {_{\boldsymbol{\omega}}}H_{d,\mathbf{Q}}^{\operatorname{aff}}(A,z)_{\boldsymbol{\omega}}.
\end{equation}
We now prove Proposition \ref{higher_center_quantum}. Let $ f \in Z(P_{d}(A))^{S_{d}} $. Then in order to show that \(f\) lies in the center of $ H_{d,\mathbf{Q}}^{\operatorname{aff}}(A,z) $, it suffices to show that \(f\) commutes with the identity morphisms $ 1_{\mathbf{i}} $, the elements of $ P_{d}(A) $, the red-black crossings, and the black-black crossings (since $ H_{d,\mathbf{Q}}^{\operatorname{aff}}(A,z) $ is generated by the elements). Since $ f \in Z(P_{d}(A)) $, we have that \(f\) commutes with the identity morphisms $ 1_{\mathbf{i}} $ and the elements of $ P_{d}(A) $. Next, we have by the relations \cref{hw11} and \cref{hw21} that \(f\) commutes with red-black crossings. Finally, \(f\) commutes with black-black crossings by \cite[Thm.~3.16]{Rosso-Savage}.
\\ \indent Conversely, let $ f \in Z(H_{d,\mathbf{Q}}^{\operatorname{aff}}(A,z)) $. Then we have $ f1_{\boldsymbol{\omega}} \in Z({_{\boldsymbol{\omega}}}H_{d,\mathbf{Q}}^{\operatorname{aff}}(A,z)_{\boldsymbol{\omega}}) $, and so since the isomorphism \cref{flour_quantum} restricts to an isomorphism between $ Z(H_{d}^{\operatorname{aff}}(A,z)) $ and $ Z({_{\boldsymbol{\omega}}}H_{d,\mathbf{Q}}^{\operatorname{aff}}(A,z)_{\boldsymbol{\omega}}) $, we obtain by \cite[Thm.~3.16]{Rosso-Savage} that $ f1_{\boldsymbol{\omega}} = g1_{\boldsymbol{\omega}} $ for some $ g \in Z(P_{d}(A))^{S_{d}} $. Let $ \mathbf{i} \in \Lambda_{d,\mathbf{Q}} $, and set $ D = T_{\mathbf{i},1,\boldsymbol{\omega}} $ (see Definition \ref{Generalized_sigma_quantum}). Then we have
\begin{equation*}
f1_{\mathbf{i}}D = Df1_{\boldsymbol{\omega}} = Dg1_{\boldsymbol{\omega}} = g1_{\mathbf{i}}D.
\end{equation*}
This implies that $ f1_{\mathbf{i}} = g1_{\mathbf{i}} $, since the map $ {_{\mathbf{i}}}H_{d,\mathbf{Q}}^{\operatorname{aff}}(A,z)_{\mathbf{i}} \rightarrow {_{\mathbf{i}}}H_{d,\mathbf{Q}}^{\operatorname{aff}}(A,z)_{\boldsymbol{\omega}} $, $ y \mapsto yD $ is injective by Theorem \ref{HL-basis_quantum}. Therefore, we have
\begin{equation*}
f = \sum_{\mathbf{i} \in \Lambda_{d,\mathbf{Q}}} f1_{\mathbf{i}} = \sum_{\mathbf{i} \in \Lambda_{d,\mathbf{Q}}} g1_{\mathbf{i}} = g \in Z(P_{d}(A))^{S_{d}}. 
\end{equation*}}

\begin{rmk}
Let $ \mathbf{Q}, \mathbf{Q}' \in F(\mathcal{E}_{P_{1}(A)}) $. Then, by Proposition \ref{higher_center_quantum}, we have an isomorphism of algebras $ Z(H_{d,\mathbf{Q}}^{\operatorname{aff}}(A,z)) \cong Z(H_{d,\mathbf{Q}'}^{\operatorname{aff}}(A,z)) $.
\end{rmk}

\subsection{Cyclotomic quotients} \label{Cyclotomic_quotients_quantum}

In this section, we define cyclotomic quotients of the higher-level affine Frobenius Hecke algebras. These algebras are natural generalizations of the cyclotomic quotients of higher-level affine Hecke algebras from \cite[Def.~6.1]{Maksimau-Stroppel} and \cite[Def.~4.7]{Webster}. Given an element $ \mathbf{i} \in F(\widehat{\mathcal{E}}_{P_{1}(A)}) $, we define $ \mathbf{i}_{1} \in \widehat{\mathcal{E}}_{P_{1}(A)} $ to be the first entry in the word $ \mathbf{i} $.

\begin{defn} 
Let $ d \in \mathbb{N}_{+} $ and $ \mathbf{Q} \in F(\mathcal{E}_{P_{1}(A)}) $, and assume that $ |\mathbf{Q}| \geq 1 $. Then we define the \emph{cyclotomic $ (d,\mathbf{Q}) $-Frobenius Hecke algebra} $ H_{d,\mathbf{Q}}^{\operatorname{cyc}}(A,z) $ to be the quotient of $ H_{d,\mathbf{Q}}^{\operatorname{aff}}(A,z) $ by the two-sided ideal generated by the elements of the set $ \{1_{\mathbf{i}} : \mathbf{i} \in \Lambda_{d,\mathbf{Q}}, \ \mathbf{i}_{1} = \go \} $.
\end{defn}

In other words, any diagram in $ H_{d,\mathbf{Q}}^{\operatorname{aff}}(A,z) $ that has a piece of black strand to the left of all the red strands is equal to \(0\) in the quotient $ H_{d,\mathbf{Q}}^{\operatorname{cyc}}(A,z) $. We can also define cyclotomic quotients of the affine Frobenius Hecke algebras $ H_{n}^{\operatorname{aff}}(A,z) $ from Definition \ref{quantum affine wreath product category}.

\begin{defn}[{\cite[§4]{Rosso-Savage}}]
Let $ n \in \mathbb{N}_{+} $ and $ Q \in \mathcal{E}_{P_{1}(A)} $. Then we define the \emph{cyclotomic Frobenius Hecke algebra} $ H_{n}^{Q}(A,z) $ to be the quotient of $ H_{n}^{\operatorname{aff}}(A,z) $ by the two-sided ideal generated by the element $ Q_{1} \in P_{n}(A) $.
\end{defn}

The following theorem shows that, when $ \mathbf{Q} \in F(\mathcal{E}_{P_{1}(A)}) $ is a word of length one, the cyclotomic $ (d,\mathbf{Q}) $-Frobenius Hecke algebra is isomorphic to a cyclotomic Frobenius Hecke algebra. When $ A = \Bbbk $, this theorem recovers \cite[Lem.~6.9]{Maksimau-Stroppel}.

\begin{theo} \label{quantum_cyclotomic_thm}
Let $ d \in \mathbb{N}_{+} $ and $ Q \in \mathcal{E}_{P_{1}(A)} $. Then we have an isomorphism of algebras 
\begin{equation*}
H_{d,Q}^{\operatorname{cyc}}(A,z) \cong H_{d}^{Q}(A,z).
\end{equation*}
\end{theo}

\begin{proof}
The proof of this theorem is analogous to the proof of Theorem \ref{cyclotomic_aff_thm}. Hence it will be omitted.
\end{proof}

\details{
We have a non-unital injective algebra homomorphism $ \varphi \colon H_{d}^{\operatorname{aff}}(A,z) \rightarrow H_{d,Q}^{\operatorname{aff}}(A,z) $ that adds a single red strand to the left of each diagram. Composing with the projection map $ p \colon H_{d,Q}^{\operatorname{aff}}(A,z) \rightarrow H_{d,Q}^{\operatorname{cyc}}(A,z) $ yields a unital algebra homomorphism $ \phi \colon H_{d}^{\operatorname{aff}}(A,z) \rightarrow H_{d,Q}^{\operatorname{cyc}}(A,z) $. So we have the following commutative diagram of maps:
\begin{equation*}
\begin{tikzpicture}[auto]
  \node (A) at (0,0) {$ H_{d}^{\operatorname{aff}}(A,z) $};
  \node (B) at (2.5,0) {$ H_{d,Q}^{\operatorname{aff}}(A,z) $};
  \node (C) at (5,0) {$ H_{d,Q}^{\operatorname{cyc}}(A,z) $};
  \draw[->] (A) -- node[above] {$ \varphi $} (B);
  \draw[->] (B) -- node[above] {$ p $} (C);
  \draw[->, bend left=40] (A) to node[above] {$ \phi $} (C);
\end{tikzpicture}
\end{equation*}
By Theorem \ref{HL-basis_quantum}, there is a spanning set of $ H_{d,Q}^{\operatorname{cyc}}(A,z) $ that lies in the image of $ \phi $. This gives that $ \phi $ is surjective. Set $ \mathcal{I} $ to be the two-sided ideal of $ H_{d}^{\operatorname{aff}}(A,z) $ generated by $ Q_{1} $. Then in order to prove the result, it only remains to show that the kernel of $ \phi $ is equal to $ \mathcal{I} $. That the latter is contained in the former follows by an application of \cref{hw31}. Namely, we have
\begin{equation*}
\phi\left(\begin{tikzpicture}[centerzero, thick]
        \draw[-] (0.3,-0.5) -- (0.3,0.5);
        \draw[-] (0.8,-0.5) -- (0.8,0.5);
        \node at (1.4,0) {$ \cdots $};
        \pin{0.3,0}{-0.5,0}{Q};
        \draw[-] (2,-0.5) -- (2,0.5);
\end{tikzpicture}\ \right)
\ =
\begin{tikzpicture}[centerzero, thick]
        \draw[wei] (-0.9,-0.5) -- (-0.9,0.5);
        \draw[-] (0.3,-0.5) -- (0.3,0.5);
        \draw[-] (0.8,-0.5) -- (0.8,0.5);
        \node at (1.4,0) {$ \cdots $};
        \node at (-0.9,-0.7) {\tiny $ Q $};
        \pin{0.3,0}{-0.5,0}{Q};
        \draw[-] (2,-0.5) -- (2,0.5);
\end{tikzpicture}
\ =
\begin{tikzpicture}[centerzero, thick]
        \draw[-] (0.2,-0.5) to[out=up,in=down] (-0.2,0) to[out=up,in=down] (0.2,0.5);
        \draw[-, wei] (-0.2,-0.5) to[out=up,in=down] (0.2,0) to[out=up,in=down] (-0.2,0.5);
        \draw[-] (0.7,-0.5) -- (0.7,0.5);
        \draw[-] (1.9,-0.5) -- (1.9,0.5);
        \node at (1.3,0) {$ \cdots $};        
        \node at (-0.2,-0.7) {\tiny $ Q $};
\end{tikzpicture}
\ = 0.
\end{equation*}
On the other hand, suppose $ y \in \operatorname{ker}(\phi) $. Then by the injectivity of $ \varphi $, it suffices to show that $ \varphi(y) \in \varphi(\mathcal{I}) $.
\\ \indent Note first that since $ y \in \operatorname{ker}(\phi) $, and $ \phi = p \circ \varphi $, we have that $ \varphi(y) $ is in the kernel of the projection map \(p\). So $ \varphi(y) $ is a $ \Bbbk $-linear combination of diagrams with a piece of black strand to the left of the red strand. If $ D $ is such a diagram, then we prove by induction that $ D \in \varphi(\mathcal{I}) $, from which it will follow by $ \Bbbk $-linearity that $ \varphi(y) \in \varphi(\mathcal{I}) $. The value $ c $ that we induct on is half the number of red-black crossings plus the number of black-black crossings to the left of the red strand in \(D\). For the base case, if $ c = 1 $, then there is a single black strand in $ D $ which crosses over the red strand and immediately crosses back:
\begin{equation*} 
\begin{tikzpicture}[centerzero, thick]
        \draw[-] (0.2,-0.5) to[out=up,in=down] (-0.2,0) to[out=up,in=down] (0.2,0.5);
        \draw[-, wei] (-0.2,-0.5) to[out=up,in=down] (0.2,0) to[out=up,in=down] (-0.2,0.5);
        \node at (-0.2,-0.7) {\tiny $ Q $};
\end{tikzpicture} \ .
\end{equation*}
Note that this double crossing may contain some tokens or dots on the black strand, but these can be moved out by the relations \cref{hw11} and \cref{hw21}. Then an application of \cref{hw31} shows that $ D \in \varphi(\mathcal{I}) $, giving that the base case holds. Next, if $ c > 1 $, then \(D\) either has a red-black double crossing or a red-black triangle, where the red strand is on the right:
\begin{equation*}
\begin{tikzpicture}[centerzero, thick]
        \draw[-] (0.2,-0.5) to[out=up,in=down] (-0.2,0) to[out=up,in=down] (0.2,0.5);
        \draw[-, wei] (-0.2,-0.5) to[out=up,in=down] (0.2,0) to[out=up,in=down] (-0.2,0.5);
        \node at (-0.2,-0.7) {\tiny $ Q $};
\end{tikzpicture} \ ,
\qquad 
\begin{tikzpicture}[centerzero, thick]
        \draw[-] (2,-0.5) -- (1.2,0.5);
        \draw[wipe] (1.2,-0.5) -- (2,0.5);
        \draw[-] (1.2,-0.5) -- (2,0.5);
        \node at (1.6,-0.7) {\tiny $ Q $};
        \draw[-, wei] (1.6,-0.5) to[out=up, in=down] (2,0) to[out=up,in=down] (1.6,0.5);
\end{tikzpicture} \ .
\end{equation*}
Again, this double crossing or triangle may contain tokens or dots on the black strands, but these can be moved out by the relations \cref{hw11} and \cref{hw21}. Then applying \cref{hw31} if \(D\) has a double crossing, or applying \cref{green1} if \(D\) has a triangle, shows that \(D\) can be written as a $ \Bbbk $-linear combination of diagrams with smaller \(c\). Thus by the induction hypothesis, we have $ D \in \varphi(\mathcal{I}) $.}


\bibliographystyle{alphaurl}
\bibliography{references}

\begin{thebibliography}{DKM25}

\bibitem[BE17]{Brundan-Ellis}
Jonathan Brundan and Alexander~P. Ellis.
\newblock Monoidal supercategories.
\newblock {\em Comm. Math. Phys.}, 351(3):1045--1089, 2017.
\newblock \href {http://arxiv.org/abs/1603.05928} {\path{arXiv:1603.05928}},
  \href {https://doi.org/10.1007/s00220-017-2850-9}
  {\path{doi:10.1007/s00220-017-2850-9}}.

\bibitem[BSW21]{Brundan-Savage-Webster2}
Jonathan Brundan, Alistair Savage, and Ben Webster.
\newblock Foundations of {F}robenius {H}eisenberg categories.
\newblock {\em J. Algebra}, 578:115--185, 2021.
\newblock \href {http://arxiv.org/abs/2007.01642} {\path{arXiv:2007.01642}},
  \href {https://doi.org/10.1016/j.jalgebra.2021.02.025}
  {\path{doi:10.1016/j.jalgebra.2021.02.025}}.

\bibitem[BSW22]{Brundan-Savage-Webster}
Jonathan Brundan, Alistair Savage, and Ben Webster.
\newblock Quantum {F}robenius {H}eisenberg categorification.
\newblock {\em J. Pure Appl. Algebra}, 226(1):Paper No. 106792, 50, 2022.
\newblock \href {http://arxiv.org/abs/2009.06690} {\path{arXiv:2009.06690}},
  \href {https://doi.org/10.1016/j.jpaa.2021.106792}
  {\path{doi:10.1016/j.jpaa.2021.106792}}.

\bibitem[Cou16]{Couture}
Chad Couture.
\newblock Skew-zigzag algebras.
\newblock {\em SIGMA Symmetry Integrability Geom. Methods Appl.}, 12:Paper No.
  062, 19, 2016.
\newblock \href {http://arxiv.org/abs/1509.08405} {\path{arXiv:1509.08405}},
  \href {https://doi.org/10.3842/SIGMA.2016.062}
  {\path{doi:10.3842/SIGMA.2016.062}}.

\bibitem[CPd14]{Chlouveraki-d'Andecy}
Maria Chlouveraki and Lo\"ic Poulain~d'Andecy.
\newblock Representation theory of the {Y}okonuma-{H}ecke algebra.
\newblock {\em Adv. Math.}, 259:134--172, 2014.
\newblock \href {http://arxiv.org/abs/1302.6225} {\path{arXiv:1302.6225}},
  \href {https://doi.org/10.1016/j.aim.2014.03.017}
  {\path{doi:10.1016/j.aim.2014.03.017}}.

\bibitem[DKM25]{Davidson-Kujawa-Muth}
Nicholas Davidson, J~Kujawa, and Robert Muth.
\newblock Superalgebra deformations of web categories: Affine and cyclotomic
  webs, 2025.
\newblock \href {http://arxiv.org/abs/2511.21671} {\path{arXiv:2511.21671}}.

\bibitem[HK01]{Huerfano-Khovanov}
Ruth~Stella Huerfano and Mikhail Khovanov.
\newblock A category for the adjoint representation.
\newblock {\em J. Algebra}, 246(2):514--542, 2001.
\newblock \href {http://arxiv.org/abs/math/0002060}
  {\path{arXiv:math/0002060}}, \href {https://doi.org/10.1006/jabr.2001.8962}
  {\path{doi:10.1006/jabr.2001.8962}}.

\bibitem[Kle05]{Kleshchev}
Alexander Kleshchev.
\newblock {\em Linear and projective representations of symmetric groups},
  volume 163 of {\em Cambridge Tracts in Mathematics}.
\newblock Cambridge University Press, Cambridge, 2005.
\newblock \href {https://doi.org/10.1017/CBO9780511542800}
  {\path{doi:10.1017/CBO9780511542800}}.

\bibitem[KM19]{Kleshchev-Muth}
Alexander Kleshchev and Robert Muth.
\newblock Affine zigzag algebras and imaginary strata for {KLR} algebras.
\newblock {\em Trans. Amer. Math. Soc.}, 371(7):4535--4583, 2019.
\newblock \href {http://arxiv.org/abs/1511.05905} {\path{arXiv:1511.05905}},
  \href {https://doi.org/10.1090/tran/7464} {\path{doi:10.1090/tran/7464}}.

\bibitem[Liu18]{Liu}
Bingyan Liu.
\newblock Presentations of linear monoidal categories and their endomorphism
  algebras, 2018.
\newblock \href {http://arxiv.org/abs/1810.10988} {\path{arXiv:1810.10988}}.

\bibitem[LNX24]{Lai-Nakano-Ziqing}
Chun-Ju Lai, Daniel~K. Nakano, and Ziqing Xiang.
\newblock Quantum wreath products and {S}chur-{W}eyl duality {I}.
\newblock {\em Forum Math. Sigma}, 12:Paper No. e108, 38, 2024.
\newblock \href {http://arxiv.org/abs/2304.14181} {\path{arXiv:2304.14181}},
  \href {https://doi.org/10.1017/fms.2024.103}
  {\path{doi:10.1017/fms.2024.103}}.

\bibitem[Men20]{Mendonca_thesis}
Eduardo~Monteiro Mendon{\c{c}}a.
\newblock Affine wreath product algebras with trace maps of generic parity.
\newblock Master's thesis, University of S{\~a}o Paulo, Brazil, 2020.
\newblock URL:
  \url{https://www.teses.usp.br/teses/disponiveis/45/45131/tde-07082020-121337/pt-br.php}.

\bibitem[Men22]{Mendonca}
Eduardo~Monteiro Mendon{\c{c}}a.
\newblock Affine wreath product algebras with trace maps of generic parity.
\newblock {\em Comm. Algebra}, 50(12):5217--5245, 2022.
\newblock \href {https://doi.org/10.1080/00927872.2022.2083629}
  {\path{doi:10.1080/00927872.2022.2083629}}.

\bibitem[MS21]{Maksimau-Stroppel}
Ruslan Maksimau and Catharina Stroppel.
\newblock Higher level affine {S}chur and {H}ecke algebras.
\newblock {\em J. Pure Appl. Algebra}, 225(8):Paper No. 106442, 44, 2021.
\newblock \href {http://arxiv.org/abs/1805.02425} {\path{arXiv:1805.02425}},
  \href {https://doi.org/10.1016/j.jpaa.2020.106442}
  {\path{doi:10.1016/j.jpaa.2020.106442}}.

\bibitem[Naz97]{Nazarov}
Maxim Nazarov.
\newblock Young's symmetrizers for projective representations of the symmetric
  group.
\newblock {\em Adv. Math.}, 127(2):190--257, 1997.
\newblock \href {https://doi.org/10.1006/aima.1997.1621}
  {\path{doi:10.1006/aima.1997.1621}}.

\bibitem[PS16]{Pike-Savage}
Jeffrey Pike and Alistair Savage.
\newblock Twisted {F}robenius extensions of graded superrings.
\newblock {\em Algebr. Represent. Theory}, 19(1):113--133, 2016.
\newblock \href {http://arxiv.org/abs/1502.00590} {\path{arXiv:1502.00590}},
  \href {https://doi.org/10.1007/s10468-015-9565-4}
  {\path{doi:10.1007/s10468-015-9565-4}}.

\bibitem[RS20]{Rosso-Savage}
Daniele Rosso and Alistair Savage.
\newblock Quantum affine wreath algebras.
\newblock {\em Doc. Math.}, 25:425--456, 2020.
\newblock \href {http://arxiv.org/abs/1902.00143} {\path{arXiv:1902.00143}}.

\bibitem[Sav19]{Savage2}
Alistair Savage.
\newblock Frobenius {H}eisenberg categorification.
\newblock {\em Algebr. Comb.}, 2(5):937--967, 2019.
\newblock \href {http://arxiv.org/abs/1802.01626} {\path{arXiv:1802.01626}},
  \href {https://doi.org/10.5802/alco.73} {\path{doi:10.5802/alco.73}}.

\bibitem[Sav20]{Savage}
Alistair Savage.
\newblock Affine wreath product algebras.
\newblock {\em Int. Math. Res. Not. IMRN}, (10):2977--3041, 2020.
\newblock \href {http://arxiv.org/abs/1709.02998} {\path{arXiv:1709.02998}},
  \href {https://doi.org/10.1093/imrn/rny092} {\path{doi:10.1093/imrn/rny092}}.

\bibitem[SW24]{Song-Wang2}
Linliang Song and Weiqiang Wang.
\newblock Affine and cyclotomic {S}chur categories, 2024.
\newblock \href {http://arxiv.org/abs/2407.10119} {\path{arXiv:2407.10119}}.

\bibitem[Wan09]{Wang}
Weiqiang Wang.
\newblock Double affine {H}ecke algebras for the spin symmetric group.
\newblock {\em Math. Res. Lett.}, 16(6):1071--1085, 2009.
\newblock \href {http://arxiv.org/abs/math/0608074}
  {\path{arXiv:math/0608074}}, \href
  {https://doi.org/10.4310/MRL.2009.v16.n6.a14}
  {\path{doi:10.4310/MRL.2009.v16.n6.a14}}.

\bibitem[Web17]{Webster}
Ben Webster.
\newblock Knot invariants and higher representation theory.
\newblock {\em Mem. Amer. Math. Soc.}, 250(1191):v+141, 2017.
\newblock \href {http://arxiv.org/abs/1309.3796} {\path{arXiv:1309.3796}},
  \href {https://doi.org/10.1090/memo/1191} {\path{doi:10.1090/memo/1191}}.

\bibitem[Web20]{Webster2}
Ben Webster.
\newblock On graded presentations of {H}ecke algebras and their
  generalizations.
\newblock {\em Algebr. Comb.}, 3(1):1--38, 2020.
\newblock \href {http://arxiv.org/abs/1305.0599} {\path{arXiv:1305.0599}},
  \href {https://doi.org/10.5802/alco.84} {\path{doi:10.5802/alco.84}}.

\bibitem[WW08]{Wan-Wang}
Jinkui Wan and Weiqiang Wang.
\newblock Modular representations and branching rules for wreath {H}ecke
  algebras.
\newblock {\em Int. Math. Res. Not. IMRN}, pages Art. ID rnn128, 31, 2008.
\newblock \href {http://arxiv.org/abs/0806.0196} {\path{arXiv:0806.0196}},
  \href {https://doi.org/10.1093/imrn/rnn128} {\path{doi:10.1093/imrn/rnn128}}.

\end{thebibliography}

\end{document}